\documentclass[11pt,oneside,english]{amsart}
\usepackage[utf8]{inputenc}
\usepackage[letterpaper]{geometry}
\usepackage{babel}
\usepackage{mathrsfs}
\usepackage{amsbsy}
\usepackage{amstext}
\usepackage{amsthm}
\usepackage{amssymb}
\usepackage{graphicx}
\usepackage{setspace}
\usepackage[unicode=true,pdfusetitle,
 bookmarks=true,bookmarksnumbered=false,bookmarksopen=false,
 breaklinks=false,pdfborder={0 0 1},backref=false,colorlinks=false]
 {hyperref}

\makeatletter
\numberwithin{equation}{section}
\numberwithin{figure}{section}
\theoremstyle{plain}
\newtheorem{thm}{\protect\theoremname}[section]
\theoremstyle{definition}
\newtheorem{defn}[thm]{\protect\definitionname}
\theoremstyle{remark}
\newtheorem{rem}[thm]{\protect\remarkname}
\theoremstyle{definition}
\newtheorem{example}[thm]{\protect\examplename}
\theoremstyle{plain}
\newtheorem{lem}[thm]{\protect\lemmaname}
\theoremstyle{plain}
\newtheorem{prop}[thm]{\protect\propositionname}
\theoremstyle{plain}
\newtheorem{cor}[thm]{\protect\corollaryname}

\include{mypackages}
\include{mymacros}
\usepackage{slashed}
\usepackage[all]{xy}

\newcommand{\kah}{\rm{K\ddot{a}hler}}
\newcommand{\dbar}{\bar{\partial}}
\newcommand{\ddbar}{\partial\bar{\partial}}
\newcommand{\MA}{\rm{Monge-Amp\`{e}re}}

\newcommand{\pdv}{\partial}

\newcommand{\iover}{\frac{\sqrt{-1}}{2}}

\def\R{\mathbb R}
\def\C{\mathbb C}
\def\Z{\mathbb Z}

\def\N{\mathbb N}

\def\La{\Lambda}

\def\<{\langle}
\def\>{\rangle}

\makeatother

\providecommand{\corollaryname}{Corollary}
\providecommand{\definitionname}{Definition}
\providecommand{\examplename}{Example}
\providecommand{\lemmaname}{Lemma}
\providecommand{\propositionname}{Proposition}
\providecommand{\remarkname}{Remark}
\providecommand{\theoremname}{Theorem}

\begin{document}
\title{On a fully nonlinear elliptic equation with differential forms}
\author{Hao Fang}
\email{hao-fang@uiowa.edu}
\author{Biao Ma}
\email{biaoma@bicmr.pku.edu.cn}
\begin{abstract}
We introduce a fully nonlinear PDE with a differential form, which
unifies several important equations in $\kah$ geometry including
Monge-Ampère equations, $J$-equations, inverse $\sigma_{k}$ equations,
and the deformed Hermitian Yang-Mills (dHYM) equation. We pose some
natural positivity conditions on $\Lambda$, and prove analytical
and algebraic criterions for the solvability of the equation. Our
results generalize previous works of G.Chen, J.Song, Datar-Pingali
and others. As an application, we prove a conjecture of Collins-Jacob-Yau
for the dHYM equation with small global phase. 
\end{abstract}

\maketitle
\tableofcontents{}

\section{Introduction}

In this paper, we study a general form of fully non-linear partial
differential equations of Monge-Ampère type on $\kah$ manifolds.

Let $(M,\omega_{0})$ be a $\kah$ manifold of dimension $n$, and
let $[\omega_{0}]$ be the corresponding $\kah$ class. We fix a closed
differential form of the following format:
\[
\Lambda=\sum_{k=1}^{n}\Lambda^{[k]},
\]
where for each $1\leq k\leq n$, $\Lambda^{[k]}$ is a real $(k,k)$-form.
Hereby we use $\alpha^{[k]}$ to denote the $(k,k)$ component of
a differential form $\alpha$. Let $\omega=\omega_{0}+i\partial\bar{\partial}u\in[\omega_{0}]$
be a $\kah$ metric. We denote 
\[
\Omega=\Omega(\omega)=\exp\omega=\sum_{k=0}^{n}\frac{\omega^{k}}{k!}.
\]
In this paper, we study the following partial differential equation
of function $u$: 
\begin{equation}
\kappa\Omega^{[n]}=(\Lambda\wedge\Omega)^{[n]},\label{eq:eq with Berezin integ}
\end{equation}
where $\kappa$ is a positive constant such that $\kappa\int_{M}\exp\omega_{0}=\int_{M}\Lambda\wedge\exp\omega_{0}.$

When $\Lambda$ has only the positive $(n,n)$ component, it is thus
a volume form of $M$. (\ref{eq:eq with Berezin integ}) is the well-known
complex Monge-Ampère equation, first completely solved by Yau \cite{yau1978ricci}.

For future convenience, we write $\Lambda^{[n]}$ as $\frac{f\rho^{n}}{n!}$,
where $f$ is a smooth function on $M$and $\rho$ is a fixed background
$\kah$ metric. Then (\ref{eq:eq with Berezin integ}) can be expanded
as:

\begin{equation}
\kappa\frac{\omega^{n}}{n!}=\sum_{k=1}^{n-1}\Lambda^{[k]}\wedge\frac{\omega^{n-k}}{(n-k)!}+f\frac{\rho^{n}}{n!}.\label{eq:equation with f}
\end{equation}

In addition to the original Monge-Ampère equation, other special cases
of equation (\ref{eq:equation with f}) include:
\begin{enumerate}
\item When $\Lambda=\rho$ is a $\kah$ form, (\ref{eq:eq with Berezin integ})
is the $J$-equation, which was first introduced by Donaldson \cite{donaldson1999moment}
and has been extensively studied by many authors. See \cite{donaldson1999moment,chen2000lower,weinkove2004convergence,song2008convergence,collins2017convergence,chen2021j,Song2020NakaiMoishezonCF}
and references therein. 
\item When $\Lambda$ is $\rho^{k}$ or $\sum_{i=k}^{n-1}c_{k}\rho^{k}$,
where $\rho$ is a $\kah$ form and $c_{k}\geq0$, (\ref{eq:eq with Berezin integ})
is the inverse $\sigma_{k}$-equation, which was first raised in Fang-Lai-Ma
\cite{Fang-Lai-Ma}. An incomplete list of works on these equations
includes \cite{Fang-Lai-Ma,Guan2014Second-order,guan2015class,collins2017convergence,szekelyhidi2018fully,datar2021numerical}. 
\item Given a specific global phase $\theta>0$ and $\Lambda=\sin\theta\cos\rho-\cos\theta\sin\rho$,
where $\rho$ is a Kähler form, (\ref{eq:eq with Berezin integ})
is the deformed Hermitian Yang-Mills (dHYM) equation. dHYM equation
was introduced in Marino-Minasian-Moore-Strominger \cite{marino2000nonlinear}
and Leung-Yau-Zaslow \cite{leung2000special}\textbf{ }and has been
extensively studied; See \cite{jacob2017special,collins20201,pingali2022deformed,chen2021j,chu2021nakai,collins2021moment,fu2021deformed,huang2022deformed,lin2022deformed}
and references therein. See also Section \ref{sec:Application-to-Hypercritical}
for details. 
\end{enumerate}
PDEs in the general form of (\ref{eq:eq with Berezin integ}) are
inspired by above-mentioned works, which have roots in rich geometric
and mathematical physics background, and are important in the field
of fully non-linear PDEs with significant geometric implications.
Known results indicate that the existence of unique smooth solutions
can be interpreted as analytical and algebraic properties of the underlying
manifold. Therefore, it is our intention to systematically explore
more general forms of such PDEs. In particular, we would like to study
necessary and sufficient conditions that ensure the existence of solutions.

We realize from past works that, aside from the unknown Kähler metric
to be solved, given geometric data often involve another Kähler form
or its powers. The exact format of (\ref{eq:eq with Berezin integ})
suggests some point-wise positivity requirement for the differential
form $\Lambda.$ In addition, the important fact of ellipticity of
Monge-Ampère type equations poses similar requirements. One of our
main goals is to find natural positivity conditions under which necessary
a priori estimates can be established and deep connections between
PDE and geometry can be reconstructed. Inspired by Gao Chen's work\textbf{
}\cite{chen2021j}, we also consider (\ref{eq:eq with Berezin integ})
in a bit more general form, in which we allow $\Lambda^{[n]}$ to
be slightly negative at some points. 

For convenience of further discussion, we decompose $\Lambda$ as
follows 
\[
\Lambda=\mathring{\Lambda}+\Lambda^{[n]},
\]
and discuss constraints on $\mathring{\Lambda}$ and $\Lambda^{[n]}$
separately. 
\begin{defn}
\label{def:UP}We call a real even closed differential form $\mathring{\Lambda}$\emph{
on M $k$-uniformly positive }($k$-UP) if there exists an integer
$1\leq k\le n$, a reference $\kah$ metric $\rho$ of $M$, and a
uniform constant $m>0$ such that $\Lambda^{[l]}=0$ for $l<k$ ;
and 
\begin{equation}
\mathring{\Lambda}-m\frac{\rho^{k}}{k!}\geq0,\label{eq:uniformly positive}
\end{equation}
which (\ref{eq:uniformly positive}) means that each $(l,l)$-component
of the left hand side is a positive form for $1\leq l\leq n$. See
Definition \ref{def:An-element-positive forms}. 
\end{defn}

Notice that Definition \ref{def:UP} is independent of the choice
of reference metric $\rho$. An $n$-UP form $\Lambda$ is a positive
volume form, which is required for solving the classic Monge-Ampère
equation. $J$-equation is a special case where $\Lambda$ is $1$-UP.
Furthermore, positive linear combination of powers of a given Kähler
form is uniformly positive. Therefore, inverse $\sigma_{k}$ equations
also fall into this category. However, as uniform positivity is an
open condition, it is far more general than these special cases. In
future sections, we will also give more general and technical conditions
under which various conclusions hold.

Since the classic Monge-Ampère equation is well known in the field,
we will discuss only cases when $\mathring{\Lambda}$ is $k$-UP for
some $1\leq k\leq n-1$. When we consider $\Lambda^{[n]},$ in spirit
of the work of G. Chen \cite{chen2021j}, we pose a slightly weaker
condition. 
\begin{defn}
Given $1\leq k_{0}\leq n-1$, $m>0$, and a reference $\kah$ metric
$\rho$, we call $(n,n)$-form $\alpha$ \emph{an almost positive
volume form} with respect to $(k_{0},m,\rho)$, if $\int_{M}\alpha\geq0$
and there exists $\epsilon=\epsilon(n,m,\kappa,k_{0},\omega_{0},\rho)>0$
such that 
\begin{equation}
\frac{\alpha}{\rho^{n}/n!}>-\epsilon.\label{eq:H2-1}
\end{equation}
\end{defn}

\begin{rem}
We will choose a specific constant $\epsilon<\min\left\{ \frac{m}{4n+2}\gamma_{\min}(\frac{2\kappa}{m},1,n,k_{0}),\frac{\kappa\int_{M}\omega_{0}^{n}}{2\int_{M}\rho^{n}}\right\} ,$where
$\gamma_{\min}$ is a positive number, defined\label{rem:We-will-chooseremark}
later in (\ref{eq:gamm_min}) in the rest of the paper.
\end{rem}

We state the following hypothesis on $\Lambda$, which will be a key
assumption for our paper. 
\begin{defn}
$\Lambda=\mathring{\Lambda}+\Lambda^{[n]}$ is called to satisfy the
condition \textbf{H1}, if and only if $\mathring{\Lambda}$ is $k_{0}$-uniformly
positive for some $1\leq k_{0}\leq n-1$ with a reference metric $\rho$
and a uniform constant $m>0$; and $\Lambda^{[n]}$ is almost positive
with respect to a specific $\epsilon$(as in Remark \ref{rem:We-will-chooseremark})
and $(k_{0},m,\rho)$. 
\end{defn}

In order to state our main analytical result, we define the following
concept based on previous works (\cite{song2008convergence,Fang-Lai-Ma,Guan2014Second-order,szekelyhidi2018fully}
). 
\begin{defn}
A $\kah$ form $\omega\in[\omega_{0}]$ is called to satisfy the \emph{cone
condition} for equation (\ref{eq:equation with f}) if and only if
pointwisely it satisfies the following 
\begin{equation}
\kappa(\exp\omega)^{[n-1]}-\left(\Lambda\wedge\exp\omega\right)^{[n-1]}>0;\label{eq:cone condition}
\end{equation}
Or equivalently, 
\begin{equation}
\frac{\kappa\omega^{n-1}}{(n-1)!}-\sum_{k=1}^{n-1}\frac{\Lambda^{[k]}\wedge\omega^{n-k-1}}{(n-k-1)!}>0,\label{eq:cone condition expanded}
\end{equation}
which means that the left hand side of (\ref{eq:cone condition})
or \ref{eq:cone condition expanded}) is a strictly positive $(n-1,n-1)$
form. We also call $\omega$ a \emph{subsolution} of the equation
(\ref{eq:eq with Berezin integ}) or (\ref{eq:equation with f}) if
(\ref{eq:cone condition}) holds everywhere. 
\end{defn}

In order to state our main geometrical result, inspired by the works
of \cite{lejmi2015j,szekelyhidi2018fully},\textbf{ }we propose the
following numerical positivity conditions: 
\begin{defn}
\label{def:Lambda-kappa}Let $[\Lambda]$ be the cohomology class
of $\La$. Let $\kappa>0$ be a constant. Let $[\alpha]$ be a $\kah$
class. We call $[\alpha]$ a \emph{$([\Lambda],\kappa)$-positive}
class if 
\begin{equation}
[\exp\alpha]\cdot[\kappa-\Lambda]\cdot[M]\geq0,\label{eq:positive class n}
\end{equation}
and for any subvariety $Y$ of $M$ with $\dim Y<n$ it holds that
\begin{equation}
[\exp\alpha]\cdot[\kappa-\Lambda]\cdot[Y]>0.\label{eq:positive class}
\end{equation}
\end{defn}

With key definitions ready, we state our main theorems, which consist
2 parts. The first is an analytic criterion for the existence of solution
to (\ref{eq:equation with f}). 
\begin{thm}
\label{thm:Let--beanlytic}Let $M$ be a connected compact Kähler
manifold of dimension $n$ with a fixed Kähler class $[\omega_{0}]$.
Suppose that $\Lambda$ is a closed real differential form satisfying
\textbf{H1}. Then, there exists a smooth solution of (\ref{eq:equation with f})
if and only if there exists a smooth subsolution of (\ref{eq:equation with f}). 
\end{thm}

Special cases of Theorem \ref{thm:Let--beanlytic} have been established
by Song-Weinkove \cite{song2008convergence}, Fang-Lai-Ma \cite{Fang-Lai-Ma},
Guan \cite{Guan2014Second-order}, Guan-Sun \cite{guan2015class},
Szèkelyhidi \cite{szekelyhidi2018fully}, Chen \cite{chen2021j},
Datar-Pingali \cite{datar2021numerical}. 

\begin{rem}
We also establish the uniqueness result regarding (\ref{eq:equation with f}).
See Appendix \ref{sec:Functional-and-uniqueness}. In fact, solutions
will be shown to be the unique minimizer of a global functional $\mathcal{F}$
defined in (\ref{eq:global functional}). We will further explore
the variational structure of equation (\ref{eq:equation with f})
in the Appendix \ref{sec:Functional-and-uniqueness} and some future
works. 
\end{rem}

\begin{rem}
Following works of \cite{Guan2014Second-order,guan2015class,szekelyhidi2018fully},
it is likely that Theorem \ref{thm:Let--beanlytic} may be extended
to Hermitian manifolds, when similar subsolution definition may be
established. 
\end{rem}

Our second main theorem relates the solvability of equation (\ref{eq:equation with f})
to Definition \ref{def:Lambda-kappa}. 
\begin{thm}
\label{thm:numerical criterion}Let $M$ be a connected compact Kähler
manifold of dimension $n$ with a fixed Kähler class $[\omega_{0}]$.
Suppose that $\Lambda$ is a closed real differential form satisfying
\textbf{H1}. Then $[\omega_{0}]$ is \textbf{$([\Lambda],\kappa)$}-positive
if and only if there exists a smooth Kähler metric $\omega$ solves
(\ref{eq:equation with f}). 
\end{thm}

Theorem \ref{thm:numerical criterion} generalizes several existing
works. In the context of $J$-equation, G. Chen \cite{chen2021j}
proved the equivalence of $J$-uniformly positivity and the solvability
of the $J$-equation; Later, J. Song \cite{Song2020NakaiMoishezonCF}
reduced the $J$-uniformly positivity to only $J$-positivity. In
the context of inverse $\sigma_{k}$ type equations, using different
methods, similar numerical results have been proved by Collins-Szèkelyhidi
\cite{collins2017convergence} for toric manifolds and Datar-Pingali
\cite{datar2021numerical} for projective manifolds. 

While many of past works are raised from strong geometric background,
our work aims at a general form of non-linear PDEs of Monge-Ampère
type on $\kah$ manifolds. In particular, our equation is one of the
few fully non-linear PDEs that involve general differential forms
of different degrees. It is hopeful that future works will reveal
more geometry of positive differential forms on complex manifolds,
which has not been studied much from the non-linear PDE point of view. 

As an application of our main results, we study supercritical deformed
Yang-Mills equations to get the following 
\begin{thm}
\label{thm:application to dhym}Let $M$ be a connected compact Kähler
manifold of dimension $n$. Let $[\omega_{0}]$ be a real $(1,1)$-cohomology
class and let $\rho$ be a Kähler form. Let 
\[
\cot\theta=\frac{\int_{M}\text{Re}(\omega+\sqrt{-1}\rho)^{n}}{\int_{M}\text{Im}(\omega+\sqrt{-1}\rho)^{n}}.
\]
If $\theta\in(0,\frac{\pi}{n-1}]$, then there exists a smooth solution
to the supercritical dHYM equation 
\begin{equation}
\text{Re}(\omega+\sqrt{-1}\rho)^{n}=\cot\theta\text{Im}(\omega+\sqrt{-1}\rho)^{n},\label{eq:dHYM eq-1}
\end{equation}
if and only if $[\omega_{0}-\cot\theta\rho]$ is Kähler and for any
analytic subvariety $V$ of dimension $d\leq n-1$, we have 
\begin{equation}
\int_{V}\left(\text{Re}(\omega_{0}+\sqrt{-1}\rho)^{d}-\cot\theta\text{Im}(\omega_{0}+\sqrt{-1}\rho)^{d}\right)>0.\label{eq:stability-1}
\end{equation}
\end{thm}

For supercritical dHYM equations, the equivalence of the existence
of a smooth solution to (\ref{eq:dHYM eq-1}) and the numerical condition
(\ref{eq:stability-1}) was conjectured by Collins-Jacob-Yau \cite{collins20201}
and confirmed by the work of G. Chen \cite{chen2021j}, Chu-Lee-Takahashi
\cite{chu2021nakai}, and A. Ballal \cite{Ballal10.1215/00192082-10417484}
under various stronger assumptions. See more details in Section \ref{sec:Application-to-Hypercritical}.
The original conjecture in \cite{collins20201} in its full generality
does not hold due to counterexamples constructed by Zhang \cite{zhang2023note}.
However, Theorem \ref{thm:application to dhym} confirms the original
conjecture when $\theta\in(0,\frac{\pi}{n-1}]$. Notice that for Zhang's
examples\cite{zhang2023note} , $\theta>\frac{\pi}{n-1}$. It is therefore
interesting to find the sharp range of $\theta$ in which the conjecture
of Collins-Jacob-Yau holds.

We make some comments on our proofs.

We have utilized numerous ideas and techniques of existing results,
many have been cited above. On the analytic side: We employ Yau’s
continuity method which has been successfully used by many to attack
Monge-Ampère type and other types of equations. We follow works of
many \cite{song2008convergence,Fang-Lai-Ma,Guan2014Second-order,collins2017convergence,szekelyhidi2018fully,chen2021j,datar2021numerical}
to derive \textit{a priori} estimates. We use Szèkelyhidi's argument
to derive $C^{0}$ estimate \cite{szekelyhidi2018fully}. We use arguments
in G. Chen \cite{chen2021j} and Datar-Pingali \cite{datar2021numerical}
to treat almost positive volume forms. The definition of subsolution
is inspired by similar concepts in Song-Weinkove \cite{song2008convergence},
Fang-Ma-Lai \cite{Fang-Lai-Ma}, Guan \cite{Guan2014Second-order},
and Szèkelyhidi \cite{szekelyhidi2018fully}. On the algebraic side:
The numerical condition in Definition \ref{def:Lambda-kappa} is inspired
by \cite{lejmi2015j,szekelyhidi2018fully}. We follow G. Chen's induction
method and overall approach in \cite{chen2021j} to tackle the numerical
criterion. We make use of Demailly-Păun's mass concentration technique
\cite{demailly2004numerical}. We follow J. Song's treatment of singular
subvarieties \cite{Song2020NakaiMoishezonCF}. 

We have made a conscious effort to examine existing methods and explore
their maximal applicability, which has lead us to the current formulation
of general equations and proper positivity conditions. As mentioned
earlier, our equation includes many existing important special cases.
We apply a new set of local notations and carry out detailed multi-linear
algebraic computation involving inverse $\sigma_{k}$ type quantities.

One key component of our proof is the definition of the uniform positivity
and the subsequent \textbf{H1} condition, which ensure that the important
concept of subsolutions can be defined and key a priori estimates
can be carried out. We would like to remark that the uniform positivity
condition can not be used directly to prove Theorem \ref{thm:numerical criterion}.
From the algebraic point of view, proper positivity conditions are
required to be preserved under two important geometric procedures
which are used to prove corresponding numerical results: the first
is a product manifold construction; the second is passing from a $\kah$
manifold to its sub-varieties. The uniform positivity condition can
be preserved under the second procedure, but unfortunately it is not
preserved under the product manifold procedure. In order to overcome
this difficulty, we introduce a more technical positivity condition,
called \textbf{H2}, which is a similar but more general version of
the \textbf{H1} condition and will be stated in later sections. \textbf{H2}
works under the product manifold procedure, but does not descend to
sub-varieties. We carefully establish all necessary algebraic and
analytic results that are needed in our proofs under these technical
assumptions by fully exploring the algebraic structure of differential
forms. The final write-up is technical but self-contained for the
convenience of readers. We believe that our main theorems, Theorems
\ref{thm:Let--beanlytic} and \ref{thm:numerical criterion}, are
cleanest in format, still incorporate all known special cases and
will be of most use in future applications. 

Another key component of our proof is a new geometric PDE (\ref{eq:product eq})
on the product manifold, which is simpler and more flexible than the
one used by Chen \cite{chen2021j}, even in the J-equation case. Under
our general setting, the new equation is compatible with proper cone
conditions, and still carries essential algebraic information. See
Example \ref{exa:non-trivial example} for further elaboration in
a simple yet illuminating case.

There are several future research directions of interest. 

First of all, our results may be extended further. The general form
of dHYM equations do not satisfy any of our positivity assumptions,
while its special cases have been studied successfully. In \cite{LIN2023110038},
Lin also studied the convexity of inverse $\sigma_{k}$ type equations
with some negative coefficients. This may indicate that some weaker
positivity conditions may lead to existence results. 

Similar to the J-equation and dHYM equations ( \cite{donaldson1999moment,collins2019stability,collins2021moment}),
there is also a proper moment map interpretation of our equation,
which involves an infinite dimensional symplectic space. Such point
of view is helpful to consider further generalizations and applications.
We would like to address this topic in a future work.

From the PDE point of view, our current approach, as well as several
important previous works, heavily depends on the continuity method.
It is interesting to explore corresponding geometric flows, which
has been successfully used in the setting of J-equation and dHYM equations
but missing in the more general setting. See\cite{chen2000lower,weinkove2004convergence,song2008convergence,Fang-Lai-Ma}
and \cite{collins20201,fu2021deformed,chen2021j,chu2021nakai,huang2022deformed,lin2022deformed,pingali2022deformed,collins2021moment,takahashi2020tan}
for an incomplete list. Overall, the parabolic approach will expose
finer geometric information.

From the geometry point of view, as mentioned earlier, one may explore
the geometry of properly defined positive differential forms using
non-linear PDEs, which will be a new research direction. Also mentioned
earlier, some of our results is likely to be extended to the Hermitian
setting. This may further open applications in the field of special
geometry and mathematical physics. From another prospective, indicated
by works of Datar-Pingali \cite{datar2021numerical}, Chu-Lee-Takahashi
\cite{chu2021nakai} and a recent lecture note of G. Chen \cite{chen2022lecture},
one may explore the more restrictive setting where $M$ is projective.
Studies of generalized Hodge conjecture explore fine distinction between
$\kah$ and projective manifolds. Our general form of PDE may be of
help.

The rest of the paper is organized as follows: In Part 1, we establish
main analytical results of this paper; In Part 2, we prove our main
algebraic result. We give more details: In Section 2, we state our
main technical analytical results; In Section 3, we describe our local
setup and establish proper notations; In Section 4, when $\Lambda^{[n]}$
is non-negative, we compute variations of our non-linear differential
operators and derive its ellipticity and concavity; In Section 5,
we discuss the cone condition and related properties; In Section 6,
we extend results of previous sections to include the almost positive
$\Lambda^{[n]}$ ; In Section 7, we establish proper a priori estimate
and use the continuity method to prove our main analytical results;
In Section 8, we set up notations and make technical preparations
for Part 2; In Section 9, we initiate the induction argument to prove
Theorem \ref{thm:numerical criterion}; In Section 10, we prove the
key mass concentration theorem; In Section 11, we complete the proof
of Theorem \ref{thm:numerical criterion}; In Section 12, we study
dHYM equations and prove Theorem \ref{thm:application to dhym}. 

\subsection*{Acknowledgements}

We thank Jian Song and Jianchun Chu for helpful discussions. We thank
Gang Tian and Xiaohua Zhu for comments and suggestions. We thank Vamsi
Pritham Pingali and Ved Datar for their interest and comments that
are helpful.

\part{Analytical Criterion}

\section{Assumptions and main results}

In this part we prove Theorem \ref{thm:Let--beanlytic}, the first
of our main theorems, by establishing a stronger result.

We begin introduce a positivity assumption, called \textbf{H2,} which
is weaker than \textbf{H1} and allows us to work with differential
forms which do not satisfy $k$-uniform positivity. In particular,
this new positivity condition can be applied to product manifolds,
which will be crucial in later proofs.

We start by defining a pointwise structure on $M$. Let $\mathcal{T}_{p}M$
be the holomorphic tangent space of $M$ at $p.$ We fix a reference
$\kah$ metric as $\rho$ and for future convenience, define 
\begin{equation}
P:=\exp\rho=\sum_{k=0}^{n}\frac{\rho^{k}}{k!}.\label{eq:P=00003Dexprho}
\end{equation}
As before, we define $\mathring{\Lambda}=\Lambda-\Lambda^{[n]}$. 
\begin{defn}
\emph{\label{def:A-labeled-orthogonal}A labeled orthogonal splitting}
at $p\in M$ with respect to $\rho$ is a tuple 
\[
\mathcal{O}_{p}=(n_{p},\mathbf{d}_{p},\mathbf{V}_{p},\mathbf{k}_{p}),
\]
where $n_{p}\in\Z_{+}$, $\mathbf{d}_{p}=(d_{1},\cdots,d_{n_{p}})\in\Z_{\geq0}^{n_{p}}$
, $\mathbf{k}_{p}=(k_{1},\cdots,k_{n_{p}})\in\Z_{\geq0}^{n_{p}}$
and $\mathbf{V}_{p}=\{\mathcal{V}_{i}:\text{linear subspaces in }\mathcal{T}_{p}M\}$
such that 
\begin{enumerate}
\item $d_{i}=\dim_{\C}\mathcal{V}_{i}$ and $1\leq k_{i}\leq d_{i}-1$; 
\item $\mathcal{V}_{i}$ are mutually orthogonal with respect to $\rho$
and $\mathcal{T}_{p}M=\oplus_{i=1}^{n_{p}}\mathcal{V}_{i}.$ 
\end{enumerate}
Given $\mathcal{O}_{p}$ at $p$, we denote $\iota_{i},\pi_{i}$ be
the standard embedding and orthogonal projection for each $\mathcal{V}_{i}$,
respectively. Denote 
\begin{equation}
\rho_{i}=\pi_{i}^{*}\iota_{i}^{*}\rho.\label{eq:rho_i}
\end{equation}
If at each $p\in M$, there is a labeled orthogonal splitting $\mathcal{O}_{p}$,
we call $\mathcal{O}:=\{\mathcal{O}_{p}:p\in M\}$ a labeled orthogonal
splitting on $M$. 
\end{defn}

We propose the following 
\begin{defn}
\label{def:O-uniform positive}Suppose that there exists a labeled
orthogonal splitting $\mathcal{O}=\{(n_{p},\mathbf{d}_{p},\mathbf{V}_{p},\mathbf{k}_{p}):p\in M\}$
with respect to some reference $\kah$ metric $\rho$. $\mathring{\Lambda}$
is called $\mathcal{O}$-\emph{uniformly positive} (\emph{$\mathcal{O}$}-UP)
if there exists a uniform constant $m>0$ such that at each point
$p$, we have 
\begin{equation}
\mathring{\Lambda}-m\sum_{i=1}^{n_{p}}\frac{\rho_{i}^{k_{i}}}{k_{i}!}\geq0,\label{eq:H1 positivity}
\end{equation}
which means the $(l,l)$-component of the left hand side is positive
(See Definition \ref{def:An-element-positive forms}) for $l=1,\cdots,n-1$. 
\end{defn}

$\mathcal{O}$-uniform positivity is a weaker condition compared to
$k$-UP, and the later corresponds to $\{(1,n,\mathcal{T}_{p}M,k)\}$-uniform
positivity. However, as $\mathcal{O}$ depends on the choice of $\rho$,\emph{
$\mathcal{O}$}-UP property also depends on the choice of $\rho$.
Moreover, the $\mathcal{O}$-UP property is not preserved when descending
to a smooth subvariety, which is why we assume $k$-UP in Theorem
\ref{thm:numerical criterion}.

For $(n,n)$-form $\Lambda^{[n]}$, as discussed in the introduction,
we allow it to be slightly negative:
\begin{defn}
Given a labeled orthogonal splitting $\mathcal{O}=\{(n_{p},\mathbf{d}_{p},\mathbf{V}_{p},\mathbf{k}_{p}):p\in M\}$
with respect to $\rho$, we call $(n,n)$-form $\alpha$ \emph{an
almost positive volume form or almost positive} with respect to $(\mathcal{O},m,\rho)$,
if $\int_{M}\alpha\geq0$ and there exists $\epsilon=\epsilon(n,m,\kappa,n_{p},\mathbf{d}_{p},\mathbf{k}_{p},\omega_{0},\rho)>0$
such that 
\begin{equation}
\frac{\alpha}{\rho^{n}/n!}|_{p}>-\epsilon.\label{eq:fmin}
\end{equation}
\end{defn}

\begin{rem}
We will chose the precise number $\epsilon=\min\left\{ \frac{m}{4n+2}\gamma_{\min}(\frac{2\kappa}{m},n_{p},\mathbf{d}_{p},\mathbf{k}_{p}),\frac{\kappa\int_{M}\omega_{0}^{n}}{2\int_{M}\rho^{n}}\right\} $,
where $\gamma_{\min}$ is a positive number defined in (\ref{eq:gamm_min}).
Notice that $\epsilon$ can be chosen independent of $p$, since variables
$\mathbf{d}_{p},\mathbf{k}_{p}$ in $\gamma_{\min}$ only takes finitely
many values for all possible local labeled orthogonal splitting. 
\end{rem}

We impose the following hypothesis on $\Lambda$, which can be decomposed
into $\mathring{\Lambda}+\Lambda^{[n]}.$ 
\begin{defn}
We call $\Lambda=\mathring{\Lambda}+\Lambda^{[n]}$ satisfies the
condition \textbf{H2}, if $\mathring{\Lambda}$ is $\mathcal{O}$-uniformly
positive for some uniform constant $m$, and $\Lambda^{[n]}$ is almost
positive with respect to a specific $\epsilon$ and $(\mathcal{O},m,\rho)$;
Furthermore, there exists a constant $C_{H2}=C_{H2}(\Lambda,M)$ such
that for any $p\in M$, $\xi\in\mathcal{T}_{p}M$ with $\|\xi\|_{\rho}\leq1$,
and each $1\leq k\leq n$, we have 
\begin{align}
-C_{H2}\left(\Lambda^{[k]}+\sum_{l\in\boldsymbol{l}_{k}}\rho_{1}^{l_{1}}\cdots\rho_{n_{p}}^{l_{n_{p}}}\right) & \leq\text{Re}\left(\nabla_{\xi}\Lambda^{[k]}\right)\leq C_{H2}\left(\Lambda^{[k]}+\sum_{l\in\boldsymbol{l}_{k}}\rho_{1}^{l_{1}}\cdots\rho_{n_{p}}^{l_{n_{p}}}\right),\label{eq:assump1}\\
-C_{H2}\left(\Lambda^{[k]}+\sum_{l\in\boldsymbol{l}_{k}}\rho_{1}^{l_{1}}\cdots\rho_{n_{p}}^{l_{n_{p}}}\right) & \leq\nabla_{\xi\bar{\xi}}^{2}\Lambda^{[k]}\leq C_{H2}\left(\Lambda^{[k]}+\sum_{l\in\boldsymbol{l}_{k}}\rho_{1}^{l_{1}}\cdots\rho_{n_{p}}^{l_{n_{p}}}\right),\label{eq:assump 2}
\end{align}
where $\boldsymbol{l}_{k}=\{(l_{1},\cdots,l_{n_{p}}):\sum_{i}l_{i}=k,\ l_{i}=0\ \text{or}\ l_{i}\geq k_{i}\}$
and $\nabla$ is the Levi-Civita connection of $\rho$. 
\end{defn}

\begin{rem}
\label{rem:If-,-then h1}If $k\geq\max_{1\leq i\leq n_{p}}\{n-d_{i}+k_{i}\}$,
then (\ref{eq:assump1}) and (\ref{eq:assump 2}) hold automatically,
since $\rho^{k}\leq c_{0}(n_{p},k)\sum_{l\in\boldsymbol{l}_{k}}\rho_{1}^{l_{1}}\cdots\rho_{n_{p}}^{l_{n_{p}}}$
(see proof in Lemma \ref{lem:spplict 2}) and we may choose $C_{H2}=(\|\Lambda\|_{C^{2}}+1)\cdot\sup_{p}c_{0}(n_{p},k)$. 
\end{rem}

\begin{rem}
\textbf{H1 }is a stronger condition comparing to \textbf{H2}. First,
the $k_{0}$-UP\textbf{ }condition is equivalent to the $\mathcal{O}$-UP
condition for $\mathcal{O}=\{(1,n,\mathcal{T}_{p}M,k_{0}):p\in M\}$.
Second, since $\Lambda^{[k]}=0$ for $k<k_{0}$, (\ref{eq:assump1})
and (\ref{eq:assump 2}) hold for $k<k_{0}$. By Remark \ref{rem:If-,-then h1},
(\ref{eq:assump1}) and (\ref{eq:assump 2}) hold for $k\geq k_{0}$. 
\end{rem}

Comparing to \textbf{H1}, condition \textbf{H2} allows more general
choices of geometric data. However, condition \textbf{H2} may be more
difficult to verify. A crucial fact is following: 
\begin{example}
\label{exa:example}Let $\{(M_{i},\rho_{i},\mathcal{O}_{i})\}_{i=1}^{l}$
be several $\kah$ manifolds, each with labeled orthogonal splitting
\emph{$\mathcal{O}_{i}$}. Let $\Lambda_{i}$ be positive differential
forms on $M_{i}$ satisfying \textbf{H2} for some uniform constant
$m$, respectively. Let $\mathcal{M}:=\prod_{i=1}^{l}M_{i}$ be the
product manifold and $\text{pr}_{i}$ be the canonical projection
from $\mathcal{M}$ to $M_{i}$. Let $\rho=\sum_{i=1}^{l}\text{pr}_{i}^{*}\rho_{i}$.
Then $\mathcal{M}$ has a natural labeled orthogonal splitting at
each $\mathbf{p}=(p_{1},\cdots,p_{l})\in\mathcal{M}$ as 
\[
\mathcal{O}_{(p_{1},\cdots,p_{l})}=\left(\sum_{i=1}^{l}n_{p_{i}},(\mathbf{d}_{p_{1}},\cdots,\mathbf{d}_{p_{l}}),\cup_{i=1}^{l}\mathbf{V}_{p_{i}},(\mathbf{k}_{p_{1}},\cdots,\mathbf{k}_{p_{l}})\right).
\]
Let 
\[
\boldsymbol{\Lambda}:=\sum_{i=1}^{l}\text{pr}_{i}^{*}\Lambda_{i}.
\]
Then $\mathbf{\Lambda}$ satisfies $\mathcal{O}$-uniformly positivity.
Note the $\nabla_{\xi}\text{pr}_{i}^{*}\Lambda_{i}=\nabla_{\xi^{\top}}\text{pr}_{i}^{*}\Lambda_{i}$
where $\xi^{\top}$ is the tangential part of $\xi$ to $M_{i}$.
Thus, $\text{pr}_{i}\Lambda_{i}$ satisfies (\ref{eq:assump1}) and
(\ref{eq:assump 2}), and so does $\boldsymbol{\Lambda}$ (with $C_{H2}=\max\{C_{H2}(\Lambda_{i},M_{i})\}$).
On the other hand, if each $\Lambda_{i}$ is $k_{0i}$-uniformly positive
for some $k_{0i}\geq2$, i.e. $\Lambda_{i}^{[k_{0i}]}\geq m\rho_{i}^{k_{0i}}/k_{0i}!$
on $M_{i}$, one does not expect $\boldsymbol{\Lambda}\geq m\rho^{k}/k!$
for some $k\geq2$. Therefore, \textbf{H1 }does not necessarily hold
on $\mathcal{M}$.

We are now ready to state the main theorem of this part. 
\end{example}

\begin{thm}
\label{thm:Let--beanlytic-1}Let $M$ be a connected compact $\kah$
manifold of dimension $n$ with a fixed $\kah$ class $[\omega_{0}]$.
Suppose $\Lambda$ is a closed real differential form satisfying \textbf{H2}.
Then, there is a smooth solution of (\ref{eq:equation with f}) if
and only if there is a smooth subsolution of (\ref{eq:equation with f}). 
\end{thm}

From discussions above, it is clear that Theorem \ref{thm:Let--beanlytic}
is a direct consequence of Theorem \ref{thm:Let--beanlytic-1}. The
rest of this part is devoted to prove Theorem \ref{thm:Let--beanlytic-1}.
It is organized as follows. In Section \ref{sec:Preliminary-set-up},
we introduce necessary notations and definitions. In Section \ref{sec:Ellipticity-and-convexity},
we compute variations of the local functional and verify its ellipticity
and convexity. In Section \ref{sec:Cone-condition}, we study the
cone condition and state several equivalent expressions. In Section
\ref{sec:Equations-with-negative}, we extend results of previous
Sections to PDEs with almost positive volume forms. In Section \ref{sec:Continuity-method},
we use the continuity method to prove Theorem \ref{thm:Let--beanlytic-1}
by establishing key a priori estimates.

\section{Preliminary set-up\label{sec:Preliminary-set-up}}

In this section, we introduce some notations and discuss some basic
properties of our equation (\ref{eq:equation with f}).

\subsection{Positivity of differential forms}

We recall several definitions of positivity for $(k,k)$-forms from
Demailly \cite{demailly2012complex}, III,1,A. Our presentation is
slightly different due to our choice of notations. 
\begin{defn}
\label{def:An-element-positive forms} $u\in\bigwedge^{k,k}T_{p}^{*}M$
(resp. $\bigwedge^{k,k}T_{p}M$) is said to be a \emph{strongly positive}
\emph{form (resp. vector)} if it is of the form 
\[
u=\sum_{s\in I}\lambda_{s}\sqrt{-1}^{k}\alpha_{s,1}\wedge\overline{\alpha}_{s,2}\wedge\cdots\wedge\alpha_{s,k}\wedge\overline{\alpha}_{s,k},
\]
where $\alpha_{s,i}\in\bigwedge^{1,0}T_{p}^{*}M$ (resp. $\bigwedge^{0,1}T_{p}M$)
and $\lambda_{s}\geq0$.

An element $u\in\bigwedge^{k,k}T_{p}^{*}M$ (resp. $\bigwedge^{k,k}T_{p}M$)
is a\emph{ positive form (resp. vector) }if 
\[
\<u,v\>\geq0
\]
for any strongly positive $v\in\bigwedge^{k,k}T_{p}M$ (resp. $\bigwedge^{k,k}T_{p}^{*}M$).
Here $\<\cdot,\cdot\>$ is the canonical pairing (complex linear)
between $\bigwedge^{*}T_{p}M$ and $\bigwedge^{*}T_{p}^{*}M$. We
denote $\alpha\geq\beta$ (resp. $\alpha\leq\beta$) if for each $k=0,\cdots,n$,
$\alpha^{[k]}-\beta^{[k]}$ (resp. $\beta^{[k]}-\alpha^{[k]}$) is
a positive $(k,k)$-form/vector.

We denote $\alpha\ge_{s}\beta$ (resp. $\alpha\leq_{s}\beta$) if
$\alpha-\beta$ (resp. $\beta-\alpha$) is a strongly positive form/vector. 
\end{defn}

We list several direct conclusions from Definition \ref{def:An-element-positive forms}
and compare these to our uniform positivity concept. A strongly positive
form (vector) is also positive. The converse is true for $k=0,1,n-1,n$;
However it is false for $k=2,\cdots,n-2$ if $n\geq4.$ Also, the
cone of positive $(k,k)$-forms is the dual cone of strongly positive
$(k,k)$-vectors. It is obvious that both positivity and strongly
positivity are invariant under coordinate change. If $\rho$ is a
$\kah$ form, then $\rho^{k}$ is a strongly positive $(k,k)$-form.
Moreover, $\rho^{k}$ is $k$-uniformly positive by \ref{eq:uniformly positive}.
A uniformly positive form is positive, but not necessarily strongly
positive. For instance, pick a weakly positive but not strongly positive
$(k,k)$-form $\alpha$. Then $\rho^{k}+\alpha$ is uniformly positive
but not strongly positive.

The following lemma shows that $\exp\omega$ is strongly positive
if $\omega$ is a non-negative $(1,1)$-form. The same argument applies
to non-negative $(1,1)$-tangent vectors. 
\begin{lem}
\label{lem:exp storng positive}Let $A$ be a non-negative Hermitian
matrix. Let $\omega=A_{i\bar{j}}\frac{\sqrt{-1}}{2}dz^{i}\wedge d\bar{z}^{j}$.
Then $\exp\omega$ is a strongly positive form. Moreover, if both
$A$,$B$ are non-negative Hermitian matrix, $\omega'=B_{i\bar{j}}\frac{\sqrt{-1}}{2}dz^{i}\wedge d\bar{z}^{j}$
and $A\geq B$, then $\exp\omega\geq_{s}\exp\omega'$ as a strongly
positive form. 
\end{lem}

\begin{proof}
We may assume that $A$ is non-singular; otherwise, we just restrict
to the column space of $A$. Since strongly positivity is invariant
under linear transform, we may assume that $A$ is the identity matrix.
Furthermore, we may assume that $B=\sum_{i=1}^{n}\lambda_{i}\frac{\sqrt{-1}}{2}dz^{i}\wedge d\bar{z}^{i}$.
Then, clearly, $A\geq B$ which implies $0\leq\lambda_{i}\leq1$.
We then compute
\[
\exp\omega=1+\sum_{k=1}^{n}\sum_{|I|=k}\prod_{i\in I}\frac{\sqrt{-1}}{2}dz^{i}\wedge d\bar{z}^{i},
\]
\[
\exp\omega-\exp\omega'=\sum_{I}(1-\prod_{i\in I}\lambda_{i})\prod_{i\in I}\frac{\sqrt{-1}}{2}dz^{i}\wedge d\bar{z}^{i},
\]
where $I$ runs over all ordered subsets of $\{1,\cdots,n\}$. It
is clear that the coefficient $1-\prod_{i\in I}\lambda_{i}$ is always
non-negative, which proves that $\exp\omega\geq_{s}\exp\omega'$ is
strongly positive. 
\end{proof}

\subsection{Point-wise setup}

We use $\Gamma_{n\times n},\Gamma_{n\times n}^{+},\overline{\Gamma_{n\times n}^{+}}$
to denote the set of all Hermitian matrices, the set of positive definite
Hermitian matrices, and the set of non-negative Hermitian matrices,
respectively.

For $p\in M,$ we pick a local normal coordinate near $p$ with respect
to $\rho.$ Therefore, $\rho=\frac{\sqrt{-1}}{2}\sum_{i,j}\delta_{i\bar{j}}dz^{i}\wedge d\bar{z}^{j}$
at $p$. We may write 
\[
\omega=\frac{\sqrt{-1}}{2}\sum_{i,j}A_{i\bar{j}}dz^{i}\wedge d\bar{z}^{j},
\]
where $A=(A_{i\bar{j}})$ is a positive definite Hermitian matrix.
Let $(A^{\bar{j}i}$) be the inverse matrix of $A$; i.e. $A_{i\bar{j}}A^{\bar{j}k}=\delta_{i}^{k}.$
We define the following local functional 
\begin{equation}
F(A)=F(A,\Lambda):=\frac{(\Lambda\wedge\Omega)^{[n]}}{\Omega^{[n]}}.\label{eq:definition of F}
\end{equation}
A coordinate change preserves $\frac{(\Lambda\wedge\Omega)^{[n]}}{\Omega^{[n]}}$,
and hence $F$ is invariant under a coordinate change. Equation (\ref{eq:eq with Berezin integ})
can be re-written as 
\begin{equation}
F(A)=\kappa.\label{eq:equation by F}
\end{equation}

\begin{defn}
\label{def:We-define-a}Notations as above. We define a canonical
$(1,1)$-vector $\chi\in\bigwedge^{1,1}T_{p}M$ as
\begin{equation}
\chi=2\sqrt{-1}\sum_{i,j}A^{\bar{j}i}\frac{\partial}{\partial\bar{z}^{j}}\wedge\frac{\partial}{\partial z^{i}}.\label{eq:def for chi}
\end{equation}
$\chi$ is a strongly positive $(1,1)$-vector which induces a Hermitian
metric on $T^{*}M$ that is dual to $\omega$.
\end{defn}

Let $\<\cdot,\cdot\>$ be the complex bi-linear pairing $\bigwedge^{*}(T_{p}^{*}M)\times\bigwedge^{*}(T_{p}M)\to\mathbb{C}.$
A direct computation shows that 
\begin{align}
\<\frac{\omega^{k}}{k!},\frac{\chi^{k}}{k!}\>=\frac{n!}{k!(n-k)!},\ \ \<\frac{\rho^{k}}{k!},\frac{\chi^{k}}{k!}\>=\sigma_{k}(A^{-1}),\label{eq:basic formula}
\end{align}
where $\sigma_{k}(\cdot)$ is the $k$-th elementary symmetric function
of eigenvalues of a Hermitian matrix.

With a labeled orthogonal splitting $\mathcal{O}_{p}=\{n_{p},\mathbf{d}_{p},\{\mathcal{V}_{i}\},\mathbf{k}_{p}\}$
at $p$, we first fix a normal coordinate $\{z^{i}\}$ at $p$ of
$\rho$ such that $\{\sqrt{2}\frac{\pdv}{\pdv z^{i}}\}$ restricts
to a unitary basis on each $\mathcal{V}_{i}$. Suppose that $\rho_{i}=\pi_{i}^{*}\iota_{i}^{*}\rho$
with respect to $\mathcal{O}_{p}$. We may write under this frame
\[
A^{-1}=\left(\begin{array}{cccc}
A^{-1}|_{\mathcal{V}_{1}} & * & * & *\\*
* & A^{-1}|_{\mathcal{V}_{2}} & * & *\\*
* & * & \ddots & *\\*
* & * & * & A^{-1}|_{\mathcal{V}_{n_{p}}}
\end{array}\right),
\]
where $A_{i}\in\Gamma_{d_{i}\times d_{i}}$ . Then 
\begin{equation}
\<\frac{\rho_{i}^{k}}{k!},\frac{\chi^{k}}{k!}\>=\<\iota_{i}^{*}\frac{\rho^{k}}{k!},(\pi_{i})_{*}\frac{\chi^{k}}{k!}\>=\sigma_{k}(A^{-1}|_{\mathcal{V}_{i}}).\label{eq:basic formula for rho_i}
\end{equation}

To future convenience, we define, for $k\leq n,$ 
\begin{equation}
F_{k}(A)=\frac{n!\Lambda^{[k]}\wedge\omega^{n-k}}{(n-k)!\omega^{n}}.\label{eq:F_k}
\end{equation}
Denote $\exp\chi=\sum_{k=0}^{n}\frac{\chi^{k}}{k!}.$ The following
lemma shows that $F_{k}$ and $F$ can be represented by $\chi$. 
\begin{lem}
\label{lem:rewrite}Notations as above. We have 
\begin{equation}
F_{k}(A)=\<\Lambda^{[k]},\frac{\chi^{k}}{k!}\>.\label{eq:Fk inverse exp}
\end{equation}
\begin{equation}
F(A)=\<\Lambda,\exp\chi\>.\label{eq:F inverse exp}
\end{equation}
\end{lem}

\begin{proof}
It is enough to prove (\ref{eq:Fk inverse exp})) for $k\leq n-1$.
For two ordered index sets $I=\{i_{1}<i_{2}<\cdots<i_{k}\},\ J=\{j_{1}<\cdots<j_{k}\}$,
we denote 
\begin{align*}
\sqrt{-1}^{k^{2}}dz^{I}\wedge d\bar{z}^{J} & =\sqrt{-1}^{k}dz^{i_{1}}\wedge d\bar{z}^{j_{1}}\wedge\cdots\wedge dz^{i_{k}}\wedge d\bar{z}^{j_{k}},\\
\sqrt{-1}^{k^{2}}\frac{\pdv}{\pdv\bar{z}^{J}}\wedge\frac{\pdv}{\pdv z^{I}} & =\sqrt{-1}^{k}\frac{\pdv}{\pdv\bar{z}^{j_{1}}}\wedge\frac{\pdv}{\pdv z^{i_{1}}}\wedge\cdots\wedge\frac{\pdv}{\pdv\bar{z}^{j_{k}}}\wedge\frac{\pdv}{\pdv z^{i_{k}}}.
\end{align*}
Thus, we have the following 
\begin{equation}
\Lambda^{[k]}=\frac{\sqrt{-1}^{k^{2}}}{2^{k}}\sum_{|I|=|J|=k}\Lambda_{I,\bar{J}}dz^{I}\wedge d\bar{z}^{J},\label{eq:add3}
\end{equation}
and 
\begin{equation}
\frac{\omega^{k}}{k!}=\frac{\sqrt{-1}^{k^{2}}}{2^{k}}\sum_{|I|=|J|=k}A_{I,\bar{J}}dz^{I}\wedge d\bar{z}^{J},\ \frac{\chi^{k}}{k!}=2^{k}\sqrt{-1}^{k^{2}}\sum_{|I|=|J|=k}A^{\bar{J},I}\frac{\pdv}{\pdv\bar{z}^{J}}\wedge\frac{\pdv}{\pdv z^{I}},\label{eq:add4}
\end{equation}
where 
\[
A_{I,\bar{J}}:=\det\left(\begin{array}{cccc}
A_{i_{1}\overline{j_{1}}} & A_{i_{1}\overline{j_{2}}} & \cdots & A_{i_{1}\overline{j_{k}}}\\
A_{i_{2}\overline{j_{1}}} & A_{i_{2}\overline{j_{2}}} & \cdots & A_{i_{2}\overline{j_{k}}}\\
\vdots & \vdots & \ddots & \vdots\\
A_{i_{k}\overline{j_{1}}} & A_{i_{k}\overline{j_{2}}} & \ldots & A_{i_{k}\overline{j_{k}}}
\end{array}\right),\ A^{\bar{J},I}:=\det\left(\begin{array}{cccc}
A^{\overline{j_{1}}i_{1}} & A^{\overline{j_{1}}i_{2}} & \cdots & A^{\overline{j_{1}}i_{k}}\\
A^{\overline{j_{2}}i_{1}} & A^{\overline{j_{2}}i_{2}} & \cdots & A^{\overline{j_{2}}i_{k}}\\
\vdots & \vdots & \ddots & \vdots\\
A^{\overline{j_{k}}i_{1}} & A^{\overline{j_{k}}i_{2}} & \ldots & A^{\overline{j_{k}}i_{k}}
\end{array}\right).
\]
As a consequence of (\ref{eq:add3}) and (\ref{eq:add4}), we have
\[
\Lambda^{[k]}\wedge\frac{\omega^{n-k}}{(n-k)!}=\sum_{|I|=|J|=k}\epsilon_{J,J^{c}}^{I,I^{c}}\Lambda_{I,\bar{J}}A_{I^{c},\bar{J^{c}}}\frac{\rho^{n}}{n!}
\]
where index sets $I^{c}$, $J^{c}$ are ordered complement of $I$
and $J$; and $\epsilon_{J,J^{c}}^{I,I^{c}}=\frac{\sqrt{-1}^{n^{2}}dz^{I}\wedge dz^{I^{c}}\wedge d\bar{z}^{J}\wedge d\bar{z}^{J^{c}}}{2^{n}\rho^{n}/n!}$.

We make the following claim: 
\begin{equation}
\frac{A_{I^{c},\bar{J^{c}}}}{\det A}\epsilon_{J,J^{c}}^{I,I^{c}}=A^{\bar{J},I}.\label{eq:ele lin ag}
\end{equation}
After proper row and column permutations, we may assume without loss
of generality that $I=\{1,2,\cdots,k\}$ and $J=\{1,\cdots,k\}$.
We decompose as follows 
\[
A=\left(\begin{array}{cc}
A_{1} & C'\\
C & A_{2}
\end{array}\right),\ A^{-1}=\left(\begin{array}{cc}
B_{1} & D'\\
D & B_{2}
\end{array}\right),
\]
where $A_{1}$ is an $(n-k)\times(n-k)$ matrix and $B_{1}$ is a
$k\times k$ matrix. Since the sign change of the permutation is $\epsilon_{J,J^{c}}^{I,I^{c}}$,
in order to prove (\ref{eq:ele lin ag}), it suffices to show $\det(A_{1})=\det A\cdot\det B_{2}$.
Suppose first that $A_{1}$ is nonsingular. Let $E=-CA_{1}^{-1}$,
$E'=-A_{1}^{-1}C'$. Since 
\[
\left(\begin{array}{cc}
I & 0\\
E & I
\end{array}\right)A\left(\begin{array}{cc}
I & E'\\
0 & I
\end{array}\right)=\left(\begin{array}{cc}
A_{1} & 0\\
0 & A_{2}'
\end{array}\right).
\]
We may write 
\[
\left(\begin{array}{cc}
I & -E'\\
0 & I
\end{array}\right)A^{-1}\left(\begin{array}{cc}
I & 0\\
-E & I
\end{array}\right)=\left(\begin{array}{cc}
\tilde{B_{1}} & \tilde{D}'\\
\tilde{D} & B_{2}
\end{array}\right).
\]
Thus 
\begin{equation}
\left(\begin{array}{cc}
\tilde{B_{1}} & \tilde{D}'\\
\tilde{D} & B_{2}
\end{array}\right)\left(\begin{array}{cc}
A_{1} & 0\\
0 & A_{2}'
\end{array}\right)=\left(\begin{array}{cc}
I & 0\\
0 & I
\end{array}\right).\label{eq:elem lin ag}
\end{equation}
(\ref{eq:elem lin ag}) implies 
\[
B_{2}A_{2}'=I.
\]
Hence $\det B_{2}\det A_{2}'=1$. Notice that $\det A=\det A_{1}\det A_{2}'$.
Therefore, $\det(A_{1})=\det(A)\det(B_{2})$. We have proved the claim
(\ref{eq:ele lin ag}) for non-singular $A_{1}$. If $A_{1}$ is singular,
we consider the perturbation $A_{\epsilon}=A+\epsilon I$ for suitable
$\epsilon$ such that both $A_{1}$ and $A$ are non-singular. Then
the claim follows from the continuity of the inversion and determinant
functions. We have now established (\ref{eq:ele lin ag}).

To finish the proof of our lemma, we observe that from (\ref{eq:add3})
\begin{equation}
\left\langle \Lambda^{[k]},\left(2^{k}\sqrt{-1}^{k^{2}}\frac{\pdv}{\pdv\bar{z}^{J}}\wedge\frac{\pdv}{\pdv z^{I}}\right)\right\rangle =\Lambda_{I,\bar{J}}.\label{eq:La_IJ}
\end{equation}
Thus, by (\ref{eq:ele lin ag}), 
\begin{align}
\<\Lambda^{[k]},\chi^{k}/k!\> & =\<\Lambda^{[k]},A^{\bar{J},I}2^{k}\sqrt{-1}^{k^{2}}(\frac{\partial}{\partial\bar{z}^{J}}\wedge\frac{\pdv}{\pdv z^{I}})\>\label{eq:F_k dfe 2}\\
 & =\left(\sum_{|I|=|J|=k}\Lambda_{I,\bar{J}}A^{\bar{J},I}\right)\nonumber \\
 & =\sum_{|I|=|J|=k}\Lambda_{I,\bar{J}}\frac{A_{I^{c},\bar{J^{c}}}}{\det A}\epsilon_{J,J^{c}}^{I,I^{c}}\nonumber \\
 & =F_{k}(A).\nonumber 
\end{align}
We have finished the proof. 
\end{proof}
For future use, we record the following linear algebraic fact. 
\begin{lem}
\label{rem:elementary linear agb}Given Hermitian matrices $V,H$,
and $A=\left(\begin{array}{cc}
H & D\\
D^{\dagger} & V
\end{array}\right)$ which are invertible, then 
\begin{equation}
A^{-1}=\left(\begin{array}{cc}
\hat{H}^{-1} & -\hat{H}{}^{-1}DV^{-1}\\
-V^{-1}D^{\dagger}\hat{H}^{-1} & V^{-1}+V^{-1}D^{\dagger}\hat{H}^{-1}DV^{-1}
\end{array}\right),\label{eq:alg fact}
\end{equation}
where $\hat{H}=\left(H-DV^{-1}D^{\dagger}\right)$, $\hat{V}=V-D^{\dagger}H^{-1}D$,
and $V^{-1}+V^{-1}D^{\dagger}\hat{H}^{-1}DV^{-1}=\hat{V}^{-1}$. Furthermore,
if $\hat{H}>0$ then 
\begin{equation}
A^{-1}\geq\left(\begin{array}{cc}
0 & 0\\
0 & V^{-1}
\end{array}\right).\label{lem:Alg2}
\end{equation}
\end{lem}

\begin{proof}
(\ref{eq:alg fact}) can be proved by column and row operation as
in the proof of Lemma \ref{lem:rewrite}. By symmetry, we have also
the following 
\[
A^{-1}=\left(\begin{array}{cc}
H^{-1}+H^{-1}D\hat{V}^{-1}D^{\dagger}H^{-1} & -H^{-1}D\hat{V}^{-1}\\
-\hat{V}^{-1}D^{\dagger}H^{-1} & \hat{V}^{-1}
\end{array}\right).
\]
For (\ref{lem:Alg2}), let $C=-\hat{H}{}^{-1}DV^{-1}=-H^{-1}D\hat{V}^{-1}$.
We have 
\begin{align*}
A^{-1}-\left(\begin{array}{cc}
0 & 0\\
0 & V^{-1}
\end{array}\right) & =\left(\begin{array}{cc}
\hat{H}^{-1} & C\\
C^{\dagger} & C^{\dagger}\hat{H}C
\end{array}\right)\\
 & =\left(\begin{array}{cc}
I & 0\\
C^{\dagger} & I
\end{array}\right)\left(\begin{array}{cc}
\hat{H}^{-1} & 0\\
0 & 0
\end{array}\right)\left(\begin{array}{cc}
I & C\\
0 & I
\end{array}\right).
\end{align*}
Since $\left(\begin{array}{cc}
\hat{H}^{-1} & 0\\
0 & 0
\end{array}\right)\geq0$, $A^{-1}\geq\left(\begin{array}{cc}
0 & 0\\
0 & V^{-1}
\end{array}\right)$. 
\end{proof}

\section{Ellipticity and convexity\label{sec:Ellipticity-and-convexity}}

In this section, we discuss the monotonicity and convexity of $F$,
when certain positivity condition on $\Lambda$ holds. First, we compute
the first and second variations of $F(A)$ using notations set earlier.
We use abbreviations $F^{i\bar{j}}$, $F^{i\bar{j},r\bar{s}}$, $F_{k}^{i\bar{j}}$,
and $F_{k}^{i\bar{j},r\bar{s}}$ to represent $\frac{\pdv F}{\pdv A^{i\bar{j}}}$,
$\frac{\pdv^{2}F}{\pdv A^{i\bar{j}}\pdv A^{r\bar{s}}}$, $\frac{\pdv F_{k}}{\pdv A^{i\bar{j}}}$,
and $\frac{\pdv^{2}F_{k}}{\pdv A^{i\bar{j}}\pdv A^{r\bar{s}}}$, respectively. 
\begin{lem}
Notations as above. 
\begin{equation}
F_{k}^{i\bar{j}}=-\sum_{ab}A^{\bar{a}i}A^{\bar{j}b}\<\Lambda^{[k]},\frac{\chi^{k-1}}{(k-1)!}\wedge2\sqrt{-1}\frac{\partial}{\partial\bar{z}^{a}}\wedge\frac{\partial}{\partial z^{b}}\>.\label{eq:lem of comp}
\end{equation}
\begin{align}
F_{k}^{i\bar{j},r\bar{s}} & =\sum_{a,b,c,d}A^{\bar{a}i}A^{\bar{j}b}A^{\bar{c}r}A^{\bar{s}d}\<\Lambda^{[k]},\left(2\sqrt{-1}\right)^{2}\frac{\chi^{k-2}}{(k-2)!}\wedge\frac{\partial}{\partial\bar{z}^{a}}\wedge\frac{\partial}{\partial z^{b}}\wedge\frac{\partial}{\partial\bar{z}^{c}}\wedge\frac{\partial}{\partial z^{d}}\>\label{eq:lem of comp-1}\\
 & +\sum_{a,b,c,d}\left(A^{\bar{a}r}A^{\bar{s}i}A^{\bar{j}b}+A^{\bar{a}i}A^{\bar{j}r}A^{\bar{s}b}\right)\<\Lambda^{[k]},\left(2\sqrt{-1}\frac{\chi^{k-1}}{(k-1)!}\frac{\partial}{\partial\bar{z}^{a}}\wedge\frac{\partial}{\partial z^{b}}\right)\>.\nonumber 
\end{align}
\end{lem}

\begin{proof}
Recall $\chi=2\sqrt{-1}A^{\bar{j}i}\frac{\pdv}{\pdv\bar{z}^{j}}\wedge\frac{\pdv}{\pdv z^{i}}$.
We use Lemma \ref{lem:rewrite} to find 
\[
\frac{\pdv F_{k}}{\pdv A^{\bar{j}i}}=\<\Lambda^{[k]},\frac{\chi^{k-1}}{(k-1)!}\wedge2\sqrt{-1}\frac{\partial}{\partial\bar{z}^{j}}\wedge\frac{\partial}{\partial z^{i}}\>.
\]
Therefore 
\begin{align}
F_{k}^{i\bar{j}} & =\frac{\pdv A^{\bar{a}b}}{\pdv A_{i\bar{j}}}\cdot\frac{\pdv F_{k}}{\pdv A^{\bar{a}b}}\label{eq:comp Fij}\\
 & =-\sum_{a,b}A^{\bar{a}i}A^{\bar{j}b}\<\Lambda^{[k]},\frac{\chi^{k-1}}{(k-1)!}\wedge2\sqrt{-1}\frac{\partial}{\partial\bar{z}^{a}}\wedge\frac{\partial}{\partial z^{b}}\>.\nonumber 
\end{align}
Furthermore, 
\begin{align}
F_{k}^{i\bar{j},r\bar{s}} & =\sum_{a,b,c,d}A^{\bar{a}i}A^{\bar{j}b}A^{\bar{c}r}A^{\bar{s}d}\<\Lambda^{[k]},\left(\frac{1}{k!}\frac{\pdv^{2}\chi^{k}}{\pdv A^{\bar{c}d}\pdv A^{\bar{a}b}}\right)\>\label{eq:Compu Fijrs}\\
 & \ +\sum_{a,b,c,d}\left(A^{\bar{a}r}A^{\bar{s}i}A^{\bar{j}b}+A^{\bar{a}i}A^{\bar{j}r}A^{\bar{s}b}\right)\<\Lambda^{[k]},\left(\frac{\pdv}{\pdv A^{\bar{a}b}}\left(\frac{\chi^{k}}{k!}\right)\right)\>\nonumber \\
 & =\sum_{a,b,c,d}A^{\bar{a}i}A^{\bar{j}b}A^{\bar{c}r}A^{\bar{s}d}\<\Lambda^{[k]},\left(\left(2\sqrt{-1}\right)^{2}\frac{\chi^{k-2}}{(k-2)!}\right)\frac{\partial}{\partial\bar{z}^{a}}\wedge\frac{\partial}{\partial z^{b}}\wedge\frac{\partial}{\partial\bar{z}^{c}}\wedge\frac{\partial}{\partial z^{d}}\>\nonumber \\
 & +\sum_{a,b,c,d}\left(A^{\bar{a}r}A^{\bar{s}i}A^{\bar{j}b}+A^{\bar{a}i}A^{\bar{j}r}A^{\bar{s}b}\right)\<\Lambda^{[k]},2\sqrt{-1}\frac{\chi^{k-1}}{(k-1)!}\wedge\frac{\partial}{\partial\bar{z}^{a}}\wedge\frac{\partial}{\partial z^{b}}\>.\nonumber 
\end{align}
\end{proof}
\begin{lem}
\label{lem:F_k monot}Notations as above. We assume that $\Lambda$
satisfies condition \textbf{H2}. For any covector $b=b_{i}dz^{i}\not=0$,
we have 
\begin{equation}
-F_{k}^{i\bar{j}}b_{i}\overline{b_{j}}\geq0.\label{eq:Fijbibj}
\end{equation}
Moreover, 
\begin{equation}
-\sum_{k=1}^{n-1}F_{k}^{i\bar{j}}b_{i}\overline{b}_{j}>0.\label{eq:Fijbibj-1}
\end{equation}
\end{lem}

\begin{proof}
Let $b^{\sharp}=\bar{b_{i}}A^{\bar{i}a}\frac{\pdv}{\pdv z^{a}}$ be
the dual of $b$ raised by $\omega$. Then by (\ref{eq:lem of comp})
\begin{align}
-F_{k}^{i\bar{j}}b_{i}\overline{b_{j}} & =\<\Lambda^{[k]},2\sqrt{-1}\overline{b^{\sharp}}\wedge b^{\sharp}\wedge\exp\chi\>.\label{eq:Fijbibj3}\\
 & \geq0.\nonumber 
\end{align}
We have used the fact that $2\sqrt{-1}\overline{b^{\sharp}}\wedge b^{\sharp}\wedge\exp\chi$
is strongly positive.

By condition \textbf{H2}, we may choose a subspace $\mathcal{V}_{i}\subset\mathcal{T}_{p}M$
such that $(\pi_{i})_{*}(b^{\sharp})\not=0$ and $\Lambda^{[k_{i}]}\geq m\frac{\left(\rho_{i}\right)^{k_{i}}}{k_{i}!}.$Then
\begin{equation}
\Lambda^{[k_{i}]}\geq m\frac{\rho_{i}^{k_{i}}}{k_{i}!}\geq m'(p)\frac{(\pi_{i}^{*}\iota_{i}^{*}\omega)^{k_{i}}}{k_{i}!},\label{eq:fijbibj3}
\end{equation}
where $m'(p)$ is a positive constant at $p$. Then 
\begin{align}
\<\Lambda^{[k_{i}]},2\sqrt{-1}\frac{\overline{b^{\sharp}}\wedge b^{\sharp}\wedge\chi^{k_{i}-1}}{(k_{i}-1)!} & \>\geq m'(p)\<\frac{(\pi_{i}^{*}\iota_{i}^{*}\omega)^{k_{i}}}{k_{i}!},2\sqrt{-1}\frac{\overline{b^{\sharp}}\wedge b^{\sharp}\wedge\chi^{k_{i}-1}}{(k_{i}-1)!}\>>0.\label{eq:strict monoton}
\end{align}
\end{proof}
\begin{prop}
\label{prop: F_k convexity}Notations as above. If $\Lambda$ satisfies
condition \textbf{H2,} for any local complex $n\times n$ matrix $B_{i\bar{j}}$,
\begin{equation}
\sum_{i,j,r,s}\left(\frac{\pdv^{2}F_{k}(A)}{\pdv A_{i\bar{j}}\pdv A_{r\bar{s}}}+\frac{\pdv F_{k}(A)}{\pdv A_{i\bar{s}}}A^{\bar{j}r}\right)B_{i\bar{j}}\overline{B_{s\bar{r}}}\geq0.\label{eq:ineq}
\end{equation}
\end{prop}

\begin{proof}
Define the matrix $C$ and the corresponding element 
\begin{align}
C^{\bar{a}b} & :=\sum_{i,j}A^{\bar{a}i}B_{i\bar{j}}A^{\bar{j}b},\label{eq:noname}\\
\zeta & :=2\sqrt{-1}\sum_{a,b}C^{\bar{a}b}\frac{\partial}{\partial\bar{z}^{a}}\wedge\frac{\partial}{\partial z^{b}}.\label{eq:zeta}
\end{align}
In addition, we use $C^{\dagger}$ to denote the adjoint matrix $\left(C^{\dagger}\right)^{\bar{a}b}:=\overline{C^{\bar{b}a}}$.
Define 
\begin{align}
\left(D\right)^{\bar{a}b}:=\sum_{i,j,r,s}A^{\bar{a}i}A^{\bar{s}b}A^{\bar{j}r}B_{i\bar{j}}\overline{B_{s\bar{r}}} & =\sum_{r,s}C^{\bar{a}r}A_{r\bar{s}}\overline{C^{\bar{b}s}}.\label{eq:D noname}
\end{align}
Then, 
\begin{align}
\sum_{i,j,r,s}A^{\bar{a}r}A^{\bar{s}i}A^{\bar{j}b}B_{i\bar{j}}\overline{B_{s\bar{r}}} & =\overline{C^{\bar{r}a}}A_{r\bar{s}}C^{\bar{s}b}=(D^{\dagger})^{\bar{a}b},\label{eq:xompu}\\
\sum_{i,j,r,s}A^{\bar{a}i}A^{\bar{j}b}A^{\bar{c}r}A^{\bar{s}d}B_{i\bar{j}}\overline{B_{s\bar{r}}} & =C^{\bar{a}b}\overline{C^{\bar{d}c}}.\label{eq:compute C}
\end{align}
We define 
\begin{equation}
\xi:=2\sqrt{-1}\sum_{a,b}D^{\bar{a}b}\frac{\partial}{\partial\bar{z}^{a}}\wedge\frac{\partial}{\partial z^{b}}.\label{eq:xi din}
\end{equation}
Therefore 
\begin{align}
\sum_{i,j,r,s}\left(F_{k}^{i\bar{j},r\bar{s}}+\frac{\pdv F_{k}}{\pdv A_{i\bar{s}}}A^{\bar{j}r}\right)B_{i\bar{j}}\overline{B_{s\bar{r}}} & =\<\Lambda^{[k]},\frac{\chi^{k-2}}{(k-2)!}\wedge\bar{\zeta}\wedge\zeta\>\label{eq:prop f convex}\\
 & +\<\Lambda^{[k]},\frac{\chi^{k-1}}{(k-1)!}\wedge\xi\>.\nonumber 
\end{align}
Define 
\begin{equation}
\Theta_{k}(B,B)=\frac{\chi^{k-2}}{(k-2)!}\wedge(2\sqrt{-1}\bar{\zeta}\wedge\zeta)+\frac{\chi^{k-1}}{(k-1)!}\wedge\xi.\label{eq:Theta_k}
\end{equation}
From (\ref{eq:prop f convex}), we see that 
\begin{equation}
\sum_{i,j,r,s}\left(F_{k}^{i\bar{j},r\bar{s}}+\frac{\pdv F_{k}}{\pdv A_{i\bar{s}}}A^{\bar{j}r}\right)B_{i\bar{j}}\overline{B_{s\bar{r}}}=\<\Lambda^{[k]},\Theta_{k}(B,B)\>.\label{eq:inte Fijrs}
\end{equation}
Now, it suffices to check the (strongly) positivity of $\Theta_{k}(B,B)$.
Since the positivity is invariant under coordinate change, we may
change the local coordinate such that at a point $p$, it holds 
\begin{equation}
\chi=2\sqrt{-1}\sum_{i}^{n}\frac{\partial}{\partial\bar{z}^{i}}\wedge\frac{\partial}{\partial z^{i}},\ \ \zeta=2\sqrt{-1}\sum_{i=1}^{n}a_{i}\frac{\partial}{\partial\bar{z}^{i}}\wedge\frac{\partial}{\partial z^{i}}.\label{eq:another coor xhi zeta}
\end{equation}
Therefore, by (\ref{eq:another coor xhi zeta}), 
\begin{equation}
\frac{\chi^{k-2}}{(k-2)!}\wedge\bar{\zeta}\wedge\zeta=\sum_{|J|=k-2}\sum_{i,j\not\in J}a_{i}a_{j}2^{k}\sqrt{-1}^{k^{2}}\frac{\partial}{\partial\bar{z}^{J}}\wedge\frac{\pdv}{\pdv\bar{z}^{i}}\wedge\frac{\pdv}{\pdv\bar{z}^{j}}\wedge\frac{\pdv}{\pdv z^{J}}\wedge\frac{\pdv}{\pdv z^{i}}\wedge\frac{\pdv}{\pdv z^{j}},\label{eq:zeta 2}
\end{equation}
\begin{equation}
\frac{\chi^{k-1}}{(k-1)!}\wedge\left(\xi\right)=\sum_{|J|=k-1}\sum_{i\not\in J}a_{i}^{2}2^{k}\sqrt{-1}^{k^{2}}\frac{\partial}{\partial\bar{z}^{J}}\wedge\frac{\pdv}{\pdv\bar{z}^{i}}\wedge\frac{\pdv}{\pdv z^{J}}\wedge\frac{\pdv}{\pdv z^{i}}.\label{eq:xi 2}
\end{equation}
Let $J$ be any ordered subset of $\{1,\cdots,n\}$ such that $|J|=k$.
The coefficient of $\sqrt{-1}^{k^{2}}2^{k}\frac{\partial}{\partial\bar{z}^{J}}\wedge\frac{\pdv}{\pdv z^{J}}$
in $\Theta_{k}(B,B)$ is given by 
\[
\sum_{i\in J}a_{i}^{2}+\sum_{i\in J}\sum_{j\not=i,j\in J}a_{i}a_{j}=\left(\sum_{i\in J}a_{i}\right)^{2}.
\]
Since each $\sqrt{-1}^{k^{2}}2^{k}\frac{\partial}{\partial\bar{z}^{J}}\wedge\frac{\pdv}{\pdv z^{J}}$
is strongly positive, we conclude that $\Theta_{k}(B,B)$ is strongly
positive.

For any general complex matrix $B$, we consider the matrix decomposition
\[
B=B^{R}+\sqrt{-1}B^{I},
\]
where 
\[
B^{R}=\frac{B+B^{\dagger}}{2},\ B^{I}=\frac{B-B^{\dagger}}{2\sqrt{-1}}.
\]
Thus, $B^{R}$ and $B^{I}$ are both Hermitian. The Hermitian bilinear
form $h(B,B)=\<\Lambda^{[k]},\Theta_{k}(B,B)\>$ is given by 
\begin{align*}
h(B,B) & =h(B^{R},B^{R})+h(B^{I},B^{I})+\sqrt{-1}\left(h(B^{I},B^{R})-h(B^{R},B^{I})\right)\\
 & =h(B^{R},B^{R})+h(B^{I},B^{I}).
\end{align*}
Hence $h$ is positive definite. Therefore, we have finished the proof. 
\end{proof}
Now, we establish the ellipticity and convexity of $F$. In this section,
we consider the simpler case when $\Lambda^{[n]}=f\frac{\rho^{n}}{n!}\geq0.$ 
\begin{cor}
\label{cor: convexity for f bigger than 0}Notations as above. Suppose
that $\Lambda^{[n]}\geq0$ at some point $p$. Then at $p$, $F(A)$
is a strictly decreasing function in $A$, i.e. 
\begin{equation}
F(A+B)<F(A)\label{eq:monotonicity f non neg}
\end{equation}
for any non-zero semi-positive Hermitian matrix $B$. Furthermore,
for any complex matrix $B_{i\bar{j}}$, we have 
\[
\sum_{i,j,r,s}\left(F^{i\bar{j},r\bar{s}}+F^{i\bar{s}}A^{\bar{j}r}\right)B_{i\bar{j}}\overline{B_{s\bar{r}}}\geq0.
\]
In particular, $F$ is a strictly convex function in $\Gamma_{n\times n}^{+}$. 
\end{cor}

\begin{proof}
Notice that 
\[
F^{i\bar{j}}=\sum_{k=1}^{n-1}F_{k}^{i\bar{j}}-\frac{f}{\det A}A^{\bar{j}i}
\]
Thus, (\ref{eq:monotonicity f non neg}) follows from Lemma \ref{lem:F_k monot}
and assumption \textbf{H2}.

For convexity, by Proposition \ref{prop: F_k convexity} 
\begin{align*}
\sum\left(F^{i\bar{j},r\bar{s}}+F^{i\bar{s}}A^{\bar{j}r}\right)B_{i\bar{j}}\overline{B_{s\bar{r}}} & =\sum_{k=1}^{n-1}\<\Lambda^{[k]},\Theta_{k}(B,B)\>+\frac{f}{\det A}A^{\bar{s}r}A^{\bar{j}i}B_{i\bar{j}}\overline{B_{s\bar{r}}}\\
 & \geq0.
\end{align*}
Here, we have used the lower bound for $f.$ By Lemma \ref{lem:F_k monot}
and assumption \textbf{H2}, $-F^{i\bar{s}}A^{\bar{j}r}B_{i\bar{j}}\overline{B_{s\bar{r}}}>0$
for non-zero $B$. Thus $F$ is strictly convex. 
\end{proof}

\section{Cone condition \label{sec:Cone-condition}}

We have introduced the concept of cone condition/subsolution in (\ref{eq:cone condition}).
In this section, we state more criteria for subsolutions and prove
some properties that will be used later. We will focus on the local
cone condition near a point $p\in M$. 
\begin{defn}
Let 
\[
\mathcal{C}_{\Lambda}^{\kappa}:=\{\omega:\omega\ \text{is\ }\kah;\ (\ref{eq:cone condition})\text{ holds}\}.
\]
At a point $p\in M$, let 
\[
\mathcal{C}_{\Lambda}^{\kappa}(p):=\{\omega:\omega\text{ is a positive }(1,1)\text{-form};\ (\ref{eq:cone condition})\text{ holds at }p\}.
\]
In a local coordinates at $p$, where $\omega=\frac{\sqrt{-1}}{2}\sum A_{i\bar{j}}dz^{i}\wedge d\bar{z}^{j}$,
we say the matrix $A\in\mathcal{C}_{\Lambda}^{\kappa}(p)$ if $\omega$
satisfies (\ref{eq:cone condition}) at $p$. By (\ref{eq:cone condition}),
$\mathcal{C}_{\Lambda}^{\kappa}(p)$ can be viewed as an open set
in $\Gamma_{n\times n}^{+}$.
\end{defn}

The cone condition (\ref{eq:cone condition}) in the study of $J$-equation
was first explored in Song-Weinkove \cite{song2008convergence}. Later
Fang-Lai-Ma \cite{Fang-Lai-Ma} extended the discussion to inverse
$\sigma_{k}$ type equations. The notion of subsolution for a class
of fully nonlinear equations was introduced by Guan \cite{Guan2014Second-order}.
See also Székelyhidi \cite{szekelyhidi2018fully}.

We have the following criterions for subsolutions. Here we do not
assume the sign of $\Lambda^{[n]}$. 
\begin{prop}
\label{prop:crite for Sub }Notations as above. Suppose that $\underline{A}$
is a positive Hermitian matrix. The followings are equivalent 
\begin{enumerate}
\item \label{enu:.TFAE}$\underline{A}\in\mathcal{C}_{\Lambda}^{\kappa}(p)$. 
\item \label{enu:TFAE 2}There is a constant $R=R(\underline{A},\Lambda,n,|f(p)|)$
s.t. if $B$ is a non-negative Hermitian matrix satisfying 
\[
F(\underline{A}+B)=\kappa,
\]
then 
\begin{equation}
|B|\le R.\label{eq:boundedness of B}
\end{equation}
Here $|B|=\left(\sum_{i,j}|B_{i\bar{j}}|^{2}\right)^{\frac{1}{2}}.$ 
\item \label{enu:TFAE 3}For any non-zero semi-positive Hermitian matrix
$B$ the following holds 
\begin{equation}
\lim_{t\to\infty}F(\underline{A}+tB)<\kappa.\label{eq:subsol cri eq}
\end{equation}
\end{enumerate}
\end{prop}

\begin{proof}
Pick a local unit covector $b=\sum_{i=1}^{n}b_{i}dz^{i}$ such that
$\|b\|_{\rho}=1.$ We may identify $b$ with an $1\times n$ matrix
$(b_{1},\cdots,b_{n})$. Define $\beta=\frac{\sqrt{-1}}{2}\sum_{i,j}b_{i}\bar{b}_{j}dz^{i}\wedge d\bar{z}^{j}$.
We have the following identity 
\begin{align}
\frac{\left(\omega+t\beta\right)^{n}}{n!}\left(\kappa-F(\underline{A}+tb^{\dagger}b)\right) & =t\left(\kappa\Omega-\Lambda\wedge\Omega\right)^{[n-1]}\wedge\beta+\left(\kappa-F(\underline{A})\right)\frac{\omega^{n}}{n!}.\label{eq:eqution in subsol}
\end{align}
$(\ref{enu:.TFAE})\Rightarrow(\ref{enu:TFAE 2})$. Suppose that $\underline{A}\in\mathcal{C}_{\Lambda}^{\kappa}$.
Then there is a positive $\delta=\delta(\underline{A},\Lambda)$ s.t.
\[
\left(\kappa\Omega-\Lambda\wedge\Omega\right)^{[n-1]}\wedge\beta\geq\delta P^{[n-1]}.
\]
Then from (\ref{eq:eqution in subsol}), 
\begin{equation}
\frac{\left(\omega+t\beta\right)^{n}}{n!}\left(\kappa-F(\underline{A}+tb^{\dagger}b)\right)\geq t\delta\|b\|_{\rho}^{2}P^{[n]}+O(1).\label{eq:inter eq}
\end{equation}
Here $O(1)$ represents a $(n,n)$-form with bounded norm with respect
to $\rho$, depending on $\underline{A}$ and $\Lambda$ but not $b$.
As 
\begin{equation}
F(\underline{A}+tb^{\dagger}b)=\sum_{k=1}^{n-1}F_{k}(\underline{A}+tb^{\dagger}b)+\frac{f(p)}{\det(\underline{A}+tb^{\dagger}b)},\label{eq:F(A+tb)}
\end{equation}
By Lemma \ref{lem:F_k monot}, the first term in the right hand side
of (\ref{eq:F(A+tb)}) is decreasing in $t$ and is non-negative.
The second term is bounded and approaches 0 when $t\to\infty$. Thus
$\lim_{t\to\infty}F(A+tb^{\dagger}b)$ exists. Hence from (\ref{eq:inter eq}),
\begin{align*}
t\delta P^{[n]}+O(1) & \leq\left(t\frac{\omega^{n-1}}{(n-1)!}\wedge\beta\right)\left(\kappa-F(\underline{A}+tb^{\dagger}b)\right)+O(1)\\
 & =t\|b\|_{\omega}\det\underline{A}\left(\kappa-F(\underline{A}+tb^{\dagger}b)\right)P^{[n]}+O(1)\\
 & \leq t\lambda_{max}\det\underline{A}\left(\kappa-F(\underline{A}+tb^{\dagger}b)\right)P^{[n]}+O(1).
\end{align*}
Here $\lambda_{max}$ is the maximal eigenvalue of $\underline{A}$.
Therefore, there is a $R=R(\underline{A},\Lambda)$ such that if $t>R$,
then 
\begin{equation}
\kappa-F(\underline{A}+tb^{\dagger}b)\geq\frac{\delta}{2\lambda_{max}\det\underline{A}}.\label{eq:quantitative sub}
\end{equation}

For any non-zero semi-positive Hermitian matrix $B$ with $\|B\|=\sqrt{\sum_{i,j}|B_{i\bar{j}}|^{2}}=N$,
let $\lambda$ be the biggest non-zero eigenvalue of $B$ and $b'=(b_{1}',\cdots,b_{n}')$
be the unit eigenvector corresponding to $\lambda$. Notice $\lambda\ge\frac{N}{n}$
and 
\[
B-\frac{N}{2n}(b')^{\dagger}b'\geq0.
\]
Suppose that $N/(2n)\geq R$. 
\begin{equation}
\kappa-F(\underline{A}+\frac{N}{2n}(b')^{\dagger}b')\geq\frac{\delta}{2\lambda_{max}\det\underline{A}}.\label{eq:lem cone-1}
\end{equation}
Now 
\begin{align}
F(\underline{A}+B) & =\sum_{k=1}^{n-1}F_{k}(\underline{A}+B)+\frac{f(p)}{\det(\underline{A}+B)}\label{eq:lem cone}\\
 & \leq F(\underline{A}+\frac{N}{2n}(b')^{\dagger}b')+|f(p)|\cdot\left|\frac{1}{\det(\underline{A}+B)}-\frac{1}{\det(\underline{A}+\frac{N}{2n}(b')^{\dagger}b')}\right|.\nonumber 
\end{align}
Take $N$ large enough so that $\det(\underline{A}+\frac{N}{2n}(b')^{\dagger}b')^{-1}<\frac{\delta}{4\lambda_{max}(\det\underline{A})(|f(p)|+1)}$.
Then by (\ref{eq:lem cone-1}) and (\ref{eq:lem cone}), 
\[
\kappa-F(\underline{A}+B)>\frac{\delta}{4\lambda_{max}\det\underline{A}}.
\]
We then have a contradiction. We have proved (\ref{eq:boundedness of B}).

$(\ref{enu:TFAE 2})\Rightarrow(\ref{enu:TFAE 3})$ is obvious since
$F_{k}(\underline{A}+tB)$ is strictly monotonic decreasing and $\frac{f}{\det(\underline{A}+tB)}$
tends to $0$ as $t\to\infty$.

$(\ref{enu:TFAE 3})\Rightarrow(\ref{enu:.TFAE})$. If $\lim_{t\to\infty}F(\underline{A}+tB)<\kappa$,
we can test each $B=(B_{i\bar{j}})=(\delta_{i\bar{j}})$ in (\ref{eq:eqution in subsol})
to see that each $\left(\kappa\Omega-\Lambda\wedge\Omega\right)^{[n-1]}\wedge(\iover dz^{i}\wedge d\bar{z}^{i})$
is positive. Thus, $\left(\kappa\Omega-\Lambda\wedge\Omega\right)^{[n-1]}$
is positive, which implies $\underline{A}\in\mathcal{C}_{\Lambda}^{\kappa}(p)$. 
\end{proof}
\begin{rem}
The criterion (2) in Proposition \ref{prop:crite for Sub } is the
definition of subsolution given by Székelyhidi \cite{szekelyhidi2018fully}.\textbf{
}The equivalence of (2) and (3) is suggested in \cite{szekelyhidi2018fully}
Remark 8. 
\end{rem}

\begin{defn}
Let $B$ be a non-zero semi-positive Hermitian matrix. We define
\begin{equation}
F_{\Lambda}(A:B):=\lim_{t\to\infty}F(A+tB,\Lambda),\label{eq:F(A:B)}
\end{equation}
\begin{equation}
\mathcal{P}_{\Lambda}(A):=\max_{B\in\overline{\Gamma_{n\times n}^{+}},\|B\|=1}\lim_{t\to+\infty}F(A+tB,\Lambda).\label{eq:P(A)}
\end{equation}
Clearly, $A\in\mathcal{C}_{\Lambda}^{\kappa}(p)$ if and only if for
any non-zero $B\in\overline{\Gamma_{n\times n}^{+}}$, $F_{\Lambda}(A:B)<\kappa$.
Equivalently, $A\in\mathcal{C}_{\Lambda}^{\kappa}(p)$ if and only
if 
\[
\mathcal{P}_{\Lambda}(A)<\kappa.
\]
\end{defn}

We use the notation $\mathcal{P}_{\Lambda}(\omega)=\mathcal{P}_{\La}(A)$
if $\omega$ is represented by matrix $A$ in a local coordinate.
There is another perspective where $F_{\Lambda}(A:B)$ and $\mathcal{P}_{\Lambda}(A)$
are given by the restrictions of $F$ on subspaces, which relies on
the the general inverse matrix in the sense of Moore-Penrose \cite{moore1920reciprocal,penrose1955generalized}. 
\begin{defn}
Suppose that $A$ is a $n\times n$ matrix. A Moore-Penrose inverse
matrix $A^{-1}$ is the unique matrix that satisfies the following
conditions: 
\begin{enumerate}
\item $AA^{-1}A=A$, $A^{-1}AA^{-1}=A^{-1}.$ 
\item Both $AA^{-1}$ and $A^{-1}A$ are Hermitian matrices. 
\end{enumerate}
\end{defn}

For any matrix, the Moore-Penrose inverse of $A$ exists and can be
constructed via singular value decomposition. The following lemma
\ref{lem:moore-penrose  inverse} will be used to relate $F_{\Lambda}(A:B)$
with a Moore-Penrose inverse. 
\begin{lem}
\label{lem:moore-penrose  inverse}Let $V\geq0$ be a Hermitian matrix
of rank $r$. Let $\mathcal{V}$ be the linear spaces spaced by the
non-zero eigenvectors of $V$ and let $\mathcal{H}$ be the orthogonal
complement of $\mathcal{V}$. Let $\Pi_{\mathcal{H}}$ be the orthogonal
projection matrix to the subspace $\mathcal{H}\subset\mathcal{T}_{p}M$.
Denote 
\begin{equation}
(A|_{\mathcal{H}})^{-1}:=\Pi_{\mathcal{H}}A\Pi_{\mathcal{H}}.\label{eq:general inverse AH}
\end{equation}
 Then
\begin{equation}
\lim_{t\to+\infty}\left(A+tV\right)^{-1}=\left(A|_{\mathcal{H}}\right)^{-1},\label{eq:add15}
\end{equation}
\end{lem}

\begin{rem}
The inverse on the right hand side of (\ref{eq:add15}) is in the
sense of Moore-Penrose and we denote it as $(A|_{\mathcal{H}})^{-1}$.
\end{rem}

\begin{proof}
Since Moore-Penrose inverse is invariant under unitary transform,
we may choose a unitary basis $\{e_{i}\}$ so that $e_{1},\cdots,e_{r}$
spans $\mathcal{V}$. Then 
\[
V=\sum_{i=1}^{r}V_{i\bar{i}}e_{i}e_{i}^{\dagger},\ A=\sum_{i,j}A_{i\bar{j}}e_{i}e_{j}^{\dagger}.
\]
Then, under this basis, we have 
\[
V=\left(\begin{array}{cc}
v & 0\\
0 & 0
\end{array}\right),\ \Pi_{\mathcal{H}}=\left(\begin{array}{cc}
0 & 0\\
0 & I_{n-r}
\end{array}\right)
\]
where $v$ is a $r\times r$ positive definite matrix and $I_{n-r}$
is an identity matrix. We may write
\[
A+tV=\left(\begin{array}{cc}
A_{1}+tv & C\\
C^{\dagger} & A_{2}
\end{array}\right).
\]
The inverse is given by
\[
(A+tV)^{-1}=\left(\begin{array}{cc}
(\hat{A}_{1}(t))^{-1} & -(\hat{A_{1}}(t))^{-1}CA_{2}^{-1}\\
-A_{2}^{-1}C^{\dagger}(\hat{A_{1}}(t))^{-1} & A_{2}^{-1}+A_{2}^{-1}C^{\dagger}(\hat{A_{1}}(t))^{-1}CA_{2}^{-1}
\end{array}\right)
\]
where $\hat{A}_{1}(t)=A_{1}+tv-CA_{2}^{-1}C^{\dagger}$. As $t\to+\infty$,
$\left(\hat{A}_{1}(t)\right)^{-1}\to0$ uniformly. Thus,
\[
\lim_{t\to+\infty}(A+tV)^{-1}=\left(\begin{array}{cc}
0 & 0\\
0 & A_{2}^{-1}
\end{array}\right).
\]
Notice that $\Pi_{\mathcal{H}}A\Pi_{\mathcal{H}}=\left(\begin{array}{cc}
0 & 0\\
0 & A_{2}
\end{array}\right)$. Hence, we easily see that $(\Pi_{\mathcal{H}}A\Pi_{\mathcal{H}})^{-1}=\lim_{t\to+\infty}(A+tV)^{-1}$.
\end{proof}
Given an orthogonal splitting $\mathcal{T}_{p}M=\mathcal{H}\oplus\mathcal{H}^{\perp}$
with respect to $\rho$. Following the notation in Lemma \ref{lem:moore-penrose  inverse},
we denote 
\begin{equation}
\chi_{\mathcal{H}}:=\sum_{i,j\geq d+1}(A|_{\mathcal{H}})^{\bar{j}i}2\sqrt{-1}\frac{\pdv}{\pdv\bar{z}^{j}}\wedge\frac{\pdv}{\pdv z^{i}}.\label{eq:chi_A chi_H}
\end{equation}
Then $\chi_{\mathcal{H}}\in\bigwedge^{1,1}\mathcal{H}$ is dual to
$\omega|_{\mathcal{H}}$. Now we characterize $\mathcal{P}_{\Lambda}$
using subspaces.
\begin{lem}
\label{lem: submatrix and restricting to subspac}Notations as above,
the following statements are true:
\begin{enumerate}
\item \label{enu:if-,-then}If $B\geq0$ and its all the non-zero eigenvectors
span $\mathcal{H}^{\perp}$, then 
\[
F_{\Lambda}(A:B)=\<\Lambda,\exp\chi_{\mathcal{H}}\>.
\]
\item \label{enu:} 
\[
\mathcal{P}_{\Lambda}(A)=\max_{\mathcal{H}\subset\mathcal{T}_{p}M:\dim\mathcal{H}=d\leq n-1}\<\Lambda,\exp\chi_{\mathcal{H}}\>.
\]
\item \label{enu:-1}
\[
\mathcal{P}_{\Lambda}(A)=\max_{\mathcal{H}\subset\mathcal{T}_{p}M:\dim\mathcal{H}=n-1}\<\Lambda,\exp\chi_{\mathcal{H}}\>.
\]
\item \label{enu:for-any--dimensional}If $\omega\in\mathcal{C}_{\Lambda}^{\kappa}(p)$,
then for any $d$-dimensional subspace $\mathcal{H}$ of $\mathcal{T}_{p}M$
with $d\leq n-1$, we have 
\[
\left((\kappa-\Lambda)\wedge\exp\omega\right)^{[d]}|_{\mathcal{H}}>0,
\]
as a positive $(d,d)$-form.
\end{enumerate}
\end{lem}

\begin{proof}
(\ref{enu:if-,-then}) follows from the definition of $F(A:B)$ and
Lemma \ref{lem:moore-penrose  inverse}.

(\ref{enu:}) follows immediately from (\ref{enu:if-,-then}).

(\ref{enu:-1}) follows from the monotonicity of $F$. In fact, $F(A)$
is monotonic in $A$, which implies $\lim_{t\to+\infty}F(A+tB,\Lambda)\leq\lim_{t\to\infty}F(A+tb^{\dagger}b)$
where $b$ is a unit eigenvector of $B$ with positive eigenvalue.
Thus, we may restrict to rank one Hermitian matrix when computing
$\mathcal{P}_{\Lambda}(A)$ in (\ref{eq:P(A)}).

(\ref{enu:for-any--dimensional}) follows from (\ref{enu:}) immediately.
\end{proof}
The following corollary gives the easy part of Theorem \ref{thm:numerical criterion}.
\begin{cor}
\label{cor:converse}Notations as above. If $\omega\in\mathcal{C}_{\Lambda}^{\kappa}$,
for any $d$-dimensional subvariety $Z\subset M$ and $d\leq n-1$,
we have 
\begin{equation}
\int_{Z}\left(\kappa-\Lambda\right)\wedge\exp\omega>0.\label{eq:converse to numeric criterion}
\end{equation}
Furthermore, if $\omega\in[\omega_{0}]$ solves (\ref{eq:equation with f}),
then $[\omega_{0}]$ is $([\Lambda],\kappa)$-positive. 
\end{cor}

\begin{proof}
For any $d$-dimensional subvariety $Z\subset M$ with its regular
component as $Z_{reg},$ $Z\backslash Z_{\text{reg}}$ is of zero
measure. Then by Lemma \ref{lem: submatrix and restricting to subspac}
(\ref{enu:for-any--dimensional}) $\left(\left(\kappa-\Lambda|_{Z}\right)\wedge\exp\left(\omega|_{Z}\right)\right)^{[d]}>0$
for any point $p\in Z_{\text{reg}}$. Hence (\ref{eq:converse to numeric criterion})
is established.

We will prove that any solution $\omega$ to (\ref{eq:equation with f})
is a subsolution in Lemma \ref{lem:Datar-Pingali}. Thus, by (\ref{eq:converse to numeric criterion}),
we conclude that $[\omega]=[\omega_{0}]$ is $([\Lambda],\kappa)$-positive. 
\end{proof}
For future use, we collect some properties of \emph{$\mathcal{P}_{\Lambda}(A)$
in the following lemma.} 
\begin{lem}
Notations as above. \label{lem:P_La property}The following properties
of $\mathcal{P}_{\Lambda}(A)$ hold: 
\begin{enumerate}
\item \label{enu:-is-a}$\mathcal{P}_{\Lambda}(A)$ is a continuous convex
function in $\Gamma_{n\times n}^{+}$. 
\item \label{enu:-is-decreasing}$\mathcal{P}_{\Lambda}(A)$ is decreasing
in $A$, i.e. $\mathcal{P}_{\Lambda}(A')\leq\mathcal{P}_{\Lambda}(A)$
if $A'-A\geq0$. 
\item \label{enu:-is-continuous}$\mathcal{P}_{\Lambda}(A)$ is continuous
in $\La$. 
\item \label{enu:-is-increasing}$\mathcal{P}_{\Lambda}(A)$ is increasing
in $\Lambda$, i.e. $\mathcal{P}_{\Lambda}(A)\leq\mathcal{P}_{\Lambda'}(A)$
if $\Lambda'-\Lambda\geq0.$ 
\item \emph{\label{enu:-as-a}$\mathcal{P}_{\Lambda}(\omega)$ as a function
of $p$ on $M$ is continuous.} 
\item \label{enu:-is-sublinear}$\mathcal{P}_{\Lambda}(A)$ is sub-linear
in $\Lambda$, i.e. $\mathcal{P}_{\Lambda+\Lambda'}(A)\leq\mathcal{P}_{\Lambda}(A)+\mathcal{P}_{\Lambda'}(A)$. 
\item \label{enu:-if} $\mathcal{P}_{\Lambda}(A)=\mathcal{P}_{\mathring{\Lambda}}(A)$,
for $\mathring{\Lambda}=\Lambda-\Lambda^{[n]}$, 
\item \emph{\label{enu:If--and}Suppose that $\Lambda$ satisfies }\textbf{H2}\emph{.
If $\{A_{l}\}\subset\Gamma_{n\times n}^{+}$ and $A_{l}\to A_{\infty}\in\pdv\Gamma_{n\times n}^{+}$
as $l\to\infty$, then $\lim_{l\to\infty}\mathcal{P}_{\Lambda}(A_{l})=\infty.$} 
\end{enumerate}
\end{lem}

\begin{proof}
(\ref{enu:-is-a}) follows from the convexity of $F$ on $\Gamma_{n\times n}^{+}$.
(\ref{enu:-is-decreasing}) follows from the monotonicity of $F$.
See Corollary \ref{cor: convexity for f bigger than 0}.

For (\ref{enu:-is-continuous}), let $\Lambda$ and $\Lambda'$ be
two differential forms. We may choose $\epsilon^{n}=n!\|\Lambda-\Lambda'\|_{\rho}$
so that 
\[
-\exp\epsilon\rho+1\leq\Lambda-\Lambda'\leq\exp\epsilon\rho-1.
\]
By Lemma \ref{lem: submatrix and restricting to subspac} (\ref{enu:-1}),
we have 
\begin{align}
|\mathcal{P}_{\Lambda}(A)-\mathcal{P}_{\Lambda'}(A)| & \leq\max_{\mathcal{H}\subset T^{1,0}M:\dim\mathcal{H}=n-1}\left|\<\Lambda-\Lambda',\exp\chi_{\mathcal{H}}\>\right|\label{eq:Difference P continuity}\\
 & \leq\max_{\mathcal{H}\subset T^{1,0}M:\dim\mathcal{H}=n-1}\<\exp\left(\epsilon\rho\right)-1,\exp\left(\chi_{\mathcal{H}}\right)\>\nonumber \\
 & \leq\<\exp\left(\epsilon\rho\right)-1,\exp\chi\>.\nonumber 
\end{align}
Then, as $\|\Lambda-\Lambda'\|_{\rho}\to0$, $\epsilon\to0$, it holds
that $|\mathcal{P}_{\Lambda}(A)-\mathcal{P}_{\Lambda'}(A)|\to0$ as
well. Hence, $\mathcal{P}_{\Lambda}(A)$ is continuous in $\Lambda$.

(\ref{enu:-is-increasing}) follows from Lemma \ref{lem: submatrix and restricting to subspac}
(\ref{enu:if-,-then}). (\ref{enu:-as-a}) follows from (\ref{enu:-is-a})
and (\ref{enu:-is-continuous}).

For (\ref{enu:-is-sublinear}), we have the following 
\begin{align*}
\mathcal{P}_{\Lambda+\Lambda'}(A) & =\max_{\mathcal{H}\subset T^{1,0}M:\dim\mathcal{H}=n-1}\<\Lambda+\Lambda',\exp\chi_{\mathcal{H}}\>\\
 & \leq\mathcal{P}_{\Lambda}(A)+\mathcal{P}_{\Lambda'}(A).
\end{align*}

For (\ref{enu:-if}), we assume that $\Lambda^{[n]}(p)=f(p)P^{[n]}$.
Then (\ref{enu:-if}) follows from the fact that $\lim_{t\to\infty}\frac{f(p)}{\det(A+tB)}=0$
for any $B\geq0$ and $B\not=0$.

For (\ref{enu:If--and}), note we may assume that $\Lambda\geq0$
since $\mathcal{P}_{\mathring{\Lambda}}(A)=\mathcal{P}_{\Lambda}$
if $\mathring{\Lambda}=\Lambda-\Lambda^{[n]}$ by (\ref{enu:-if}).
Let $\chi_{A_{l}}=(A_{l})^{\bar{ji}}2\sqrt{-1}\frac{\pdv}{\pdv\bar{z}^{j}}\wedge\frac{\pdv}{\pdv z^{i}}$.
We first prove that 
\begin{equation}
\lim_{l\to\infty}\<\Lambda,\exp\chi_{A_{l}}\>=\infty\label{eq:enu if and}
\end{equation}
as $A_{l}\to A_{\infty}\in\pdv\Gamma_{n\times n}^{+}$. By \textbf{H2},
\begin{align}
\<\Lambda,\exp\chi_{A_{i}}\> & \geq m\<\sum_{i=1}^{n_{p}}\frac{\rho_{i}^{k_{i}}}{k_{i}!},\exp\chi_{A_{l}}\>\label{eq:to boundary}\\
 & \geq m\sum_{i=1}^{n_{p}}\sigma_{k_{i}}(A_{l}^{-1}|_{\mathcal{V}_{i}})\nonumber \\
 & \geq mC(n,\mathbf{d}_{p},\mathbf{k}_{p})\det(A_{l}^{-1})^{\frac{1}{\sum_{i=1}^{n_{p}}\frac{d_{i}}{k_{i}}}}.\nonumber 
\end{align}
We have used Newton-Maclaurin inequality and the mean value inequality
in the last inequality. (\ref{eq:enu if and}) follows from the fact
that $\lim_{l\to\infty}\det(A_{l}^{-1})=\infty.$ To prove (\ref{enu:If--and}),
we choose a hyperplane $\mathcal{H}$ which contains the kernel of
$A_{\infty}$. Then, Lemma \ref{lem: submatrix and restricting to subspac}
(\ref{enu:-1}) and (\ref{eq:enu if and}) implies (\ref{enu:If--and}). 
\end{proof}

\section{Equations with almost positive volume forms\label{sec:Equations-with-negative}}

With the cone condition in place, we extend results of Section \ref{sec:Ellipticity-and-convexity}
to include cases where $\Lambda$ contains an almost positive volume
form. Many ideas are from works of Chen \cite{chen2021j} and Datar-Pingali
\cite{datar2021numerical}. For simplicity, we write 
\begin{equation}
\Lambda^{[n]}=fP^{[n]}.\label{eq:fP}
\end{equation}
In this section, we allow $f$ to be almost positive.

The following lemma implies that a solution to (\ref{eq:equation with f})
is also a subsolution if \textbf{H2 }is satisfied. 
\begin{lem}
\label{lem:Datar-Pingali}Assume that $\mathring{\Lambda}$ is $\mathcal{O}$-uniformly
positive with constant $m$.\textbf{ }Suppose that for some $p_{0}\in M$,
$f(p_{0})\geq0$. For any $p\in M$, let 
\begin{equation}
\gamma_{\min}(\frac{\kappa}{m},n_{p},\mathbf{d}_{p},\mathbf{k}_{p}):=\frac{\min_{i}\left\{ \binom{d_{i}}{k_{i}-1}\binom{d_{i}}{k_{i}}^{\frac{1}{k_{i}}-1}\right\} \prod_{i=1}^{n_{p}}\binom{d_{i}}{k_{i}}^{\frac{d_{i}}{k_{i}}}}{\max_{i}\{d_{i}\binom{d_{i}}{k_{i}}^{\frac{1}{k_{i}}}\}(\frac{\kappa}{m})^{\sum_{i=1}^{n_{p}}\frac{d_{i}}{k_{i}}}\cdot\left(\sum_{i=1}^{n_{p}}\left(\frac{\kappa}{m}\right)^{\frac{1}{k_{i}}}\right)}.\label{eq:gamm_min}
\end{equation}
If $\omega$ is a solution to (\ref{eq:equation with f}), and for
all $p\in M$, $f(p)>-m\gamma_{\min}(\frac{\kappa}{m},n_{p},\mathbf{d}_{p},\mathbf{k}_{p})$,
then $\omega\in\mathcal{C}_{\Lambda}^{\kappa}$. 
\end{lem}

We first state some lemmas before proving Lemma \ref{lem:Datar-Pingali}. 
\begin{lem}
\label{lem: quotient k-hessian-1}Notations as above. Let $A$ be
a positive Hermitian matrix. Then the following statements hold. 
\begin{enumerate}
\item If $k\geq l\geq0$, $r\geq s\geq0$, the function 
\begin{equation}
A\mapsto\frac{\sigma_{l}(A)}{\sigma_{k}(A)}\label{eq:decreasing sigma k}
\end{equation}
is decreasing in $A$. 
\item Let $\mathcal{H}$ be a linear subspace with codimension $r$. Then
\begin{equation}
\frac{\sigma_{l-r}(A|_{\mathcal{H}})}{\sigma_{k-r}(A|_{\mathcal{H}})}\leq\frac{\sigma_{l}(A)}{\sigma_{k}(A)}.\label{eq:decreasing}
\end{equation}
\item Let $T_{k-1}(A)$ be the linearized operator of $\sigma_{k}(A)$,
i.e. $\<T_{k-1}(A),B\>=\frac{d}{dt}\sigma_{k}(A+tB)|_{t=0}$. Then
\begin{equation}
\sigma_{r-1}(A)T_{k-1}(A)\geq\sigma_{k-1}(A)T_{r-1}(A),\label{eq:lem T_k}
\end{equation}
if $r\geq k$. 
\end{enumerate}
\end{lem}

\begin{proof}
(1) is well known in the literature of quotient Hessian equations.
See for instance, \cite{spruck2005geometric}.

To prove (2), we may assume that $\mathcal{H}=\text{span}\{\frac{\pdv}{\pdv z^{1}},\cdots,\frac{\pdv}{\pdv z^{n-r}}\}$.
We define 
\begin{equation}
A_{t}=A+t\left(\begin{array}{cc}
0 & 0\\
0 & I_{r}
\end{array}\right)=\left(\begin{array}{cc}
A|_{\mathcal{H}} & C\\
C^{\dagger} & A|_{\mathcal{H}^{\perp}}+tI_{r}
\end{array}\right).\label{eq:A_t}
\end{equation}
Then $A_{t}\geq A$ and 
\begin{equation}
\sigma_{l}(A_{t})=t^{r}\sigma_{l-r}(A|_{\mathcal{H}})+o(t^{r}),\ \sigma_{k}(A_{t})=t^{r}\sigma_{k-r}(A|_{\mathcal{H}})+o(t^{r}).\label{eq:sigma A_t}
\end{equation}
Hence, 
\begin{equation}
\frac{\sigma_{l}(A)}{\sigma_{k}(A)}\ge\lim_{t\to\infty}\frac{\sigma_{l}(A_{t})}{\sigma_{k}(A_{t})}=\frac{\sigma_{l-r}(A|_{\mathcal{H}})}{\sigma_{k-r}(A|_{\mathcal{H}})}.\label{eq:qut simga A}
\end{equation}

To prove (3), we may assume that $A$ is diagonal after a unitary
transform. Write $A=\text{diag}\{\lambda_{1},\cdots,\lambda_{n}\}.$Then
\begin{equation}
\left(T_{k-1}(A)\right)^{i\bar{j}}=\sigma_{k-1}(A|i)\delta^{i\bar{j}}.\label{eq:Tij}
\end{equation}
Here, $(A|i)$ denote a matrix obtained by deleting $i$-th row and
$i$-th column of $A$. Thus, (\ref{eq:lem T_k}) is equivalent to
\begin{equation}
\sigma_{r-1}(A)\sigma_{k-1}(A|i)\geq\sigma_{k-1}(A)\sigma_{r-1}(A|i),\label{eq:T_k 1}
\end{equation}
for each $i$. If $\sigma_{r-1}(A|i)=0$, then (\ref{eq:T_k 1}) holds
trivially. Otherwise (\ref{eq:T_k 1}) is equivalent to 
\begin{equation}
\frac{\sigma_{k-1}(A|i)}{\sigma_{r-1}(A|i)}\geq\frac{\sigma_{k-1}(A)}{\sigma_{r-1}(A)}.\label{eq:T_k 2}
\end{equation}
By 1), $\frac{\sigma_{k-1}(A)}{\sigma_{r-1}(A)}$ is decreasing in
$A$. Thus (\ref{eq:T_k 2}) holds since $(A|i)\leq A$. 
\end{proof}
With a labeled orthogonal splitting $\mathcal{O}_{p}=\{n_{p},\mathbf{d}_{p},\{\mathcal{V}_{i}\},\mathbf{k}_{p}\}$
at $p$, we fix a normal coordinate $\{z^{i}\}$ at $p$ of $\rho$
such that $\{\sqrt{2}\frac{\pdv}{\pdv z^{i}}\}$ restricts to a unitary
frame on each $\mathcal{V}_{i}$. For any non-negative Hermitian matrix
$A$, we may write under this frame 
\begin{equation}
A=\left(\begin{array}{cccc}
A_{1} & * & * & *\\*
* & A_{2} & * & *\\*
* & * & \ddots & *\\*
* & * & * & A_{n_{p}}
\end{array}\right),\label{eq:A form}
\end{equation}
where $A_{i}\in\Gamma_{d_{i}\times d_{i}}$ . 
\begin{lem}
\label{lem:spplict 2}Notations as above. For any non-negative $A$,
Let 
\begin{equation}
A'=\left(\begin{array}{cccc}
A_{1} & 0 & 0 & 0\\
0 & A_{2} & 0 & 0\\
0 & 0 & \ddots & 0\\
0 & 0 & 0 & A_{n_{p}}
\end{array}\right).\label{eq:A' form}
\end{equation}
Let $\omega'=(A')_{i\bar{j}}\frac{\sqrt{-1}}{2}dz^{i}\wedge d\bar{z}^{j}.$
Then 
\begin{enumerate}
\item 
\[
\det(A)\leq\det(A')=\prod_{i=1}^{n_{p}}\det A_{i};
\]
\item 
\begin{align}
\sigma_{k}(A) & \leq\sigma_{k}(A')=\sum_{l\in\boldsymbol{l}_{k}}\prod_{j=1}^{n_{p}}\sigma_{l_{j}}(A_{j}),\label{eq:spplitct 2 2}
\end{align}
where $\boldsymbol{l}_{k}=\{(l_{1},\cdots,l_{n_{p}})\in\N:\sum_{j=1}^{n_{p}}l_{j}=k\}$; 
\item 
\begin{align}
T_{k-1}(A) & \leq2^{(k-1)(n_{p}-1)}T_{k-1}(A')=2^{(k-1)(n_{p}-1)}\text{diag}(T_{1},\cdots,T_{n_{p}}),\label{eq:spllict 3}
\end{align}
where 
\begin{equation}
T_{i}=\sum_{l\in\boldsymbol{l}_{k,i}}\prod_{j\not=i}\sigma_{l_{j}}(A_{j})T_{l_{i}-1}(A_{i}),\label{eq:spllit 4}
\end{equation}
and $\boldsymbol{l}_{k,i}=\{(l_{1},\cdots,l_{n_{p}})\in\N^{n_{p}}:\sum_{j=1}^{n_{p}}l_{j}=k,\ l_{i}\geq1\}$. 
\end{enumerate}
\end{lem}

\begin{proof}
For claim (1), we first prove the case when $n_{p}=2$. We may write
\[
A=\left(\begin{array}{cc}
A_{1} & C\\
C^{\dagger} & A_{2}
\end{array}\right).
\]
Therefore $\det(A)=\det\left(A_{1}-C^{\dagger}A_{2}^{-1}C\right)\det A_{2}$
if $A_{2}$ is invertible. If $A_{2}$ is not invertible, we replace
$A$with $A_{\epsilon}=A+\epsilon I$ and let $\epsilon\to0$. Then,
(1) follows from the continuity of the determinant function. The general
case follows from induction.

For claim (2), for any $1\leq i_{1}<\cdots<i_{k}\leq n$, by (1),
we have 
\begin{align}
\sigma_{k}(A) & =\sum_{1\leq i_{1}<i_{2}<\cdots<i_{k}\leq n}A\left(\begin{array}{cccc}
i_{1} & i_{2} & \cdots & i_{k}\\
i_{1} & i_{2} & \cdots & i_{k}
\end{array}\right)\nonumber \\
 & \leq\sum_{1\leq i_{1}<i_{2}<\cdots<i_{k}\leq n}A'\left(\begin{array}{cccc}
i_{1} & i_{2} & \cdots & i_{k}\\
i_{1} & i_{2} & \cdots & i_{k}
\end{array}\right)\label{eq:sigmk A}\\
 & =\sigma_{k}(A').\nonumber 
\end{align}
Thus the inequality in (\ref{eq:spplitct 2 2}) holds. The equality
in (\ref{eq:spplitct 2 2}) can be calculated directly.

For claim (3), if $n_{p}=2$, we have 
\begin{equation}
2A'-A=\left(\begin{array}{cc}
2A_{1} & 0\\
0 & 2A_{2}
\end{array}\right)-\left(\begin{array}{cc}
A_{1} & C\\
C^{\dagger} & A_{2}
\end{array}\right)=\left(\begin{array}{cc}
A_{1} & -C\\
-C^{\dagger} & A_{2}
\end{array}\right)\ge0.\label{eq:nonnega diff}
\end{equation}
If $n_{p}\geq2$, by iteration, we have 
\begin{equation}
A\leq2^{n_{p}-1}A',\ \omega\leq_{s}2^{n_{p}-1}\omega',\ \frac{\omega^{k}}{k!}\leq_{s}2^{k(n_{p}-1)}\frac{\left(\omega'\right)^{k}}{k!}.\label{eq:-4}
\end{equation}
Now, for any \textbf{$b=b_{i}dz^{i}$}, by (\ref{eq:-4}), we have
\begin{align}
\left(T_{k-1}(A)\right)^{i\bar{j}}b_{i}\bar{b}_{j}\frac{\rho^{n}}{n!} & =\frac{\rho^{n-k}}{(n-k)!}\wedge\frac{\omega^{k-1}}{(k-1)!}\wedge\frac{\sqrt{-1}}{2}b\wedge\bar{b}\label{eq:-5-1}\\
 & \leq2^{(k-1)(n_{p}-1)}\frac{\rho^{n-k}}{(n-k)!}\wedge\frac{(\omega')^{k-1}}{(k-1)!}\wedge\frac{\sqrt{-1}}{2}b\wedge\bar{b}\nonumber \\
 & =2^{(k-1)(n_{p}-1)}T_{k-1}(A')^{i\bar{j}}b_{i}\bar{b}_{j}\frac{\rho^{n}}{n!}.\nonumber 
\end{align}
Thus, by (\ref{eq:-5-1}), 
\[
T_{k-1}(A)\leq2^{(k-1)(n_{p}-1)}T_{k-1}(A').
\]
(\ref{eq:spllit 4}) can be verified by direct calculation. 
\end{proof}
\begin{lem}
\label{lem: induced cone}Suppose $\mathring{\Lambda}$ satisfies
the $\mathcal{O}$-UP condition. Let $b$ be a covector in $\mathcal{T}_{p}^{*}M$.
Let $\mathcal{B}=\{\zeta\in\mathcal{T}_{p}M:\<b,\zeta\>=0\}$ be a
complex hyperplane. Denote 
\begin{equation}
\chi_{\mathcal{B}}:=(A|_{\mathcal{B}})^{\bar{j}i}2\sqrt{-1}\frac{\pdv}{\pdv\bar{z}^{j}}\wedge\frac{\pdv}{\pdv z^{i}}.\label{eq:chi_B}
\end{equation}
If at $p$, $\mathcal{P}_{\Lambda}(A)\leq\kappa$, then we have 
\begin{equation}
\sum_{i=1}^{n_{p}}\<\frac{\rho_{i}^{k_{i}}}{k_{i}!},\exp\chi_{\mathcal{B}}\>\leq\frac{\kappa}{m}.\label{eq:induced cone h1}
\end{equation}
Furthermore, 
\begin{equation}
\sum_{i=1}^{n_{p}}\<\frac{\rho_{i}^{k_{i}}}{k_{i}!},\exp\chi\>\leq\frac{n\kappa}{m}.\label{eq:induced cone h2}
\end{equation}
\end{lem}

\begin{proof}
Let $\Lambda'=\sum_{i=1}^{n_{p}}(\exp\rho_{i})^{[k_{i}]}$. By the
$\mathcal{O}$-UP condition, at $p$, 
\[
\mathring{\Lambda}\geq m\Lambda'.
\]
As a result, we have 
\begin{equation}
\mathcal{P}_{\Lambda'}(A)\leq\frac{\kappa}{m}.\label{eq:coone condi}
\end{equation}
Let $B_{i\bar{j}}=b_{i}\bar{b}_{j}$. Then, 
\begin{equation}
F_{\Lambda'}(A:B)\leq\mathcal{P}_{\Lambda'}(A)\le\frac{\kappa}{m}.\label{eq:cone reduced form}
\end{equation}
By Lemma \ref{lem: submatrix and restricting to subspac}, (\ref{eq:cone reduced form})
implies (\ref{eq:induced cone h1}).

Choose a normal coordinate of $\omega$ at $p$. Let $\mathcal{B}_{i}=\{\xi:dz^{i}(\xi)=0\}$.
Then 
\[
\chi=\sum_{i=1}^{n}\frac{\sqrt{-1}}{2}dz^{i}\wedge d\bar{z}^{i},\ \chi_{\mathcal{B}_{j}}=\sum_{i\not=j}\frac{\sqrt{-1}}{2}dz^{i}\wedge d\bar{z}^{i}.
\]
Then, for $k\leq n-1$, 
\[
\frac{\chi^{k}}{k!}\leq_{s}\sum_{j=1}^{n}\frac{\chi_{\mathcal{B}_{j}}^{k}}{k!}.
\]
Thus, 
\begin{equation}
\sum_{i=1}^{n_{p}}\<\frac{\rho_{i}^{k_{i}}}{k_{i}!},\exp\chi\>\leq\sum_{j=1}^{n}\sum_{i=1}^{n_{p}}\<\frac{\rho_{i}^{k_{i}}}{k_{i}!},\exp\chi_{\mathcal{B}_{j}}\>\leq\frac{n\kappa}{m}.\label{eq:requred etech}
\end{equation}
We have proved \ref{eq:induced cone h2}. 
\end{proof}
The following lemma gives an explicit estimate of $-F^{i\bar{j}}$
when the cone condition holds. 
\begin{lem}
\label{lem:H1''' inequality}Suppose $\mathring{\Lambda}$ satisfies
$\mathcal{O}$-UP and for all $p\in M$, $f(p)>-m\gamma_{\min}(\frac{\kappa}{m},n_{p},\mathbf{d}_{p},\mathbf{k}_{p})$.
Let $b$ be a covector in $\mathcal{T}_{p}^{*}M$. Let $\mathcal{B}$
and $\chi_{\mathcal{B}}$ be given as in Lemma \ref{lem: induced cone}.
Let $\xi\in\mathcal{T}_{p}M$. If at $p$, $\mathcal{P}_{\Lambda}(A)\leq\kappa$
, then we have 
\begin{equation}
\<\mathring{\Lambda},2\sqrt{-1}\xi\wedge\bar{\xi}\wedge\exp\chi_{\mathcal{B}}\>\geq m\gamma_{\min}(\frac{\kappa}{m},n_{p},\mathbf{d}_{p},\mathbf{k}_{p})\frac{|\<\xi,b\>|^{2}}{\left(\det A\right)\|b\|_{\omega}^{2}}.\label{eq:lemma prod sp}
\end{equation}
\end{lem}

\begin{proof}
If $\<b,\xi\>=0$ then (\ref{eq:lemma prod sp}) holds trivially.
Otherwise, we may assume by rescaling that $\|b\|_{\rho}^{2}=2$ and
$|\<b,\xi\>|=1$. We may choose a normal coordinate $\{z^{i}\}$ of
$\rho$ at $p$ such that 
\begin{enumerate}
\item $dz^{1}|_{p}=b$; 
\item $\{\sqrt{2}e_{j}^{i}\}_{j=1}^{d_{i}}$ is an orthornormal frame of
$\mathcal{V}_{i}$ with respect to $\rho_{i}$ ; 
\item $\frac{\pdv}{\pdv z^{\sum_{l=1}^{i-1}(d_{l}-1)+n_{p}+j}}=e_{j+1}^{i}$
for $j=1,\cdots,d_{i}-1$; 
\item $e_{1}^{i}=\sum_{j=1}^{n_{p}}\alpha^{ij}\frac{\pdv}{\pdv z^{j}}$
and $\alpha^{-1}=(\alpha^{ij})$ is a unitary matrix of dimension
$n_{p}$. 
\end{enumerate}
The construction of the coordinate can be done as follows: First,
we construct the unitary frame. Take $dz^{1}|_{p}=b$. In each $\mathcal{V}_{i}$,
let $\tilde{e}_{1}^{i}=(\pi_{i})_{*}\frac{\pdv}{\pdv z^{1}}|_{p}$.
If $\tilde{e}_{1}^{i}=0$, we pick any norm $\sqrt{1/2}$ vector in
$\mathcal{V}_{i}$ to be $e_{1}^{i}$; if $\tilde{e}_{1}^{i}\not=0$,
let $e_{1}^{i}=\frac{\tilde{e}_{1}^{i}}{\sqrt{2}\|\tilde{e}_{1}^{i}\|_{\rho}}$.
We then pick the rest vectors so that $\{\sqrt{2}e_{j}^{i}\}_{j=1}^{d_{i}}$
is a unitary frame of $\mathcal{V}_{i}$ with respect to $\rho_{i}$.
At $p$, we choose $\frac{\pdv}{\pdv z^{2}}|_{p},\cdots,\frac{\pdv}{\pdv z^{n_{p}}}|_{p}$
together with $\frac{\pdv}{\pdv z^{1}}|_{p}$ to span the space $\text{span}\{e_{1}^{i}\}_{i=1}^{n_{p}}$.
Let $\frac{\pdv}{\pdv z^{\sum_{l=1}^{i-1}(d_{l}-1)+n_{p}+j}}|_{p}=e_{j+1}^{i}$
for $j=1,\cdots,d_{i}-1$. Then we extend $\{z^{i}\}$ to be a normal
coordinate.

By our choice of coordinate, $\mathcal{B}=\text{span}\{\frac{\pdv}{\pdv z^{2}},\cdots,\frac{\pdv}{\pdv z^{n}}\}$.
Use $(A|1)$ to denote the matrix obtained by deleting 1st row and
1st column of $A$. We may write

\begin{equation}
\chi=2\sqrt{-1}\sum_{i,j}A^{\bar{j}i}\frac{\pdv}{\pdv\bar{z}^{j}}\wedge\frac{\pdv}{\pdv z^{i}},\ \chi_{\mathcal{B}}=2\sqrt{-1}\sum_{i,j>1}(A|1)^{\bar{j}i}\frac{\pdv}{\pdv\bar{z}^{j}}\wedge\frac{\pdv}{\pdv z^{i}},\label{eq:def chi}
\end{equation}
where $((A|1)^{\bar{j}i})$ is the inverse matrix of $(A|1)$. Let
$\{\check{e}_{k}^{j}\}\subset\bigwedge^{1,0}T_{p}^{*}M$ be the dual
frame of $\{e_{j}^{k}\}$, $k=1,2,\cdots,n_{p}$. Therefore, we have
the following decompositions: 
\[
A=\left(\begin{array}{cc}
a & q\\
q^{\dagger} & (A|1)
\end{array}\right),A^{-1}=\left(\begin{array}{cc}
\hat{a}^{-1} & \hat{q}\\
\hat{q}^{\dagger} & (A|1)^{-1}+\hat{q}^{\dagger}\hat{a}\hat{q}
\end{array}\right),
\]
where $a=A_{1\bar{1}}$, 
\begin{equation}
\hat{a}=a-q(A|1)^{-1}q^{\dagger}=\frac{1}{A^{\bar{1}1}}=\frac{1}{\|b\|_{\omega}^{2}},\label{eq:def hat a}
\end{equation}
and $\hat{q}=-\hat{a}^{-1}q(A|1)^{-1}$. We may write 
\begin{equation}
(A|1)^{-1}=\left(\begin{array}{ccccc}
V & Q_{01} & Q_{02} & \ldots & Q_{1n_{p}}\\
Q_{01}^{\dagger} & \hat{A}_{1}^{-1} & Q_{12} & \ldots & Q_{1n_{p}}\\
Q_{02}^{\dagger} & Q_{12}^{\dagger} & \hat{A}_{2}^{-1} & \cdots & Q_{2n_{p}}\\
\vdots & \vdots & \vdots & \ddots & \vdots\\
Q_{0n_{p}}^{\dagger} & Q_{1n_{p}}^{\dagger} & Q_{2n_{p}}^{\dagger} & \ldots & \hat{A}_{n_{p}}^{-1}
\end{array}\right),\label{eq:A reistrcit}
\end{equation}
where $V$ is a matrix of dimension $n_{p}-1$. For $l=1,2,\cdots,d_{1}$,
we have 
\begin{align}
\frac{\rho_{1}^{l}}{l!} & =\sum_{|I|=l,I\subset\{n_{p}+1,\cdots,n_{p}+d_{1}-1\}}\frac{\sqrt{-1}^{l^{2}}}{2^{l}}dz^{I}\wedge d\bar{z}^{I}\label{eq:expansio of rho}\\
 & +\frac{\sqrt{-1}}{2}\check{e}_{1}^{1}\wedge\overline{\check{e}}_{1}^{1}\wedge\sum_{|I|=l-1,I\subset\{n_{p}+1,\cdots,n_{p}+d_{1}-1\}}\frac{\sqrt{-1}^{(l-1)^{2}}}{2^{l-1}}dz^{I}\wedge d\bar{z}^{I}.\nonumber 
\end{align}
Let $\alpha_{ij}=\overline{\alpha^{ji}}$. Define $\alpha_{i}':=(\alpha_{2i},\cdots,\alpha_{n_{p}i})$,
and 
\begin{equation}
\tilde{A}_{i}^{-1}:=\left(\begin{array}{ccc}
0 & 0 & 0\\
0 & \left(\alpha_{i}'\right)^{\dagger}V\alpha_{i}' & (\alpha_{i}')^{\dagger}Q_{0i}\\
0 & Q_{0i}^{\dagger}\alpha_{i}' & \hat{A}_{i}^{-1}
\end{array}\right).\label{eq:tidld A_i}
\end{equation}
Notice 
\begin{equation}
\check{e}_{i}^{1}\wedge\overline{\check{e}}_{i}^{1}=\sum_{l,j=2}^{n_{p}}\alpha_{li}\overline{\alpha_{ji}}dz^{l}\wedge d\bar{z}^{j}+\text{terms with }dz^{1}\text{ or }d\bar{z}^{1}.\label{eq:expan}
\end{equation}
Thus by (\ref{eq:expansio of rho}), (\ref{eq:expan}), (\ref{eq:tidld A_i}),
and (\ref{eq:basic formula}), $\<\frac{\rho_{1}^{l}}{l!},\frac{\chi_{\mathcal{B}}^{l}}{l!}\>=\sigma_{l}(\tilde{A}_{1}^{-1})$
and similarly 
\begin{equation}
\<\frac{\rho_{i}^{l}}{l!},\frac{\chi_{\mathcal{B}}^{l}}{l!}\>=\sigma_{l}(\tilde{A}_{i}^{-1}).\label{eq:rho_i chi_B}
\end{equation}
Let 
\begin{equation}
\tilde{\xi}=\<\check{e}_{1}^{1},\xi\>e_{1}^{1}+\sum_{i=n_{p}+1}^{n_{p}+d_{1}-1}\xi^{i}\frac{\pdv}{\pdv z^{i}}.\label{eq:tild exi}
\end{equation}
Then, by (\ref{eq:lem T_k}), we have
\begin{align}
\<\frac{\rho_{1}^{k_{1}}}{k_{1}!},2\sqrt{-1}\xi\wedge\bar{\xi}\wedge\exp\chi_{\mathcal{B}}\> & =\<T_{k_{1}-1}(\tilde{A}_{1}^{-1}),\tilde{\xi}^{\dagger}\tilde{\xi}\>\label{eq:-33}\\
 & \geq\frac{\sigma_{k_{1}-1}(\tilde{A}_{1}^{-1})}{\sigma_{d_{1}-1}(\tilde{A}_{1}^{-1})}\<T_{d_{1}-1}(\tilde{A}_{1}^{-1}),\tilde{\xi}^{\dagger}\tilde{\xi}\>.\nonumber 
\end{align}
Notice that
\begin{equation}
T_{d_{1}-1}(\tilde{A}_{1}^{-1})=\left(\begin{array}{cc}
\sigma_{d_{1}-1}(\tilde{A}_{1}^{-1}) & 0\\
0 & T'
\end{array}\right),\label{eq:-34}
\end{equation}
where $T'\geq0$ is a non-negative Hermitian matrix. Thus, from (\ref{eq:-33})
and (\ref{eq:-34}), we have
\begin{align}
\<\frac{\rho_{1}^{k_{1}}}{k_{1}!},2\sqrt{-1}\xi\wedge\bar{\xi}\wedge\exp\chi_{\mathcal{B}}\> & =\geq\frac{\sigma_{k_{1}-1}(\tilde{A}_{1}^{-1})}{\sigma_{d_{1}-1}(\tilde{A}_{1}^{-1})}\sigma_{d_{1}-1}(\tilde{A}_{1}^{-1})|\<\check{e}_{1}^{1},\xi\>|^{2}\label{eq:chi_B ineq}\\
 & =\sigma_{k_{1}-1}(\tilde{A}_{1}^{-1})|\<\check{e}_{1}^{1},\xi\>|^{2}.\nonumber 
\end{align}
Apply the same argument on each $\mathcal{V}_{i}$ to obtain 
\begin{equation}
\<\Lambda',2\sqrt{-1}\xi\wedge\bar{\xi}\wedge\exp\chi_{\mathcal{B}}\>\geq\sum_{i=1}^{n_{p}}\sigma_{k_{i}-1}(\tilde{A}_{i}^{-1})|\<\check{e}_{1}^{i},\xi\>|^{2},\label{eq:similar arg}
\end{equation}
where $\Lambda'=\sum_{i=1}^{n_{p}}(\exp\rho_{i})^{[k_{i}]}.$ On the
other hand, by applying Lemma \ref{lem:spplict 2} (2), we have 
\begin{equation}
\sigma_{n-1}\left((A|1)^{-1}\right)\leq\sum_{i=1}^{n_{p}}\sigma_{d_{i}-1}(\tilde{A}_{1}^{-1})\prod_{j\not=i}\sigma_{d_{j}}(\tilde{A}_{j}^{-1}).\label{eq:simga_n-1}
\end{equation}
Let $s_{i}=\sigma_{k_{i}}(\tilde{A}_{i}^{-1})^{\frac{1}{k_{i}}}$.
Then by Newton-Maclaurin inequality, 
\begin{equation}
\sigma_{k_{i}-1}(\tilde{A}_{i}^{-1})\geq\binom{d_{i}}{k_{i}-1}\binom{d_{i}}{k_{i}}^{\frac{1}{k_{i}}-1}s_{i}^{k_{i}-1},\label{eq:NM eq}
\end{equation}
\begin{equation}
\sum_{i=1}^{n_{p}}\sigma_{d_{i}-1}(\tilde{A}_{i}^{-1})\prod_{j\not=i}\sigma_{d_{j}}(\tilde{A}_{j}^{-1})\leq\frac{\max_{i}\{d_{i}\binom{d_{i}}{k_{i}}^{\frac{1}{k_{i}}}\}}{\prod_{i=1}^{n_{p}}\binom{d_{i}}{k_{i}}^{\frac{d_{i}}{k_{i}}}}\prod_{i=1}^{n_{p}}s_{i}^{d_{i}}\sum_{j=1}^{n_{p}}s_{j}^{-1}.\label{eq:NM eq 2}
\end{equation}
Thus, by (\ref{eq:similar arg}), (\ref{eq:simga_n-1}), (\ref{eq:NM eq}),
and (\ref{eq:NM eq 2}), we get 
\begin{align}
\<\Lambda',2\sqrt{-1}\xi\wedge\bar{\xi}\wedge\exp\chi_{\mathcal{B}}\> & \geq\sigma_{n-1}((A|1))^{-1}\cdot\frac{\sum_{i=1}^{n_{p}}\sigma_{k_{i}-1}(\tilde{A}_{i}^{-1})|\<\check{e}_{1}^{i},\xi\>|^{2}}{\sum_{i=1}^{n_{p}}\sigma_{d_{i}-1}(\tilde{A}_{1}^{-1})\prod_{j\not=i}\sigma_{d_{j}}(\tilde{A}_{j}^{-1})}\label{eq:pre est}\\
 & =\det(A)^{-1}\hat{a}\cdot\frac{\sum_{i=1}^{n_{p}}\sigma_{k_{i}-1}(\tilde{A}_{i}^{-1})|\<\check{e}_{1}^{i},\xi\>|^{2}}{\sum_{i=1}^{n_{p}}\sigma_{d_{i}-1}(\tilde{A}_{1}^{-1})\prod_{j\not=i}\sigma_{d_{j}}(\tilde{A}_{j}^{-1})}\nonumber \\
 & \geq c_{0}(\mathcal{O}_{p})\cdot\hat{a}\cdot\det(A)^{-1}\cdot\frac{\sum_{i=1}^{n_{p}}s_{i}^{k_{i}-1}|\<\check{e}_{1}^{i},\xi\>|^{2}}{\prod_{i=1}^{n_{p}}s_{i}^{d_{i}}\sum_{i=1}^{n_{p}}s_{i}^{-1}},\nonumber 
\end{align}
where 
\begin{equation}
c_{0}(\mathcal{O}_{p})=\frac{\min_{i}\left\{ \binom{d_{i}}{k_{i}-1}\binom{d_{i}}{k_{i}}^{\frac{1}{k_{i}}-1}\right\} \prod_{i=1}^{n_{p}}\binom{d_{i}}{k_{i}}^{\frac{d_{i}}{k_{i}}}}{\max_{i}\{d_{i}\binom{d_{i}}{k_{i}}^{\frac{1}{k_{i}}}\}}.\label{eq:c_0 O}
\end{equation}

On the other hand, by the cone condition and Lemma \ref{lem: induced cone},
we have $s_{i}\leq(\frac{\kappa}{m})^{\frac{1}{k_{i}}}$. Since $\<\xi,b\>=1$
and $b\in\text{span}\{\check{e}_{1}^{i}\}_{i=1}^{n_{p}}$, we have
$\sum_{i=1}^{n_{p}}|\<\check{e}_{1}^{i},\xi\>|^{2}\geq1$. Define
a function 
\begin{equation}
\gamma(s_{1},\cdots,s_{n_{p}}):=c_{0}(\mathcal{O}_{p})\frac{\sum_{i=1}^{n_{p}}s_{i}^{k_{i}-1}|\<\check{e}_{1}^{i},\xi\>|^{2}}{\prod_{i=1}^{n_{p}}s_{i}^{d_{i}}\sum_{i=1}^{n_{p}}s_{i}^{-1}}.\label{eq:gamm s1}
\end{equation}
Then $\gamma$ is decreasing in each $s_{i}$. Since $s_{i}\leq(\frac{\kappa}{m})^{\frac{1}{k_{i}}}$,
we have 
\begin{align}
\gamma & \geq c_{0}\frac{\sum_{i=1}^{n_{p}}\left(\frac{\kappa}{m}\right)^{1-\frac{1}{k_{i}}}|\<\check{e}_{1}^{i},\xi\>|^{2}}{(\frac{\kappa}{m})^{\sum_{i=1}^{n_{p}}\frac{d_{i}}{k_{i}}}\sum_{i=1}^{n_{p}}\left(\frac{\kappa}{m}\right)^{-\frac{1}{k_{i}}}}\label{eq:gamma ineq}\\
 & \geq c_{0}\frac{\min\{\left(\frac{\kappa}{m}\right)^{1-\frac{1}{k_{i}}}\}}{(\frac{\kappa}{m})^{\sum_{i=1}^{n_{p}}\frac{d_{i}}{k_{i}}}\sum_{i=1}^{n_{p}}\left(\frac{\kappa}{m}\right)^{-\frac{1}{k_{i}}}}\nonumber \\
 & =\gamma_{\min}(\frac{\kappa}{m},n_{p},\mathbf{d}_{p},\mathbf{k}_{p}).\nonumber 
\end{align}
Thus, by (\ref{eq:pre est}) and (\ref{eq:gamma ineq}), we have 
\begin{equation}
\<\Lambda',2\sqrt{-1}\xi\wedge\bar{\xi}\wedge\exp\chi_{\mathcal{B}}\>\geq\hat{a}\cdot\det A^{-1}\cdot\gamma_{\min}.\label{eq:Lem prod sp 2}
\end{equation}
Notice $m\Lambda'\leq\mathring{\Lambda}$ and $\hat{a}=\frac{1}{\|b\|_{\omega}^{2}}.$
Hence, (\ref{eq:lemma prod sp}) follows from (\ref{eq:Lem prod sp 2})
immediately. 
\end{proof}
An immediate consequence of Lemma \ref{lem:H1''' inequality} is the
monotonicity of $F(A)$. 
\begin{lem}
\label{lem:mononton with f}Suppose $\mathring{\Lambda}$ satisfies
$\mathcal{O}$-UP and for all $p\in M$, $f(p)>-m\gamma_{\min}(\frac{\kappa}{m},n_{p},\mathbf{d}_{p},\mathbf{k}_{p})$.
If $\omega$ satisfies the cone condition (\ref{eq:cone condition})
at a point $p$, then 
\begin{equation}
-F^{i\bar{j}}(A)b_{i}\bar{b_{j}}\geq-\frac{1}{2}\sum_{k=1}^{n-1}F_{k}^{i\bar{j}}b_{i}\bar{b}_{j}>0,\label{eq:first order mono}
\end{equation}
for any non-zero covector $b=b_{i}dz^{i}$. As a result, $F(A)$ is
strictly decreasing in $\mathcal{C}_{\Lambda}^{\kappa}$. 
\end{lem}

\begin{proof}
From (\ref{eq:lem of comp}), we have 
\begin{equation}
-F_{k}^{i\bar{j}}(A)b_{i}\bar{b}_{j}=\<\Lambda^{[k]},2\sqrt{-1}b^{\sharp}\wedge\bar{b^{\sharp}}\wedge\exp\chi\>.\label{eq:mon with f}
\end{equation}
Let $\mathcal{B}=\{\xi\in\mathcal{T}_{p}M:b(\xi)=0\}$. Since $A^{-1}|_{\mathcal{B}}\geq\left(A|_{\mathcal{B}}\right)^{-1}$
by Lemmas \ref{rem:elementary linear agb} and \ref{lem:exp storng positive},
$\exp\chi\geq_{s}\exp\chi_{\mathcal{B}}$ . Thus, if we take $\xi=b^{\sharp}$
in (\ref{eq:lemma prod sp}), then we have 
\begin{equation}
-\sum_{k=1}^{n-1}F_{k}^{i\bar{j}}b_{i}\bar{b}_{j}\geq\<\Lambda^{[k]},2\sqrt{-1}b^{\sharp}\wedge\bar{b^{\sharp}}\wedge\exp\chi_{\mathcal{B}}\>\geq m\gamma_{\min}(\frac{\kappa}{m},n_{p},\mathbf{d}_{p},\mathbf{k}_{p})\frac{\|b\|_{\omega}^{2}}{\det A}.\label{eq:deri ineq}
\end{equation}
Direct computation shows that 
\begin{align*}
-F^{i\bar{j}}b_{i}\bar{b}_{j} & =-\sum_{k=1}^{n-1}F_{k}^{i\bar{j}}b_{i}\bar{b}_{j}+\frac{f(p)}{\det A}A^{\bar{j}i}b_{i}\overline{b_{j}}
\end{align*}
Then, from Lemma \ref{lem:H1''' inequality}, we have 
\begin{align}
-F^{i\bar{j}}b_{i}\bar{b}_{j} & \geq-\frac{1}{2}\sum_{k=1}^{n-1}F_{k}^{i\bar{j}}b_{i}\bar{b}_{j}+\left(\frac{m}{2}\gamma_{\min}(\frac{\kappa}{m},n_{p},\mathbf{d}_{p},\mathbf{k}_{p})+f(p)\right)\frac{\|b\|_{\omega}^{2}}{\det A}\label{eq:interm monot with f}\\
 & \ge-\frac{1}{2}\sum_{k=1}^{n-1}F_{k}^{i\bar{j}}b_{i}\bar{b}_{j}>0.\nonumber 
\end{align}
\end{proof}
Now we prove Lemma \ref{lem:Datar-Pingali}. 
\begin{proof}[Proof of Lemma \ref{lem:Datar-Pingali}]
We consider a generic point $p_{1}\in\mathcal{M}$. If $f(p_{1})\geq0$,
the cone condition holds automatically. Thus, we may assume $f(p_{1})<0.$
Let $\gamma(s),$$s\in[0,1]$ be a curve connecting $p_{0}$ and $p_{1}$.
Let $s_{0}=\min\{s\in(0,1],\ \omega(\gamma(s))\notin\mathcal{C}_{\Lambda}^{\kappa}\}$
and let $p=\gamma(s_{0})$. By Lemma \ref{lem: submatrix and restricting to subspac},
the degeneracy of cone condition implies that we can find a rank 1
Hermitian matrix $B=b^{\dagger}b$ such that 
\begin{equation}
F_{\Lambda}(A:B)=\mathcal{P}_{\Lambda}(A)=\kappa.\label{eq:eq:F(a:B) equation}
\end{equation}
At $p$, we pick a normal coordinate of $\rho$ as in Lemma \ref{lem:H1''' inequality}.
Let 
\begin{equation}
\chi=2\sqrt{-1}\sum_{i,j}A^{\bar{j}i}\frac{\pdv}{\pdv\bar{z}^{j}}\wedge\frac{\pdv}{\pdv z^{i}},\ \chi_{\infty}=2\sqrt{-1}\sum_{i,j>1}(A|1)^{\bar{j}i}\frac{\pdv}{\pdv\bar{z}^{j}}\wedge\frac{\pdv}{\pdv z^{i}},\label{eq:def chi ag}
\end{equation}
where $((A|1)^{\bar{j}i})$ is the inverse matrix of $(A|1)$. From
(\ref{eq:eq:F(a:B) equation}), $\<\Lambda,\exp(\chi_{\infty})\>=\kappa.$
From equation (\ref{eq:equation with f}), $\<\Lambda,\exp(\chi)\>=\kappa.$
Thus, we have 
\[
\sum_{k=1}^{n-1}\frac{\<\Lambda^{[k]},\left(\chi^{k}-\chi_{\infty}^{k}\right)\>}{k!}+f(p)\<\frac{\rho^{n}}{n!},\frac{\chi^{n}}{n!}\>=0.
\]
Hence 
\begin{align}
-f(p)\<\frac{\rho^{n}}{n!},\frac{\chi^{n}}{n!}\> & =\sum_{k=1}^{n-1}\frac{\<\Lambda^{[k]},\left(\chi^{k}-\chi_{\infty}^{k}\right)\>}{k!}.\label{eq:DP lem 2 prod}
\end{align}
We may write 
\begin{equation}
A=\left(\begin{array}{cc}
a & Q\\
Q^{\dagger} & (A|1)
\end{array}\right),A^{-1}=\left(\begin{array}{cc}
\hat{a}^{-1} & \hat{Q}\\
\hat{Q}^{\dagger} & (A|1)^{-1}+\hat{Q}^{\dagger}\hat{a}\hat{Q}
\end{array}\right),\label{eq:a rew}
\end{equation}
where $\hat{a}=a-Q(A|1)^{-1}Q^{\dagger}=1/\|dz^{1}\|_{\omega}^{2}$
and $\hat{Q}=-\hat{a}^{-1}Q(A|1)^{-1}$. Let 
\begin{equation}
\xi=\frac{1}{\sqrt{\hat{a}}}\frac{\pdv}{\pdv z^{1}}+\sqrt{\hat{a}}\sum_{i=2}^{n}\hat{Q}_{i}\frac{\pdv}{\pdv z^{i}}.\label{eq:xi def d}
\end{equation}
Then 
\begin{align}
\chi & =2\sqrt{-1}\left(\xi\wedge\bar{\xi}+(A|1)^{\bar{j}i}\frac{\pdv}{\pdv\bar{z}^{j}}\wedge\frac{\pdv}{\pdv z^{i}}\right),\label{eq:chi eq}
\end{align}
Thus, we have 
\begin{equation}
\exp\chi-\exp\chi_{\infty}=2\sqrt{-1}\xi\wedge\bar{\xi}\wedge\exp\chi_{\infty}.\label{eq:DP prod 3}
\end{equation}
By (\ref{eq:DP lem 2 prod}), (\ref{eq:DP prod 3}), and Lemma \ref{lem:H1''' inequality},
we have 
\begin{align}
-\frac{f(p)}{\det A} & \geq\frac{1}{\det A}\frac{|\<\xi,dz^{1}\>|^{2}}{\|dz^{1}\|_{\omega}^{2}}m\gamma_{\min}\label{eq:fp lower bound}\\
 & =\frac{1}{\det A}\cdot m\gamma_{\min}.\nonumber 
\end{align}
However, it is impossible as $|f(p)|<m\gamma_{\min}$. We have finished
the proof. 

Finally, we show that $F$ is strict convex in $\mathcal{C}_{\Lambda}^{\kappa}(p)$.
\end{proof}
\begin{lem}
\label{lem:conv cont path} Notations as above. If $\mathring{\Lambda}$
satisfies $\mathcal{O}$-UP then 
\begin{equation}
\sum_{r,s,i,j}\left(F^{i\bar{j},r\bar{s}}(A)+F^{i\bar{s}}(A)A^{\bar{j}r}\right)B_{i\bar{j}}\overline{B_{s\bar{r}}}\geq\frac{f(p)}{\det A}A^{\bar{s}r}A^{\bar{j}i}B_{i\bar{j}}\overline{B_{s\bar{r}}}.\label{eq:convexity in continuous path}
\end{equation}
 Suppose further that for all $p\in M$,$f(p)>-\frac{m}{2n+1}\gamma_{\min}(\frac{\kappa}{m},n_{p},\mathbf{d}_{p},\mathbf{k}_{p}).$
Then,  $F$ is a strictly convex function in $\mathcal{C}_{\Lambda}^{\kappa}$.
\end{lem}

\begin{proof}
 In a local normal coordinate of $\rho$ at $p$, we may assume that
$A=\text{diag}\{\lambda_{1},\cdots,\lambda_{n}\}$. By Proposition
\ref{prop: F_k convexity}, 
\begin{align}
\sum_{i,j,r,s}\left(F^{i\bar{j},r\bar{s}}+F^{i\bar{s}}A^{\bar{j}r}\right)B_{i\bar{j}}\overline{B_{s\bar{r}}} & \geq\frac{f(p)}{\det A}\left(A^{\bar{s}r}A^{\bar{j}i}\right)B_{i\bar{j}}\overline{B_{s\bar{r}}}\label{eq:convex with f}\\
 & =\frac{f(p)}{\det A}\left|\sum_{j}\frac{B_{j\bar{j}}}{\lambda_{j}}\right|^{2}.\nonumber 
\end{align}
By (\ref{eq:interm monot with f}), we have 
\begin{align}
-\sum_{i,j,r,s}F^{i\bar{s}}(A)A^{\bar{j}r}B_{i\bar{j}}\overline{B_{s\bar{r}}} & \geq\frac{1}{\det A}\left(\frac{m}{2}\gamma_{\min}(\frac{\kappa}{m},n_{p},\mathbf{d}_{p},\mathbf{k}_{p})+f(p)\right)\sum_{i,j}\frac{|B_{i\bar{j}}|^{2}}{\lambda_{i}\lambda_{j}}.\label{eq:strc cov-1}
\end{align}
Then by (\ref{eq:strc cov-1}), Cauchy inequality, and the assumption
on $f$, we have
\begin{align}
 & \sum_{i,j,r,s}\left(F^{i\bar{j},r\bar{s}}+\frac{1}{2}F^{i\bar{s}}A^{\bar{j}r}\right)B_{i\bar{j}}\overline{B_{s\bar{r}}}\label{eq:strc conv 2-1}\\
 & \geq\frac{1}{\det A}\sum_{j}\frac{|B_{j\bar{j}}|^{2}}{\lambda_{j}^{2}}\left(\frac{m}{2}\gamma_{\min}(\frac{\kappa}{m},n_{p},\mathbf{d}_{p},\mathbf{k}_{p})-(n+\frac{1}{2})|f(p)|\right)\nonumber \\
 & \geq0.\nonumber 
\end{align}
Since $F^{i\bar{j}}$ is strictly monotone in $\mathcal{C}_{\Lambda}^{\kappa}$
by Lemma \ref{lem:mononton with f}, by (\ref{eq:strc conv 2-1})
$F^{i\bar{j},r\bar{s}}$ is positive definite in $\mathcal{C}_{\Lambda}^{\kappa}$.
Thus, $F$ is strictly convex in $\mathcal{C}_{\Lambda}^{\kappa}$.
\end{proof}

\section{Continuity method\label{sec:Continuity-method}}

In this section, we prove the existence of a solution to (\ref{eq:equation with f})
assuming the existence of a subsolution $\omega_{\text{sub}}\in[\omega_{0}].$
The main result of this section is the following 
\begin{thm}
\label{thm:Analytic theorem}Let $M$ be a connected compact Kähler
manifold. Suppose $\Lambda$ satisfies \textbf{H2}. If there is a
Kähler metric $\omega_{\text{sub}}\in[\omega_{0}]$ satisfies the
cone condition (\ref{eq:cone condition}), then there exists a unique
smooth solution to equation (\ref{eq:equation with f}). 
\end{thm}

\begin{rem}
Since a solution itself is a subsolution by Lemma \ref{lem:Datar-Pingali},
Theorem \ref{thm:Let--beanlytic} and Theorem \ref{thm:Let--beanlytic-1}
are immediate corollaries of Theorem \ref{thm:Analytic theorem}.
The uniqueness will be addressed in Appendix \ref{sec:Functional-and-uniqueness}. 
\end{rem}

To prove Theorem \ref{thm:Analytic theorem}, we use the continuity
method. Consider the following continuity path depending on the parameter
$t\in[0,1]$. Consider $\Omega_{t}=\exp\omega_{t}$ which solves the
PDE 
\begin{equation}
\kappa(\Omega_{t})^{[n]}=\left(t\mathring{\Lambda}\wedge\Omega_{t}\right)^{[n]}+(tf+\left(1-t\right)\kappa_{0})P^{[n]},\label{eq:contin path}
\end{equation}
where the constant $\kappa_{0}$ is chosen such that 
\[
\kappa[\omega_{0}]^{n}=\kappa_{0}[\rho]^{n}.
\]
Denote $f_{t}(p):=tf(p)+(1-t)\kappa_{0}$, and $\Lambda_{t}=t\mathring{\Lambda}+f_{t}P^{[n]}$,
then we may rewrite (\ref{eq:contin path}) as
\begin{equation}
\kappa(\Omega_{t})^{[n]}=\left(\Lambda_{t}\wedge\Omega_{t}\right)^{[n]}.\label{eq:new cont path}
\end{equation}
Let 
\[
\mathbf{I}=\{t\in[0,1]:(\ref{eq:contin path})\ \text{has a smooth solution}\}.
\]
Then, from Yau's theorem \cite{yau1978ricci}, there is a smooth $\kah$
metric $\hat{\omega}_{0}$ which solves (\ref{eq:new cont path})
at $t=0$. Without confusion, we replace $\omega_{0}$ by $\hat{\omega}_{0}$
and consider $\omega_{t}=\omega_{0}+i\ddbar\varphi_{t}$ which solves
(\ref{eq:contin path}) for $t\in\mathbf{I}\subset[0,1]$. Notice
the linearization of (\ref{eq:contin path}) is 
\begin{equation}
\left(\kappa\Omega_{t}-t\Lambda\wedge\Omega_{t}\right)^{[n-1]}\wedge i\ddbar u=(t\mathring{\Lambda}\wedge\Omega_{t})^{[n]}+\left(f(p)-\kappa_{0}\right)P^{[n]}.\label{eq:add31}
\end{equation}
By Lemma \ref{lem:Datar-Pingali}, $\left(\kappa\Omega_{t}-t\Lambda\wedge\Omega_{t}\right)^{[n-1]}>0$;
Therefore, (\ref{eq:add31}) is strictly elliptic, which implies the
openness of $\mathbf{I}$. If $\Lambda$ satisfies \textbf{H2}, then
for $t>0$, $\Lambda_{t}$ satisfies \textbf{H2 }with respect to proper
positive constants. In fact, $f_{t}(p)\geq0$ if $t\leq\frac{1}{2}$,
and since $m_{t}\geq\frac{m}{2}$ for $t\geq\frac{1}{2}$, we have
\[
f_{t}(p)\geq-\min\left\{ \frac{m_{t}}{2n+1}\gamma_{\min}(\frac{\kappa}{m_{t}},n_{p},\mathbf{d}_{p},\mathbf{k}_{p}),\frac{\kappa_{0}}{2}\right\} .
\]
Thus, if $\Lambda$ satisfies \textbf{H2,} $\Lambda_{t}$ satisfies
the following \textbf{H2' }condition. 
\begin{defn}
We say $\Lambda$ satisfies \textbf{H2' }if $\mathring{\Lambda}$
satisfies \textbf{H2 }with some uniform constant $m>0$, and $\Lambda^{[n]}$
is almost positive with respect to $(\mathcal{O},2m,\rho)$, i.e.
for any $p\in M$,
\[
\frac{\Lambda^{[n]}}{P^{[n]}}(p)\geq-\min\left\{ \frac{m}{2n+1}\gamma_{\min}(\frac{\kappa}{m},n_{p},\mathbf{d}_{p},\mathbf{k}_{p}),\frac{\kappa_{0}}{2}\right\} .
\]
\end{defn}

\begin{rem}
Readers may check that all arguments in section \ref{sec:Equations-with-negative}
are valid if condition \textbf{H2'} is assumed.
\end{rem}

In the following, by replacing $\Lambda$ with $\Lambda_{t}$, we
will suppress the subscript $t$ and derive a priori estimates for
equation (\ref{eq:equation with f}) assuming \textbf{H2'}. 

We proceed to prove a priori $C^{0}$-estimate of equation (\ref{eq:equation with f}).
The idea of using Alexandroff-Bakelman-Pucci type estimates based
on \cite{blocki2005uniform} follows \cite{szekelyhidi2018fully}.
Let 
\begin{equation}
\omega=\omega_{0}+i\ddbar\varphi\label{eq:-5}
\end{equation}
be a solution to (\ref{eq:equation with f}). Suppose that in the
$\kah$ class $[\omega_{0}]$ there is a $\omega_{\text{sub}}\in\mathcal{C}_{\Lambda}^{\kappa}$.
We denote 
\begin{equation}
\omega_{\text{sub}}=\omega_{0}+i\ddbar\varphi_{\text{sub}}\label{eq:-6}
\end{equation}
Let $u=\varphi-\varphi_{\text{sub}}$. 
\begin{prop}
\label{prop:C0 estimate}Suppose that $\Lambda$ satisfies \textbf{H2'}.
If $\omega=\omega_{\text{sub}}+i\ddbar u$ is a solution to (\ref{eq:equation with f})
and $\sup u=0$, then there is a constant $C$ depends on $n$, $M$,
$\Lambda$, $\omega_{\text{sub}}$, $\rho$ s.t. 
\[
\sup_{M}|u|<C.
\]
\end{prop}

Suppose that $\underline{A}\in\mathcal{C}_{\Lambda}^{\kappa}(p)$.
Since the $\mathcal{C}_{\Lambda}^{\kappa}(p)$ is open, $\text{dist}(\underline{A},\pdv\mathcal{C}_{\Lambda}^{\kappa}(p))>0$.
Thus there is a $0<r<\text{dist}(\underline{A},\mathcal{C}_{\Lambda}^{\kappa}(p))$
s.t. the radius $r$ ball $\mathcal{B}_{r}(\underline{A})$ in Hermitian
matrix space is contained in $\mathcal{C}_{\Lambda}^{\kappa}(p)$
. As a result, we have 
\begin{lem}
If $\underline{A}\in\mathcal{C}_{\Lambda}^{\kappa}(p)$, then there
is a constant $r=r(\underline{A},\Lambda)$ s.t. $\underline{A}-r\text{Id}\in\mathcal{C}_{\Lambda}^{\kappa}(p)$. 
\end{lem}

\begin{proof}[Proof of Proposition \ref{prop:C0 estimate}]
Since $\omega$ is positive, we have 
\begin{equation}
\Delta_{\rho}u>-\text{tr}_{\rho}\omega_{\text{sub}}>-C(\omega_{\text{sub}},\rho).\label{eq:Delta u}
\end{equation}
Thus, we can use Green function representation to obtain that $\|u\|_{L^{1}(\rho^{n})}<C(\omega_{\text{sub}},\rho)$.

Let $L=-\inf_{M}u$ and assume that $L$ is achieved at $x_{0}$.
Pick a normal coordinate of $\rho$ at $x_{0}$. After a proper rescaling
of $\rho,$ we may assume that the chosen coordinate exists in the
unit ball $B_{1}(0)\subset\C^{n}$. For $x\in B_{1}(0)$, we pick
a uniform $r=r(\omega_{\text{sub}},\Lambda)$ s.t. $\omega_{\text{sub}}-r\sqrt{-1}\ddbar|x|^{2}$
belongs to $\mathcal{C}_{\Lambda}^{\kappa}(x)$ for all $x\in B_{1}(0)$.
Let $a>0$ and $a<r/2$. Let $w=u+a|x|^{2}$. Note $w>-L+a$ on $\pdv B_{1}(0)$.
Define the following set: 
\[
W=\{x\in B_{1}(0):|Dw(x)|<a,\ w(y)\ge w(x)+Dw(x)\cdot(y-x)\}.
\]
We use Alexandroff-Bakelman-Pucci maximum principle (Gilbarg-Trudinger
Lemma 9.2) to claim that 
\[
B_{a}(0)\subset Dw(W).
\]
In $W$, we have $D^{2}w\geq0$ and hence 
\begin{equation}
c(n)a^{2n}\leq\int_{W}\det(D^{2}w)\leq2^{2n}\int_{W}\left(\det w_{i\bar{j}}\right)^{2}.\label{eq:vol of P}
\end{equation}
As $D^{2}w\geq0$ in $W$, $D^{2}u\geq-2a\text{Id}_{2n}$ which implies
that $u_{i\bar{j}}+a\delta_{i\bar{j}}\geq0$ as a Hermitian matrix.
Since $a<r/2$, $\omega_{\text{sub}}-a\sqrt{-1}\ddbar|x|^{2}\in\mathcal{C}_{\Lambda}^{\kappa}(x)$,
we apply Proposition \ref{prop:crite for Sub } to $\omega_{\text{sub}}-a\sqrt{-1}\ddbar|x|^{2}$
to conclude that $|w_{i\bar{j}}|<R(\Lambda,\omega_{\text{sub}})$.
By (\ref{eq:vol of P}), we have 
\begin{equation}
c(n)a^{2n}\leq C(\Lambda,\omega_{\text{sub}},n)\text{vol}(W).\label{eq:add33}
\end{equation}
On the other hand, it is obvious that
\begin{equation}
\text{vol}(W)\leq\frac{\|w\|_{L^{1}}}{|L-a|}\leq\frac{C(\omega_{\text{sub}},\rho)}{|L-a|}.\label{eq:add32}
\end{equation}
Thus by (\ref{eq:add33}) and (\ref{eq:add32}),
\[
L<C(\Lambda,\omega_{\text{sub}},\rho,n)\left(\frac{1}{a^{2n}}+1\right)<C(\Lambda,\omega_{\text{sub}},\rho,n).
\]
\end{proof}
Next, we state a $C^{2}$-estimate for solutions to (\ref{eq:contin path}). 
\begin{prop}
\label{prop:C2 estimate} Let $\omega_{\text{sub}}\in\mathcal{C}_{\Lambda}^{\kappa}$.
Let $u=\varphi-\varphi_{\text{sub}}$ where $\varphi,\varphi_{\text{sub}}$
are given in (\ref{eq:-5}),(\ref{eq:-6}) and $\omega=\omega_{\text{sub}}+\sqrt{-1}\ddbar u$
solves (\ref{eq:equation with f}). If $\Lambda$ satisfies \textbf{H2'},
and $F^{i\bar{j},r\bar{s}},F^{i\bar{j}}$ satisfies inequality (\ref{eq:convexity in continuous path})
then it holds that 
\[
|\ddbar u|_{\rho}<C.
\]
where $C$ depends on $\|u\|_{C^{0}},M,\Lambda,\omega_{\text{sub}}$,$n,m,\kappa,C_{H2}$,
$\rho$. 
\end{prop}

We first prove some technical lemmas. 
\begin{lem}
\label{lem: tech}Notations as above. We have 
\[
-F^{i\bar{j}}A_{i\bar{j}}\leq nF(A).
\]
\end{lem}

\begin{proof}
Since each $F_{k}$ is homogeneous of degree $-k$, we have 
\begin{align*}
-\frac{\pdv F(A)}{\pdv A_{i\bar{j}}}A_{i\bar{j}} & =\sum_{k=1}^{n-1}kF_{k}(A)+\frac{nf}{\det A}\\
 & \leq\sum_{k=1}^{n-1}nF_{k}(A)+\frac{nf}{\det A}\\
 & =nF(A).
\end{align*}
\end{proof}
Let $\{z^{i}\}$ be a normal coordinate at $p$. Pick the direction
$\frac{\pdv}{\pdv z^{1}}$ and denote $\Lambda_{,1},\Lambda_{,1\bar{1}}$
to be the corresponding covariant derivatives of $\Lambda$ with respect
to the Chern connection of $\rho$. We denote 
\begin{align}
F_{,1\bar{1}} & =F(A,\Lambda_{,1\bar{1}}),\ F_{,\bar{1}}^{i\bar{j}}=\frac{\pdv F\left(A,\Lambda_{,\bar{1}}\right)}{\pdv A_{i\bar{j}}},\label{eq:F_11}\\
F_{k,1\bar{1}} & =F_{k}(A,\Lambda_{,1\bar{1}}),\ F_{k,\bar{1}}^{i\bar{j}}=\frac{\pdv F_{k}\left(A,\Lambda_{,\bar{1}}\right)}{\pdv A_{i\bar{j}}}.\label{eq:F_k11}
\end{align}

\begin{lem}
\label{lem:tech2} Notations as above. If $\Lambda$ satisfies \textbf{H2'},
and $F(A)\geq\kappa$, we have 
\begin{equation}
\left|F_{,1\bar{1}}\right|<C_{\ref{lem:tech2}}F(A),\label{eq:F_11 estimate}
\end{equation}
\begin{equation}
\left|2\text{Re}\left(F_{k,\bar{1}}^{i\bar{j}}B_{i\bar{j}}\right)\right|\leq\frac{C_{\ref{lem:tech2}}}{\epsilon}F(A)+C_{\ref{lem:tech2}}\epsilon F^{i\bar{s}}A^{r\bar{j}}B_{i\bar{j}}\overline{B_{s\bar{r}}},\label{eq:F_k,11 estimate}
\end{equation}
for some $C_{\ref{lem:tech2}}=C_{\ref{lem:tech2}}(\kappa,n,m,C_{H2})$. 
\end{lem}

\begin{proof}
From the cone condition, Lemma \ref{lem: induced cone}, and Newton-Maclaurin
inequality, we have 
\begin{equation}
\<\frac{\rho_{i}^{l_{i}}}{l_{i}!},\exp\chi\>\leq\binom{d_{i}}{l_{i}}\binom{d_{i}}{k_{i}}^{-\frac{l_{i}}{k_{i}}}(\frac{n\kappa}{m})^{\frac{l_{i}}{k_{i}}}.\label{eq:use cone conditiopn}
\end{equation}
By \textbf{H2'}, there is a constant $C_{H2}$ s.t. for $k=1,\cdots,n-1$,
it holds 
\[
-C_{H2}\left(\Lambda^{[k]}+\sum_{l\in\boldsymbol{l}_{k}}\rho_{1}^{l_{1}}\cdots\rho_{n_{p}}^{l_{n_{p}}}\right)\leq\left(\Lambda^{[k]}\right)_{,1\bar{1}}\leq C_{H2}\left(\Lambda^{[k]}+\sum_{l\in\boldsymbol{l}_{k}}\rho_{1}^{l_{1}}\cdots\rho_{n_{p}}^{l_{n_{p}}}\right),
\]
where $\boldsymbol{l}_{k}=\{(l_{1},\cdots,l_{n_{p}}):\sum_{i}l_{i}=k,\ l_{i}=0\ \text{or}\ l_{i}\geq k_{i}\}$.
Hence 
\begin{align}
|F_{k,1\bar{1}}| & \leq C_{H2}\left(F_{k}+\sum_{l\in\boldsymbol{l}_{k}}\prod_{i=1}^{n_{p}}\<\frac{\rho_{i}^{l_{i}}}{l_{i}!},\exp\chi\>\right)\label{eq:est F_k11}\\
 & =C_{H2}\left(F_{k}+\sum_{l\in\boldsymbol{l}_{k}}C_{1}(\frac{\kappa}{m},n_{p},\mathbf{d}_{p},\mathbf{k}_{p})\right)\nonumber \\
 & \leq C_{H2}\left(F_{k}+C_{2}(\frac{\kappa}{m},n_{p},\mathbf{d}_{p},\mathbf{k}_{p})\right)\nonumber 
\end{align}
By the cone condition and Lemma \ref{lem: induced cone},

\begin{equation}
\det A^{-1}\leq C_{3}(m,\kappa,n_{p},\mathbf{d}_{p},\mathbf{k}_{p}).\label{eq:det est}
\end{equation}
By (\ref{eq:est F_k11}) and (\ref{eq:det est}), we have 
\begin{align}
\left|F_{,1\bar{1}}\right| & <\sum_{k}|F_{k,1\bar{1}}|+|f_{,1\bar{1}}|\frac{1}{\det A}\label{eq:F_11 est}\\
 & <C_{4}F(A),\nonumber 
\end{align}
for some $C_{4}=C_{4}(\kappa,m,n_{p},\mathbf{d}_{p},\mathbf{k}_{p},C_{H2})$.

We use \textbf{H2' }to deduce that 
\begin{equation}
\left|\text{Re}\left(F_{k,\bar{1}}^{i\bar{j}}b_{i}\bar{b}_{j}\right)\right|\leq C_{H2}\left(-F_{k}^{i\bar{j}}b_{i}\bar{b}_{j}+\<\sum_{l\in\boldsymbol{l}_{k}}\rho_{1}^{l_{1}}\cdots\rho_{n_{p}}^{l_{n_{p}}},\exp\chi\wedge2\sqrt{-1}\overline{b^{\sharp}}\wedge b^{\sharp}\>\right).\label{eq:new inter gra}
\end{equation}
From (\ref{eq:spllit 4}) in Lemma \ref{lem:spplict 2}, we have 
\begin{equation}
\<\sum_{l\in\boldsymbol{l}_{k}}\rho_{1}^{l_{1}}\cdots\rho_{n_{p}}^{l_{n_{p}}},\exp\chi\wedge2\sqrt{-1}\overline{b^{\sharp}}\wedge b^{\sharp}\>=\sum_{i=1}^{n_{p}}\sum_{l\in\boldsymbol{l}_{k,i}}\left(\prod_{j\not=i}\<\frac{\rho_{j}^{l_{j}}}{l_{j}!},e^{\chi}\>\right)\<\frac{\rho_{i}^{l_{i}}}{l_{i}!},e^{\chi}\wedge2\sqrt{-1}\overline{b^{\sharp}}\wedge b^{\sharp}\>,\label{eq:new inte grad 2}
\end{equation}
where $\boldsymbol{l}_{k,i}=\{(l_{1},\cdots,l_{n_{p}})\in\boldsymbol{l}_{k}:l_{i}\geq1\}$.
By (\ref{eq:lem T_k}) in Lemma \ref{lem: quotient k-hessian-1} and
Newton-Maclaurin inequality, for $l\geq k_{i}$ and any positive $d_{i}\times d_{i}$
Hermitian matrix $D$, we have 
\begin{align}
T_{l-1}(D) & \leq\frac{\sigma_{l-1}(D)}{\sigma_{k_{i}-1}(D)}T_{k_{i}-1}(D)\label{eq:from kenm T_k-1}\\
 & \leq C(d_{i},l,k_{i})\frac{\sigma_{k_{i}-1}(D)(\sigma_{k_{i}}(D))^{\frac{l-k_{i}}{k_{i}}}}{\sigma_{k_{i}-1}(D)}T_{k_{i}-1}(D)\nonumber \\
 & =C(d_{i},l,k_{i})(\sigma_{k_{i}}(D))^{\frac{l-k_{i}}{k_{i}}}T_{k_{i}-1}(D).\nonumber 
\end{align}
We apply (\ref{eq:from kenm T_k-1}) to $A^{-1}|_{\mathcal{V}_{i}}$
and use (\ref{eq:use cone conditiopn}) to obtain 
\begin{equation}
\<\frac{\rho_{i}^{l-1}}{(l-1)!},\exp\chi\wedge2\sqrt{-1}\overline{b^{\sharp}}\wedge b^{\sharp}\>\leq C_{5}(\frac{\kappa}{m},n_{p},\mathbf{d}_{p},\mathbf{k}_{p},l)\<\frac{\rho_{i}^{k_{i}}}{k_{i}!},\exp\chi\wedge2\sqrt{-1}\overline{b^{\sharp}}\wedge b^{\sharp}\>.\label{eq:intermideate tech 2-1}
\end{equation}
Thus, by (\ref{eq:new inter gra}), (\ref{eq:new inte grad 2}), and
(\ref{eq:intermideate tech 2-1}), we have 
\begin{align}
\left|\text{Re}\left(F_{k,\bar{1}}^{i\bar{j}}b_{i}\bar{b}_{j}\right)\right| & \leq C_{H2}\left(-F_{k}^{i\bar{j}}b_{i}\bar{b}_{j}+C_{6}(\frac{\kappa}{m},n_{p},\mathbf{d}_{p},\mathbf{k}_{p})\sum_{i=1}^{n_{p}}\<\frac{\rho_{i}^{k_{i}}}{k_{i}!},\exp\chi\wedge2\sqrt{-1}\overline{b^{\sharp}}\wedge b^{\sharp}\>\right)\label{eq:ineqau inter tech}\\
 & \leq-C_{7}F^{i\bar{j}}b_{i}\bar{b}_{j},\nonumber 
\end{align}
where $C_{7}=C_{7}(\kappa,m,n_{p},\mathbf{d}_{p},\mathbf{k}_{p},C_{H2})$.

Now for matrix $B$, by (\ref{eq:ineqau inter tech}) and mean value
inequality, 
\begin{align}
\left|\text{Re}\left(F_{,1}^{i\bar{s}}A^{r\bar{j}}(A_{i\bar{j}}+\epsilon B_{i\bar{j}})\left(\overline{A_{s\bar{r}}+\epsilon B_{s\bar{r}}}\right)\right)\right| & \leq-C_{9}F^{i\bar{s}}A^{r\bar{j}}(A_{i\bar{j}}+\epsilon B_{i\bar{j}})\left(\overline{A_{s\bar{r}}+\epsilon B_{s\bar{r}}}\right)\label{eq:re inner est}\\
 & \leq2C_{9}\left(F(A)-\epsilon^{2}F^{i\bar{s}}A^{r\bar{j}}B_{i\bar{j}}\overline{B_{s\bar{r}}}\right),\nonumber 
\end{align}
where $C_{9}=C_{9}(n,\kappa,m,n_{p},\mathbf{d}_{p},\mathbf{k}_{p},C_{H2})$.
As a result of (\ref{eq:re inner est}), we have 
\[
\left|2\text{Re}\left(F_{,\bar{1}}^{i\bar{j}}B_{i\bar{j}}\right)\right|\leq\frac{C_{10}}{\epsilon}F(A)-C_{10}\epsilon F^{i\bar{s}}A^{r\bar{j}}B_{i\bar{j}}\overline{B_{s\bar{r}}},
\]
where $C_{10}=C_{10}(n,\kappa,m,n_{p},\mathbf{d}_{p},\mathbf{k}_{p},C_{H2})$.

Finally, we choose 
\[
C_{\ref{lem:tech2}}=\max\{C_{2},C_{4},C_{10}:p\in M\}.
\]
Notice that for a fixed labeled\emph{ }orthogonal splitting $\mathcal{O}$,
the set $\{(n_{p},\mathbf{d}_{p},\mathbf{k}_{p}):p\in M\}$ is finite.
Thus $C_{\ref{lem:tech2}}$ has a uniform upper bound which only depends
on $\kappa,m,n$ and $C_{H2}$. 
\end{proof}
\begin{lem}
\label{lem:ddbar F}Notations as above. If $\Lambda$ satisfies \textbf{H2'},
and $F(A)\geq\kappa$, there is a constant $C_{\ref{lem:ddbar F}}$
depends on $C_{\ref{lem:tech2}}$, and the bisectional curvature $\text{Rm}_{\rho}$
of $\rho$ s.t. for any $\epsilon>0$
\begin{align*}
\pdv_{1\bar{1}}F(A) & \geq A_{1\bar{1}}F^{i\bar{j}}\left(\log A_{1\bar{1}}\right)_{,i\bar{j}}+A^{\bar{1}1}F^{i\bar{j}}A_{i\bar{1},1}\overline{A_{j\bar{1},1}}+F^{i\bar{j},r\bar{s}}A_{i\bar{j},1}A_{r\bar{s},\bar{1}}\\
 & +\epsilon F^{i\bar{s}}A^{\bar{j}r}A_{i\bar{j},1}\overline{A_{s\bar{r},1}}-\frac{C_{\ref{lem:ddbar F}}}{\epsilon}F(A)+C_{\ref{lem:ddbar F}}A_{1\bar{1}}\sum_{i}F^{i\bar{i}}.
\end{align*}
\end{lem}

\begin{proof}
We have 
\begin{equation}
\pdv_{1\bar{1}}F(A)=F^{i\bar{j}}A_{i\bar{j},1\bar{1}}+F^{i\bar{j},r\bar{s}}A_{i\bar{j},1}A_{r\bar{s},\bar{1}}+2\text{Re}\left(F_{,\bar{1}}^{i\bar{j}}A_{i\bar{j},1}\right)+F_{1\bar{1}}.\label{eq:compute pdv11}
\end{equation}
Note 
\begin{equation}
F^{i\bar{j}}A_{i\bar{j},1\bar{1}}=F^{i\bar{j}}A_{1\bar{1},i\bar{j}}+F^{i\bar{j}}(\rho^{a\bar{b}}A_{a\bar{j}}R_{1\bar{1},i\bar{b}}-\rho^{a\bar{b}}A_{a\bar{1}}R_{i\bar{j},1\bar{b}}),\label{eq:compute pdv 11 2}
\end{equation}
where $R_{i\bar{j},k\bar{l}}$ is the bisectional curvature of $\rho$.
By (\ref{eq:compute pdv11}) and (\ref{eq:compute pdv 11 2}), 
\begin{align}
\pdv_{1\bar{1}}F(A) & =F^{i\bar{j}}A_{1\bar{1},i\bar{j}}+F^{i\bar{j}}\left(\rho^{a\bar{b}}A_{a\bar{j}}R_{1\bar{1},i\bar{b}}-\rho^{a\bar{b}}A_{a\bar{1}}R_{i\bar{j},1\bar{b}}\right)\label{eq:ddbar_11 F}\\
 & +F^{i\bar{j},r\bar{s}}A_{i\bar{j},1}A_{r\bar{s},\bar{1}}+2\text{Re}\left(F_{,\bar{1}}^{i\bar{j}}A_{i\bar{j},1}\right)+F_{1\bar{1}}.\nonumber 
\end{align}
The first term in (\ref{eq:ddbar_11 F}) is 
\begin{equation}
F^{i\bar{j}}A_{1\bar{1},i\bar{j}}=A_{1\bar{1}}F^{i\bar{j}}\left(\log A_{1\bar{1}}\right)_{,i\bar{j}}+A^{\bar{1}1}F^{i\bar{j}}A_{i\bar{1},1}\overline{A_{j\bar{1},1}}.\label{eq:-2}
\end{equation}
The second term in (\ref{eq:ddbar_11 F}) is controlled by 
\begin{align}
F^{i\bar{j}}\left(\rho^{a\bar{b}}A_{a\bar{j}}R_{1\bar{1},i\bar{b}}-\rho^{a\bar{b}}A_{a\bar{1}}R_{i\bar{j},1\bar{b}}\right) & \geq C_{1}F^{i\bar{j}}A_{i\bar{j}}+C_{1}A_{1\bar{1}}\sum_{i}F^{i\bar{i}}\label{eq:-3}\\
 & \geq-C_{2}F(A)+C_{2}A_{1\bar{1}}\sum_{i}F^{i\bar{i}},\nonumber 
\end{align}
where the constants $C_{1},C_{2}$ depend a bound of the bisectional
curvature $|\text{Rm}_{\rho}|$. Apply Lemma \ref{lem:tech2} to (\ref{eq:ddbar_11 F})
to obtain 
\begin{align*}
\pdv_{1\bar{1}}F(A) & \geq A_{1\bar{1}}F^{i\bar{j}}\left(\log A_{1\bar{1}}\right)_{,i\bar{j}}+A^{\bar{1}1}F^{i\bar{j}}A_{i\bar{1},1}\overline{A_{j\bar{1},1}}+F^{i\bar{j},r\bar{s}}A_{i\bar{j},1}A_{r\bar{s},\bar{1}}\\
 & +\epsilon F^{i\bar{s}}A^{\bar{j}r}A_{i\bar{j},1}\overline{A_{s\bar{r},1}}-C_{3}\left(1+\frac{1}{\epsilon}\right)F(A)+C_{3}A_{1\bar{1}}\sum_{i}F^{i\bar{i}}.
\end{align*}
We have proved the claim. 
\end{proof}
The following proposition is a key ingredient in the proof of $C^{2}$
estimate. It has been proved in several context. See Song-Weinkove
\cite{song2008convergence}, Fang-Lai-Ma \cite{Fang-Lai-Ma}, Guan
\cite{Guan2014Second-order}, Guan-Sun \cite{guan2015class}, Collins-Szèkelyhidi
\cite{collins2017convergence}, Szèkelyhidi \cite{szekelyhidi2018fully},
Datar-Pingali \cite{datar2021numerical}. The current form of Proposition
\ref{prop: Fang-Lai-Ma1} is adapted from\textbf{ }Fang-Lai-Ma \cite{Fang-Lai-Ma}
and Datar-Pingali \cite{datar2021numerical}. We will put the proof
in the appendix. 
\begin{prop}
\label{prop: Fang-Lai-Ma1} Suppose that $\Lambda$ satisfies \textbf{H2'}.
Let $\omega_{\text{sub}}\in\mathcal{C}_{\Lambda}^{\kappa}$ and $\omega=\omega_{\text{sub}}+i\ddbar u$
be a solution to (\ref{eq:eq with Berezin integ}). Then there is
a $N=N(\omega_{\text{sub}},\Lambda,M)$ and $\mu=\mu(\omega_{\text{sub}},\Lambda,M)>0$
s.t. if $|\ddbar u|_{\rho}>N$, 
\[
F^{i\bar{j}}(A)\left(u_{i\bar{j}}\right)\geq\mu\left(1-\sum_{i}F^{i\bar{i}}(A)\right).
\]
\end{prop}

Now with all the preparations, we prove the $C^{2}$-estimate. 
\begin{proof}[Proof of Proposition \ref{prop:C2 estimate}]
We use maximum principle to deduce Proposition \ref{prop:C2 estimate}.
Let $u=\varphi-\varphi_{\text{sub}}$ and $g$ be the metric tensor
of $\omega$. We may assume that $\inf_{M}u=0$, $|\ddbar u|>N$ so
that we can apply Proposition \ref{prop: Fang-Lai-Ma1}. Pick the
test function 
\begin{equation}
G(x,\xi)=\log(g_{i\bar{j}}\xi^{i}\xi^{\bar{j}})-\phi(u),\label{eq:test G}
\end{equation}
where $\xi\in T_{x}^{1,0}M,\rho_{i\bar{j}}\xi^{i}\xi^{\bar{j}}=1$.
$\phi:\R_{\geq0}\to\R$ is a smooth function: 
\begin{equation}
\phi(x)=2Lx-\frac{L\tau}{2}x^{2}.\label{eq:test phi}
\end{equation}
The choice of $\tau$ relies on $\sup u$ so that 
\begin{equation}
L\leq\phi'\leq2L,\ \phi''=-L\tau.\label{eq:test L}
\end{equation}
For instance, we may start by assuming that $\tau=\frac{1}{\sup u+1}$.
Suppose that $G(x,\xi)$ achieves maximum at $(p,\xi_{0})$. We choose
a normal coordinate $\{z^{i}\}$ of $\rho$ at $p$ so that $\xi_{0}$
is along the direction of $\frac{\pdv}{\pdv z^{1}}$ and $\omega$
is diagonal at $p$. Locally 
\begin{equation}
H=\log g_{1\bar{1}}-\phi(u)\label{eq:test H}
\end{equation}
also achieves maximum at $p$. At $p$, we have 
\begin{equation}
0=H_{,i}=g_{1\bar{1}}^{-1}g_{1\bar{1},i}-\phi'u_{,i},\label{eq:C2 est 1}
\end{equation}
\begin{equation}
0\leq F^{i\bar{j}}H_{,i\bar{j}}.\label{eq:C2 est 2}
\end{equation}

At $p$, we have 
\begin{equation}
0=H_{,i}=g_{1\bar{1}}^{-1}g_{1\bar{1},i}-\phi'u_{,i}.\label{eq:grad}
\end{equation}
We have $g_{1\bar{1},i}=g_{i\bar{1},1}$ and (\ref{eq:grad}) implies
that 
\begin{equation}
u_{,i}=\frac{g^{\bar{1}1}}{\phi'}g_{i\bar{1},1}.\label{eq:gradeqal0}
\end{equation}
Now 
\begin{equation}
H_{,i\bar{j}}=\left(\log g_{1\bar{1}}\right)_{,i\bar{j}}-\phi''u_{,i}u_{,\bar{j}}-\phi'u_{,i\bar{j}}.\label{eq:Hij}
\end{equation}
We have $g_{1\bar{1},i}=g_{i\bar{1},1}$. Thus, by Lemma \ref{lem:ddbar F},
\begin{align}
0 & \leq F^{i\bar{j}}H_{,i\bar{j}}=F^{i\bar{j}}\left(\left(\log g_{1\bar{1}}\right)_{,i\bar{j}}-\phi''u_{,i}u_{,\bar{j}}-\phi'u_{,i\bar{j}}\right)\label{eq:compute est 2}\\
 & \leq-g^{\bar{1}1}\left(F^{i\bar{j},r\bar{s}}g_{i\bar{j},1}g_{r\bar{s},\bar{1}}+F^{i\bar{j}}g^{\bar{1}1}g_{i\bar{1},1}\overline{g_{j\bar{1},1}}+\epsilon F^{i\bar{j}}g^{\bar{l}l}g_{i\bar{l},1}\overline{g_{j\bar{l},1}}\right)-\frac{C_{\ref{lem:ddbar F}}}{\epsilon}F(A)\nonumber \\
 & \ -C_{\ref{lem:ddbar F}}\sum_{i}F^{i\bar{i}}-F^{i\bar{j}}\left(\phi''u_{,i}u_{,\bar{j}}+\phi'u_{,i\bar{j}}\right).\nonumber 
\end{align}
We apply (\ref{eq:convexity in continuous path}) to (\ref{eq:compute est 2})
to get 
\begin{align}
0 & \leq g^{\bar{1}1}\left((1-\epsilon)F^{i\bar{j}}g^{\bar{l}l}g_{i\bar{l},1}\overline{g_{j\bar{l},1}}-\frac{f(p)}{\det g}\left|\sum_{j}g^{\bar{j}j}g_{j\bar{j},1}\right|^{2}-F^{i\bar{j}}g^{\bar{1}1}g_{i\bar{1},1}\overline{g_{j\bar{1},1}}\right)\label{eq:compute est 22}\\
 & +g^{\bar{1}1}\frac{C_{\ref{lem:ddbar F}}}{\epsilon}F-C_{\ref{lem:ddbar F}}\sum_{i}F^{i\bar{i}}-F^{i\bar{j}}\left(\phi''u_{,i}u_{,\bar{j}}+\phi'u_{,i\bar{j}}\right).\nonumber 
\end{align}
Substitute $u_{,i}=\frac{g^{\bar{1}1}}{\phi'}g_{i\bar{1},1}$ into
(\ref{eq:compute est 22}) to get 
\begin{align}
0 & \leq(1-\epsilon)g^{\bar{1}1}F^{i\bar{j}}g^{\bar{l}l}g_{i\bar{l},1}\overline{g_{j\bar{l},1}}-\frac{f(p)}{\det g}\left|\sum_{j}g^{\bar{j}j}g_{j\bar{j},1}\right|^{2}-F^{i\bar{j}}(g^{\bar{1}1})^{2}g_{i\bar{1},1}\overline{g_{j\bar{1},1}}\label{eq:-7}\\
 & +g^{\bar{1}1}\frac{C_{\ref{lem:ddbar F}}}{\epsilon}F(A)-\frac{\phi''}{(\phi')^{2}}F^{i\bar{j}}\left(g^{\bar{1}1}\right)^{2}g_{i\bar{1},1}\overline{g_{j\bar{1},1}}-C_{\ref{lem:ddbar F}}\sum_{i}F^{i\bar{i}}-\phi'F^{i\bar{j}}u_{,i\bar{j}}\nonumber \\
 & =g^{\bar{1}1}\left[(1-\epsilon)F^{i\bar{j}}g^{\bar{l}l}g_{i\bar{l},1}\overline{g_{j\bar{l},1}}-\frac{f(p)}{\det g}\left|\sum_{j}g^{\bar{j}j}g_{j\bar{j},1}\right|^{2}-\left(1+\frac{\phi''}{(\phi')^{2}}\right)F^{i\bar{j}}g^{\bar{1}1}g_{i\bar{1},1}\overline{g_{j\bar{1},1}}\right]\nonumber \\
 & +g^{\bar{1}1}\frac{C_{\ref{lem:ddbar F}}}{\epsilon}F-C_{\ref{lem:ddbar F}}\sum_{i}F^{i\bar{i}}-\phi'F^{i\bar{j}}u_{,i\bar{j}}.\nonumber 
\end{align}
Take $\epsilon\leq\min\{-\frac{\phi''}{(\phi')^{2}},\frac{1}{2}\}$,
then we have 
\begin{align}
(1-\epsilon)F^{i\bar{j}}g^{\bar{l}l}g_{i\bar{l},1}\overline{g_{j\bar{l},1}}-\left(1+\frac{\phi''}{(\phi')^{2}}\right)F^{i\bar{j}}g^{\bar{1}1}g_{i\bar{1},1}\overline{g_{j\bar{1},1}} & \leq\frac{1}{2}\sum_{l\geq2}F^{i\bar{j}}g^{\bar{l}l}g_{i\bar{l},1}\overline{g_{j\bar{l},1}}.\label{eq:-1}
\end{align}

If $f(p)\geq0$, then by (\ref{eq:-1}), (\ref{eq:-1}), and the fact
that $F^{i\bar{j}}\leq0$ in Lemma \ref{lem:mononton with f}, 
\begin{equation}
0\leq g^{\bar{1}1}\frac{C_{\ref{lem:ddbar F}}}{\epsilon}F-C_{\ref{lem:ddbar F}}\sum_{i}F^{i\bar{i}}-\phi'F^{i\bar{j}}u_{,i\bar{j}}.\label{eq:impor est 1}
\end{equation}
If $f(p)<0$, then by (\ref{eq:interm monot with f}), we have 
\begin{equation}
-\sum_{l\geq2}F^{i\bar{j}}g^{\bar{l}l}g_{i\bar{l},1}\overline{g_{j\bar{l},1}}\geq\left(\frac{m}{2}\gamma_{\min}(\frac{\kappa}{m},n_{p},\mathbf{d}_{p},\mathbf{k}_{p})-|f(p)|\right)\frac{1}{\det g}\sum_{l\geq2}\sum_{i=1}^{n}\frac{|g_{i\bar{l},1}|^{2}}{g_{l\bar{l}}g_{i\bar{i}}}.\label{eq:add est 2}
\end{equation}
We argue similarly as in Lemma \ref{lem:conv cont path}: 
\begin{align}
 & \frac{1}{2}\sum_{l\geq2}F^{i\bar{j}}g^{\bar{l}l}g_{i\bar{l},1}\overline{g_{j\bar{l},1}}-\frac{f(p)}{\det g}\left|\sum_{j}g^{\bar{j}j}g_{j\bar{j},1}\right|^{2}\label{eq:est impo}\\
 & \leq-\frac{1}{\det g}\left(\left(\frac{m}{4}\gamma_{\min}(\frac{\kappa}{m},n_{p},\mathbf{d}_{p},\mathbf{k}_{p})-\frac{1}{2}|f(p)|\right)\sum_{l\geq2}\sum_{i=1}^{n}\frac{|g_{i\bar{l},1}|^{2}}{g_{l\bar{l}}g_{i\bar{i}}}-|f(p)|n\sum_{j=1}^{n}\frac{|g_{j\bar{j},1}|^{2}}{g_{j\bar{j}}^{2}}\right)\nonumber \\
 & \leq\frac{n|f(p)|}{\det g}\frac{|g_{1\bar{1},1}|^{2}}{g_{1\bar{1}}^{2}}.\nonumber 
\end{align}
The last line is due to the almost positive volume condition. 

By (\ref{eq:-7}), (\ref{eq:-1}), (\ref{eq:est impo}), and (\ref{eq:impor est 1}),
we have 
\begin{align}
0 & \leq-\frac{\min\{0,f(p)\}n(g^{\bar{1}1})^{3}}{\det g}|g_{1\bar{1},1}|^{2}+g^{\bar{1}1}\frac{C_{\ref{lem:ddbar F}}}{\epsilon}F-C_{\ref{lem:ddbar F}}\sum_{i}F^{i\bar{i}}-\phi'F^{i\bar{j}}u_{,i\bar{j}}\label{eq:impor 2}\\
 & \leq C_{0}(m,\kappa,n)(g^{\bar{1}1})^{3}|g_{1\bar{1},1}|^{2}+g^{\bar{1}1}\frac{C_{\ref{lem:ddbar F}}}{\epsilon}F-C_{\ref{lem:ddbar F}}\sum_{i}F^{i\bar{i}}-\phi'F^{i\bar{j}}u_{,i\bar{j}},\nonumber 
\end{align}
where the last line is due to $\det g^{-1}\leq C_{1}(m,\kappa,n_{p},\mathbf{d}_{p},\mathbf{k}_{p})$
by Lemma \ref{lem: induced cone}. Since $u_{,i\bar{j}}=g_{i\bar{j}}-g_{i\bar{j}}^{\text{sub}},$
if $|\ddbar u|>N$, from Proposition \ref{prop: Fang-Lai-Ma1}, we
have 
\begin{align}
F^{i\bar{j}}\left(u_{,i\bar{j}}\right) & =F^{i\bar{j}}\left(g_{i\bar{j}}-g_{i\bar{j}}^{\text{sub}}\right)\label{eq:}\\
 & \geq\mu(1-\sum_{i}F^{i\bar{i}}).\nonumber 
\end{align}
Thus, use \ref{eq:impor 2}) and $g_{1\bar{1},1}=\phi'u_{,1}g_{1\bar{1}}$
to get 
\begin{align}
0 & \leq C_{1}g^{\bar{1}1}(\phi')^{2}\left(1+|\nabla u|^{2}\right)+g^{\bar{1}1}\frac{C_{\ref{lem:ddbar F}}}{\epsilon}F(A)-C_{\ref{lem:ddbar F}}\sum_{i}F^{i\bar{i}}-\phi'\mu(1-\sum_{i}F^{i\bar{i}})\label{eq:est 2 3}\\
 & \leq C_{1}g^{\bar{1}1}(\phi')^{2}\left(1+|\nabla u|^{2}\right)+g^{\bar{1}1}\frac{C_{\ref{lem:ddbar F}}}{\epsilon}F(A)-\phi'\mu+\left(\phi'\mu-C_{\ref{lem:ddbar F}}\right)\sum_{i}F^{i\bar{i}}.\nonumber 
\end{align}
We choose $\phi'\mu>C_{\ref{lem:ddbar F}}$ and then by (\ref{eq:est 2 3})
, 
\begin{equation}
0\leq C_{1}g^{\bar{1}1}(\phi')^{2}\left(1+|\nabla u|^{2}\right)+g^{\bar{1}1}\frac{C_{\ref{lem:ddbar F}}}{\epsilon}F(A)-\phi'\mu.\label{eq: est 2 4}
\end{equation}
Therefore, by (\ref{eq: est 2 4}), we have $g_{1\bar{1}}\leq C_{2}(1+|\nabla u|^{2}),$
and hence at $p$ 
\begin{equation}
|\ddbar u|<C_{3}(1+|\nabla u|^{2}),\label{eq:ddbar u smaller}
\end{equation}
for some $C_{3}=C_{3}(C_{\ref{lem:ddbar F}},L,\epsilon,\mu,m,\kappa,n)$.

Now, we fix choices of $L,\tau,\epsilon$. First, we need $\phi'\mu>C_{\ref{lem:ddbar F}}$
which requires that$L>\frac{C_{\ref{lem:ddbar F}}}{\mu}+1$. Next,
we need $\epsilon\leq-\frac{\phi''}{(\phi')^{2}}.$ Note 
\begin{equation}
-\frac{\phi''}{(\phi')^{2}}\ge\frac{L\tau}{4L^{2}}=\frac{\tau}{4L}.\label{eq:check phi}
\end{equation}
Therefore, we may pick $\epsilon=\min\{\frac{\tau}{4L},\frac{1}{2}\}=\frac{1}{4L(\sup u+1)}$.
Note $L,\epsilon,\tau$ only depend on $M,\Lambda,\omega_{\text{sub}},n,m,\kappa,$$C_{H2},\sup|u|,|\text{Rm}_{\rho}|$.

Once we have arranged $L,\epsilon,\tau$ properly, from (\ref{eq:ddbar u smaller}),
\begin{equation}
|\ddbar u|<C_{4}\sup(|\nabla u|^{2}+1),\label{eq:ddbar u}
\end{equation}
for some $C_{4}=C_{4}(M,\Lambda,\omega_{\text{sub}},n,m,\kappa,$$C_{H2},\sup|u|,|\text{Rm}_{\rho}|)$.
Finally, we apply the blow-up technique in \cite{collins20201} (Proposition
5.1) to (\ref{eq:ddbar u}) to deduce that $|\ddbar u|\leq C_{5}(M,\Lambda,\omega_{\text{sub}},n,$$m,\kappa,C_{H2},\sup|u|,|\text{Rm}_{\rho}|)$
as desired. 
\end{proof}
We now give the proof of Theorem \ref{thm:Analytic theorem}. 
\begin{proof}[Proof of Theorem \ref{thm:Analytic theorem}]
Suppose that $\Lambda$ satisfies \textbf{H2 }with uniform constant
$m>0$. From discussions above, we only need to show that $\mathbf{I}$
is closed. Since $\mathbf{I}$ is open $(0,t_{0})\subset\mathbf{I}$
for some $t_{0}>0$. Suppose that $t_{1}=\sup\{t:[0,t)\subset\mathbf{I}\}>t_{0}>0$.
Let $t\in[t_{0},t_{1}]$, and denote 
\[
\Lambda_{t}=t\mathring{\Lambda}+f_{t}P^{[n]}.
\]
Then as explained before, $\Lambda_{t}$ satisfies \textbf{H2'}. It
is straightforward to see that if $\omega_{\text{sub}}$ is a subsolution
for (\ref{eq:equation with f}), then it is a subsolution for (\ref{eq:contin path})
at all $t\in[0,1]$. So we can apply the a priori $C^{0}$, $C^{2}$
estimates derived in Proposition \ref{prop:C0 estimate} and \ref{prop:C2 estimate}.
$C^{1}$ estimate is obtained by the blow-up technique as in \cite{collins20201}
(Proposition 5.1). Together with the complex version of Evans-Krylov
theory we obtain $C^{2,\alpha}$ bounds for solutions of (\ref{eq:contin path})
for all $t\in\mathbf{I}$. The standard Schauder theory implies $C^{k,\alpha}$
bounds for every $k$. The closedness of $\mathbf{I}$ can then be
obtained from Arzela-Ascoli theorem. Thus, $\mathbf{I}$ is closed
which implies the existence of smooth solution of (\ref{eq:equation with f}).
\end{proof}

\part{Numerical Criterion}

The second part of this paper aims to give a proof of Theorem \ref{thm:numerical criterion}.
The main approach to prove Theorem \ref{thm:numerical criterion}
is an induction argument on dimension introduced first by G. Chen
\cite{chen2021j}. A key component is the so called mass concentration
technique, which was originally due to Demailly-Păun \cite{demailly2004numerical}
and has been successfully employed by Chen \cite{chen2021j} to the
$J$-equation and the supercritical dHYM equation. Our work is also
based on Song \cite{Song2020NakaiMoishezonCF}, in which Song improved
Chen's argument to treat singular sub-varieties.

The rest of this part is organized as follows. In Section \ref{sec:Set-up-for-the},
we set up the above-mentioned induction process and establish the
base case. In Section \ref{sec:Mass-concentration}, we prove a mass
concentration result, Theorem \ref{thm:Mass}. In Section \ref{sec:Completing-the-induction},
we complete the induction argument, and prove Theorem \ref{thm:numerical criterion}.
In Section \ref{sec:Application-to-Hypercritical}, we apply Theorem
\ref{thm:numerical criterion} to study special cases of dHYM equations.

\section{Notations and technical preparation}

In this section, we set up some further notations and definitions.
Assume that $M$, $[\omega_{0}]$, $\Lambda$, $\kappa$ are given
as before.

We first define regularized maximum functions. Readers may consult
\cite{demailly2012complex} I.5.E. for general discussion. 
\begin{defn}
(Regularized Maximum). For any $\eta\in(0,\infty)^{l}$, the regularized
maximum function is defined as 
\[
\widetilde{\max}_{\eta}(t_{1},\cdots,t_{l}):=\int_{\R^{l}}\max(t_{1}+h_{1},\cdots,t_{l}+h_{l})\prod_{1\leq j\leq n}\theta\left(\frac{h_{j}}{\eta_{j}}\right)\frac{dh_{1}\cdots dh_{l}}{\eta_{1}\cdots\eta_{l}},
\]
for $\eta=(\eta_{1},\cdots,\eta_{l})\in\left(0,\infty\right)^{l}$.
Here $\theta$ is a smooth non-negative function supported on $(-1,1)$
s.t. $\int_{\R}\theta(t)dt=1$ and $\int_{\R}t\theta(t)dt=0$.

We will use the regularized maximum function and Richberg's technique
\cite{Richberg1967StetigeSP} to glue local PSH functions. Some related
known facts are collected in the following lemma. See \cite{demailly2012complex}
I.5.18 for proofs. 
\end{defn}

\begin{lem}
\label{lem:reg max}For any $\eta\in(0,\infty)^{l}$, $\widetilde{\max}_{\eta}$
possesses the following properties: 
\begin{enumerate}
\item \label{enu:reg maximum 1}$\widetilde{\max}_{\eta}(t_{1},\cdots,t_{l})$
is non-decreasing in all variables, smooth and convex on $\R^{l}$; 
\item $\max\{t_{1},\cdots,t_{l}\}\leq\widetilde{\max}_{\eta}(t_{1},\cdots,t_{l})\leq\max\{t_{1}+\eta_{1},\cdots,t_{l}+\eta_{l}\}$; 
\item \label{enu:important feature}if $\ t_{j}+\eta_{j}\leq\max_{k\not=j}\{t_{k}-\eta_{k}\}$,
then 
\[
\widetilde{\max}_{\eta}(t_{1},\cdots,t_{l})=\widetilde{\max}_{(\eta_{1},\cdots,\hat{\eta}_{j},\cdots,t_{l})}(t_{1},\cdots,\hat{t}_{j},\cdots,t_{l});
\]
\item \label{enu:.reg maximum 4}$\widetilde{\max}_{\eta}(t_{1}+a,\cdots,t_{l}+a)=\widetilde{\max}_{\eta}(t_{1},\cdots,t_{l})+a$. 
\end{enumerate}
\end{lem}

We prove the following technical lemma. 
\begin{lem}
\label{cor:reg max preserves subsolution}If $g=g(A_{1},\cdots,A_{l})$
is a convex function on the spaces of Hermitian matrices and is monotone
decreasing in each $A_{i}$, then 
\begin{equation}
g(\sqrt{-1}\ddbar\widetilde{\max}_{\eta}(u_{1},\cdots,u_{l}))\leq\sum_{i}\frac{\pdv\widetilde{\max}_{\eta}}{\pdv u_{i}}g(\sqrt{-1}\ddbar u_{i}).\label{eq:reg max pres sub}
\end{equation}
In particular, if $u_{1},\cdots,u_{l}$ are in $\text{PSH}(M,\omega_{0})$
and satisfies $\mathcal{P}_{\Lambda}(\omega_{0}+\sqrt{-1}\ddbar u_{i})<\kappa,$
then 
\begin{equation}
\mathcal{P}_{\Lambda}(\omega_{0}+\sqrt{-1}\ddbar\widetilde{\max}_{\eta}(u_{1},\cdots,u_{l}))\leq\mathcal{P}_{\Lambda}(\omega_{0}+\sqrt{-1}\ddbar u_{i})<\kappa.\label{eq:reg max pres sub-1}
\end{equation}
\end{lem}

\begin{proof}
Suppose that $u_{1},\cdots,u_{l}$ are $C^{2}$ functions on $\C^{n}$.
Then 
\begin{align}
\sqrt{-1}\ddbar\widetilde{\max}_{\eta}(u_{1},\cdots,u_{l}) & =\sum_{i}\frac{\pdv\widetilde{\max}_{\eta}}{\pdv u_{i}}\sqrt{-1}\ddbar u_{i}+\sum_{i,j}\frac{\pdv^{2}\widetilde{\max}_{\eta}}{\pdv u_{i}\pdv u_{j}}\sqrt{-1}\pdv u_{j}\wedge\dbar u_{i}.\label{eq:reg max deriv}
\end{align}
From Lemma \ref{lem:reg max} (\ref{enu:.reg maximum 4}), we see
that $\sum_{i}\frac{\pdv\widetilde{\max}_{\eta}}{\pdv u_{i}}=1$ and
each $\frac{\pdv\widetilde{\max}}{\pdv u_{i}}$ is non-negative. Thus,
from (\ref{eq:reg max deriv}), 
\[
\sqrt{-1}\ddbar\widetilde{\max}_{\eta}(u_{1},\cdots,u_{l})\geq\sum_{i}\frac{\pdv\widetilde{\max}_{\eta}}{\pdv u_{i}}\sqrt{-1}\ddbar u_{i}.
\]
The right hand side is a convex combination of $\sqrt{-1}\ddbar u_{i}$.
Thus, (\ref{eq:reg max pres sub}) holds by the monotonicity and the
convexity of $g$. (\ref{eq:reg max pres sub-1}) holds by the corresponding
properties of $\mathcal{P}_{\Lambda}$ in Lemma \ref{lem:P_La property}. 
\end{proof}
We use Richberg technique (\cite{demailly2012complex}, I. Corollary
5.19) to patch up PSH functions on manifolds. 
\begin{cor}
\label{cor:re max}Let $u_{\alpha}\in C^{\infty}(\overline{U}_{\alpha})\cap PSH(U_{\alpha},\omega)$
where $U_{\alpha}\subset\subset M$ is a finite open covering of $M$.
Assume that $u_{\beta}<\max\{u_{\alpha}(z)\}$ at every point $z\in\pdv U_{\beta}$
when $\alpha$ runs over the indices s.t. $z\in U_{\alpha}$. Choose
a family $\{\eta_{\alpha}\}$ of positive numbers so small that $u_{\beta}(z)+\eta_{\beta}\leq\max_{U_{\alpha}\ni z}\{u_{\alpha}-\eta_{\alpha}\}$
for all $\beta$ s.t. $z\in\pdv U_{\beta}$. Then the function 
\[
\tilde{u}(z)=\widetilde{\max}_{(\eta_{\alpha})}(u_{\alpha}(z))
\]
is in $C^{\infty}(M)\cap\text{PSH}(M,\omega)$. 
\end{cor}

Next, we recall some well known definitions in complex geometry. 
\begin{defn}
Let $T$ be a closed positive $(1,1)$-current on $M$. We call $T$
a $\kah$ current if 
\[
T-\epsilon\gamma\geq0,
\]
for some $\epsilon>0$ and $\gamma$ is a Hermitian metric on $M$. 
\end{defn}

We recall the definition of Lelong number: 
\begin{defn}
For $p\in B_{R}\subset V\subset\C^{d}$, and $\varphi\in\text{PSH}(V)$,
define 
\[
\nu_{\varphi}(p,r)=\frac{\bar{\varphi}(p,R)-\bar{\varphi}(p,r)}{\log R-\log r}
\]
where $0<r<R$, and $\bar{\varphi}(p,r)=\sup_{B_{r}(p)}\varphi(z)$.
The Lelong number of $\varphi$ at $p$ is given by 
\[
\nu_{\varphi}(p)=\lim_{r\to0^{+}}\nu_{\varphi}(p,r).
\]
The Lelong number of a closed positive $(1,1)$-current $T$, denoted
as $\nu_{T}(p)$, is defined to be the Lelong number of the local
potential function. 
\end{defn}

The classic result of Y.-T. Siu shows that the Lelong number is upper
semi-continuous with respect to analytic Zariski topology. 
\begin{thm}[Siu's semi-continuity theorem \cite{Siu1974AnalyticityOS}]
If $T$ is a closed positive $(1,1)$-current on $M$, then the upper
level sets 
\[
E_{c}(T)=\{p:\nu_{T}(p)\geq c\}
\]
are analytic subsets of $M$ of dimension $\leq n-1$. 
\end{thm}

Finally we define the local regularization of a current.
\begin{defn}
\label{def:Local-regularization of current} Let $T$ be a closed
positive $(1,1)$-current on a smooth variety $Y$ of dimension $d$
and $R>0$. We call $T^{(r)}=\{T_{j}^{(r)}\}_{j\in\mathcal{J}}$ a
local regularization of $T$ with respect to a finite open covering
$\mathscr{P}=\{B_{j,3R}\}_{j\in\mathcal{J}}$ of $Y$ if the following
conditions (1) and (2) are satisfied. 
\begin{enumerate}
\item Each $B_{j,3R}$ is biholomorphic to a Euclidean ball $B_{3R}(0)$
in $\C^{d}$ equipped with standard Euclidean metric $g_{j}$. Furthermore,
$B_{j,R}\simeq B_{R}(0)\subset\C^{d}$ is also a covering of $Y$; 
\item $T_{j}^{(r)}(z)$ is the standard smoothing of $T$ in $B_{j,2R}$
defined by 
\[
T_{j}^{(r)}(z)=\int_{B_{r}(0)}r^{-2m}\vartheta\left(\frac{z'}{|r|}\right)T(z-z')dV_{\C^{n}}(z')
\]
for $r\in(0,R)$. $\vartheta(t)$ is a smooth non-negative function
with support in $[0,1]$ satisfying $\vartheta\equiv\text{const}$
in $[0,1/2]$ and a normalization condition: 
\[
\int_{B_{1}(0)}\vartheta(|z'|)dV_{\C^{n}}(z')=1.
\]
\end{enumerate}
If in each $B_{3R}$, $T=\sqrt{-1}\ddbar\varphi_{i}$ for some local
PSH function $\varphi_{i}$. Then it is easy to see that for $r\in(0,R)$,
\[
T_{j}^{(r)}=\sqrt{-1}\ddbar\varphi_{i}^{(r)},
\]
where 
\[
\varphi_{i}^{(r)}(z)=\int_{B_{r}(0)}r^{-2m}\vartheta\left(\frac{z'}{|r|}\right)\varphi_{i}(z-z')dV_{\C^{n}}(z').
\]
\end{defn}

\section{Initiation of induction argument\label{sec:Set-up-for-the}}

In this section, we state the main technical theorem of this part
and initiate an induction proof. Several technical issues including
the resolution of singularities are also discussed.

The following theorem is the main goal of the rest of this part:
\begin{thm}
\label{thm:cone condition in a neighborhood }Let $M$ be a $n$-dimensional
connected compact Kähler manifold. Let $\kappa>0$ be a constant.
Let $\Lambda$ be a close real form satisfies \textbf{H1}. If $[\omega_{0}]$
is $([\Lambda],\kappa)$-positive, then for any analytic subvariety
$Y$ with dimension $d\leq n$, there is a neighborhood $U_{Y}$ of
$Y$ in $M$ and a Kähler form $\omega_{U_{Y}}\in[\omega_{0}]|_{U_{Y}}$
satisfies the cone condition (\ref{eq:cone condition}) on $U_{Y}$. 
\end{thm}

Without loss of generality, by a rescaling of $\Lambda$, we may assume
that $\kappa=1$ in the rest of this part. Then the corresponding
cone condition is 
\begin{equation}
((1-\Lambda)\wedge\exp\omega)^{[n-1]}>0\quad\text{or}\quad\mathcal{P}_{\Lambda}(\omega)<1.\label{eq:recal kapp}
\end{equation}

Note Theorem \ref{thm:numerical criterion}, presented in our introduction,
follows immediately from Theorem \ref{thm:cone condition in a neighborhood }
and Theorem \ref{thm:Analytic theorem}. 
\begin{rem}
If $Y$ is a smooth subvariety of $M$, Theorem \ref{thm:cone condition in a neighborhood }
implies that $\omega_{Y}=\omega_{U}|_{Y}$ satisfies the cone condition
on $Y$:
\begin{equation}
(\exp\omega_{Y})^{[\dim Y-1]}>(\Lambda|_{Y}\wedge\exp\omega_{Y})^{[\dim Y-1]}\quad\text{or}\quad\mathcal{P}_{\Lambda|_{Y}}(\omega_{Y})<1.\label{eq: restricted}
\end{equation}
On the other hand, (\ref{eq: restricted}) also implies the existence
of a subsolution in a neighborhood $U$ of a smooth subvariety $Y$.
This is stated in the following lemma. 
\end{rem}

\begin{lem}
\label{lem: construct a local extension}Notations as above. Suppose
that $Y$ is a smooth subvariety of $M$ and $\omega_{Y}\in[\omega_{0}|_{Y}]$
satisfies the cone condition (\ref{eq: restricted}) on $Y$. Then
there exists a neighborhood $U$ of $Y$ in $M$ and a Kähler form
$\omega_{U}\in[\omega_{0}|_{U}]$ such that $\omega_{U}|_{Y}=\omega_{Y}$
and $\omega_{U}$ satisfies the cone condition (\ref{eq:recal kapp})
in $U$. 
\end{lem}

\begin{proof}
Let $U_{1}$ be a tubular neighborhood of $Y$ in $M$ such that he
projection map $\text{pr}_{Y}:U_{1}\to Y$ is well defined. Let $N>0$
and define 
\[
\omega_{U}=\omega_{0}+i\ddbar\left(\text{pr}_{Y}^{*}\varphi+Nd_{\rho}^{2}(\cdot,Y)\right).
\]
Clearly, there exists $N>0$ and a neighborhood $U_{2}\subset\subset U_{1}$
of $Y$ such that $\omega_{U}$ is $\kah$ in $U_{2}$. At a point
$p\in Y$, we choose a local normal coordinate $\{z^{i}\}$ with respect
to $\rho$ such that $\frac{\pdv}{\pdv z^{1}},\cdots,\frac{\pdv}{\pdv z^{d}}$
are tangential to $Y$ at $p$, $\frac{\pdv}{\pdv z^{d+1}},\cdots,\frac{\pdv}{\pdv z^{n}}$
are orthogonal to $Y$, and $\omega_{U}=A_{i\bar{j}}\frac{\sqrt{-1}}{2}dz^{i}\wedge d\bar{z}^{j}$
with 
\[
A=\left(\begin{array}{cc}
H & C\\
C^{\dagger} & V
\end{array}\right).
\]
At $p$, $H_{i\bar{j}}\frac{\sqrt{-1}}{2}dz^{i}\wedge d\bar{z}^{j}=\pi^{*}\omega_{Y}$,
$V_{i\bar{j}}\ge N\delta_{i\bar{j}}$, and $C=0$. In a neighborhood
of $p$, we have $C=O(d_{\rho})$ and $V>\frac{N}{2}\text{Id}_{(n-d)\times(n-d)}$.
By the continuity of $\mathcal{P}_{\Lambda}$ in Lemma \ref{lem:P_La property},
we only need to verify the cone condition (\ref{eq:recal kapp}) at
$p$. Notice that at $p$, if $N\to\infty$, $A^{-1}$ converges to
$H^{-1}$ uniformly. By (\ref{eq: restricted}), there exists $\epsilon_{Y}>0$
such that for any linear subspace $\mathcal{H}'\subset\mathcal{T}_{p}Y$,
it holds 
\[
1-\epsilon_{Y}>\<\Lambda,\exp\chi_{\mathcal{H}'}\>,
\]
where $\chi_{\mathcal{H}}=\left(H|_{\mathcal{H}'}\right)^{\bar{j}i}2\sqrt{-1}\frac{\pdv}{\pdv\bar{z}^{j}}\wedge\frac{\pdv}{\pdv z^{i}}$.
Therefore, there exists $N>>1$ s.t. $\mathcal{P}_{\Lambda}(A)<1-\frac{\epsilon_{Y}}{2}$.
Hence, by the continuity of $\mathcal{P}_{\Lambda}$ and the compactness
of $Y$, we may pick a uniform $N$, a uniform $\epsilon_{Y}$, and
a neighborhood $U\subset\subset U_{2}$ of $Y$ s.t. $\mathcal{P}_{\Lambda}(\omega_{U})<1-\frac{\epsilon_{Y}}{4}$
in $U$, which verifies the cone condition (\ref{eq:recal kapp}). 
\end{proof}
We prove Theorem \ref{thm:cone condition in a neighborhood } by induction
on the dimension of $Y$. We start with the base case.

\subsection{Base case for induction}

Since we have assumed the \textbf{H1} condition, $\mathring{\Lambda}$
is $k_{0}$-UP. Clearly, for any smooth subvariety $Y$ of dimension
$d>k_{0}$, $\mathring{\Lambda}|_{Y}$ is also $k_{0}$-UP. To initiate
the induction argument, we need to show that Theorem \ref{thm:cone condition in a neighborhood }
is valid for any subvareity $Y$ with dimension $d\leq k_{0}$. The
following lemma generalizes Song's Lemma 2.1 in \cite{Song2020NakaiMoishezonCF}. 
\begin{lem}
\label{lem:base lemma}Notations as above. Suppose that $\mathring{\Lambda}$
is $k_{0}$-UP. Let $Y$ be an analytic subvariety of $M$ with $\dim Y\leq k_{0}$.
If $[\omega_{0}]$ is $(\kappa,[\Lambda])$-positive, then there exists
a neighborhood $U_{Y}$ of $Y$ in $M$ and a Kähler metric $\omega_{U_{Y}}\in[\omega_{0}]|_{U_{Y}}$
such that the cone condition (\ref{eq:cone condition}) holds for
$\omega_{U_{Y}}$ in $U_{Y}$.
\end{lem}

\begin{proof}
Case 1. $\dim Y=0$, the result is obvious since any Kähler class
in a neighborhood of a point is trivial.

We may assume fron now on that there exists $d\geq0$ such that Lemma
\ref{lem:base lemma} holds for any $d'$ dimensional subvariety $Y'\subset M$
where $d'\leq d$. Let $Y$ be an subvariety of dimenison $d\leq k_{0}$.

Case 2. $1\leq d\leq k_{0}.$ Let $\Lambda'=\Lambda^{[d]}$ if $k_{0}=d$
or $\Lambda'=0$ if $k_{0}>d$.

Note that $([\Lambda],\kappa)$-positivity implies the following
\[
\int_{Y}(1-\Lambda')\wedge\exp\omega>0.
\]
 Let $S_{Y}$ be the singular set of $Y$. Let $\Phi:M'\to M$ be
the resolution of singularities of $Y$ by successive blowups along
smooth centers. Let $\hat{Y}$ be the strict transform of $Y$. Denote
the exceptional divisor $E_{\Phi}$. Let $\sigma$ be a defining section
of the line bundle $[E_{\Phi}]$ and $h$ be a hermitian metric on
$[E_{\Phi}]$. Let $F_{h}$ be the curvature form on $h$. Then there
is a small $\delta>0$ such that $\varpi_{Y}:=\omega_{Y}-\delta F_{h}$
is a $\kah$ metric on $\Phi^{-1}(W)$. Pick a small $\epsilon$ such
that 
\begin{equation}
\int_{Y}\exp\omega-(1+2\epsilon)\Lambda'>0.\label{eq:Lemma 9.4}
\end{equation}
By choosing a smaller $\delta$ if necessary, we may assume that 
\begin{equation}
\int_{\hat{Y}}\exp\varpi-(1+\epsilon)\Phi^{*}\Lambda'>0.\label{eq:Lemma 9.4 1}
\end{equation}
On $\hat{Y}$, we may solve the following \MA equation 
\begin{equation}
(\varpi_{Y}+\sqrt{-1}\ddbar u)^{d}=(1+\epsilon)\Phi^{*}\Lambda'+c\varpi_{Y}^{d},\label{eq:Ma equa}
\end{equation}
for some constant $c>0$. In $M'\backslash E_{\Phi}$, we may write
$-F_{h}=\sqrt{-1}\ddbar\log|\sigma|_{h}^{2}$. Thus, there exists
$\varphi_{Y}\in C^{\infty}(W\backslash S_{Y})\cap\text{PSH}(W,\omega_{0})$
such that
\[
(\omega_{0}+\sqrt{-1}\ddbar\varphi_{Y})^{d}-(1+\epsilon)\Lambda'>0,
\]
as a $(d,d)$ form away from $S_{Y}$; Furthermore, the Lelong number
of $\varphi_{Y}$ at $S_{Y}$ is larger than $\delta$.

By the induction hypothesis, there exists a neighborhood $U$ of $S_{Y}$
in which there is a smooth $\kah$ metric $\omega_{U}=\omega_{0}+\sqrt{-1}\ddbar\varphi_{U}$
satisfying the condition $((1-\Lambda)\wedge\exp\omega_{U})^{[n-1]}>0$.
We pick neighborhoods $U_{0}\Subset U_{1}\Subset U_{2}\Subset U$
of $p$. 
\begin{figure}
\includegraphics[width=6cm]{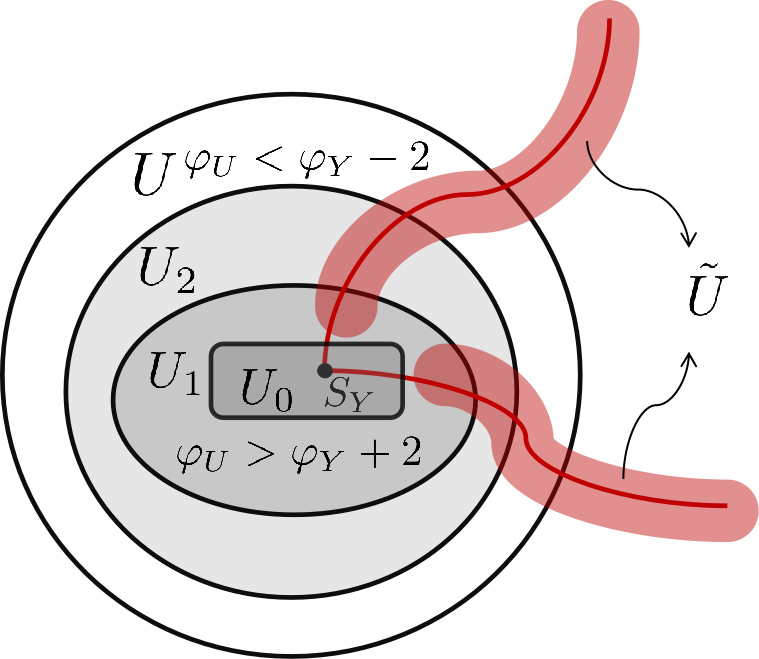}

\caption{Neighborhoods}
\end{figure}

Without loss of generality, we may subtract a large number from $\varphi_{U}$
such that $\varphi_{U}<\varphi_{Y}-2$ in $W\backslash U_{2}$. Since
$\varphi_{Y}$ diverges to $-\infty$ at $p$, shrinking $U_{1}$
if necessary, we may assume that in $U_{1}$, $\varphi_{Y}+2<\varphi_{U}$
in $U_{1}$.

For a point $z\in W\backslash U_{0}$, let $d_{\rho}(z)$ be the distant
function to $Y\cap(W\backslash U_{0})$ with respect to $\rho$. Let
\begin{equation}
\tilde{\varphi}_{Y}=\varphi_{Y}+Nd_{\rho}^{2}.\label{eq:def tilde varphi}
\end{equation}
By the same argument as in Lemma \ref{lem: construct a local extension},
for sufficiently large $N$, $\tilde{\omega}=\omega_{0}+\sqrt{-1}\ddbar\tilde{\varphi}_{Y}$
is a $\kah$ form in a neighborhood $\tilde{U}$ of $Y\cap(W\backslash U_{1})$
and satisfies the cone condition (\ref{eq:recal kapp}) on $\tilde{U}.$
Shrinking $\tilde{U}$ if necessary, we may assume that in $\tilde{U}\cap U_{1}$
the follow holds:
\begin{align}
\tilde{\varphi}_{Y} & <\varphi_{U}-1,\ \text{in}\ \tilde{U}\cap U_{1};\ \tilde{\varphi}_{Y}>\varphi_{U}+1,\ \text{in }\tilde{U}\backslash U_{2}.\label{eq:tildevarphi}
\end{align}
Let 
\[
\tilde{\varphi}=\widetilde{\max}_{(1/2,1/2)}(\tilde{\varphi}_{Y},\varphi_{U})
\]
in $U_{Y}=\tilde{U}\cup U_{1}$. By the above construction, we have
on $U_{1}\cap\tilde{U}$ $\varphi_{U}-\frac{1}{2}>\tilde{\varphi}_{Y}+\frac{1}{2}$
and on $\tilde{U}\backslash U_{2}$, $\tilde{\varphi}_{Y}-\frac{1}{2}>\varphi_{U}+\frac{1}{2}$.
Therefore, by Richberg's technique (Corollary \ref{cor:re max}),
$\tilde{\varphi}$ is smooth and in $\text{PSH}(U_{Y},\omega_{0})$.
Moreover, by Corollary \ref{cor:reg max preserves subsolution}, $\widetilde{\max}$
preserves the subsolution. Thus, $\omega_{0}+\sqrt{-1}\ddbar\tilde{\varphi}$
satisfies the cone condition. 
\end{proof}
Now that we have established the base case for the induction argument,
from now on, we may assume $\dim Y>k_{0}$ in the later discussions.
We state our induction hypothesis.

\textbf{Induction Hypothesis}: There exists $d\in\N$ and $d>k_{0}$
such that for any subvariety $Y$ if $\dim Y\leq d-1\leq n-1$, then
there exists a neighborhood $U$ of $Y$ in $M$ and a $\kah$ metric
$\omega_{U}\in[\omega_{0}|_{U}]$ such that $\omega_{U}|_{Y}=\omega_{Y}$
and $\omega_{U}$ satisfies the cone condition (\ref{eq:recal kapp})
in $U$.

\subsection{Resolution of singularities}

Let $Y$ be a subvariety of $M$ with $\dim Y=d>k_{0}$. Assume that
$Y$ is irreducible for simplicity. Apply the resolution of singularities
for $Y$ to get
\begin{equation}
\Phi:M'\to M,\label{eq:Phidef}
\end{equation}
where $\Phi$ is achieved by a sequence of blow-ups along smooth centers.
Let $\hat{Y}$ be the strict transform of $Y$ in $M'$ which is a
smooth $d$-dimensional submanifold. Let $S_{Y}$ be the singular
set of $Y$ and 
\begin{equation}
S_{\hat{Y}}=\Phi^{-1}(S_{Y})\cap\hat{Y}.\label{eq:Sing Y}
\end{equation}
Since $Y$ is irreducible, $Y\backslash S_{Y}=\Phi(\hat{Y}\backslash S_{\hat{Y}})$.

Let $S$ be the exceptional locus of $\Phi$, $h_{S}$ be a hermitian
metric on the line bundle $[S]$ associated to $S$, and $F_{h_{S}}$
be the curvature. Then for some small $\delta_{S}>0$,$\varpi:=\rho-\delta_{S}F_{h_{S}}$
is $\kah$ on $M'$ for some small $\delta_{S}>0$. We may further
assume that 
\begin{equation}
\varpi=\rho-\delta_{S}F_{h_{S}}>\frac{\rho}{2}.\label{eq:varpi bigger than rho}
\end{equation}
If $\sigma_{S}$ is a defining section of the line bundle $[S]$,
and $\phi_{S}=\delta_{S}\log|\sigma_{S}|_{h_{S}}^{2}$ then on $M'\backslash S$
we have 
\[
\varpi=\rho+\sqrt{-1}\ddbar\phi_{S}.
\]
Notice $\phi_{S}$ has positive Lelong number along $S$.

We need to perturb $[\omega_{0}]$ and $[\Lambda]$ to obtain strict
$(\kappa,[\Lambda])$-positivity on $\hat{Y}$. For this purpose,
we define the following: 
\begin{equation}
\hat{\omega}_{0}=\hat{\omega}_{0}(t,\hat{\epsilon})=\left(1+Kt\right)\omega_{0}+\hat{\epsilon}t\varpi,\ \hat{\Lambda}=\hat{\Lambda}(t,\hat{\epsilon})=\Lambda+\hat{\epsilon}^{n}t^{n}\frac{\varpi^{k_{0}}}{k_{0}!},\label{eq:perturbed classes}
\end{equation}
\begin{equation}
\hat{\rho}=\hat{\rho}(t,\hat{\epsilon})=\rho+\hat{\epsilon}^{n}t^{n}\varpi.\label{eq:perturbed rho}
\end{equation}
Here $K>1$ is chosen large enough such that 
\begin{equation}
\Lambda\leq K\left(\exp(\frac{\omega_{0}}{2n})-1\right).\label{eq:choice of K}
\end{equation}
The following result shows that $[\hat{\omega}_{0}]$ is $(\kappa,[\hat{\Lambda}])$-positive. 
\begin{lem}[Song Lemma 4.1]
There is a small $\hat{\epsilon}_{Y}$ s.t. if $\hat{\epsilon}\in(0,\hat{\epsilon}_{Y}),\ t\in(0,1]$
then 
\begin{equation}
\int_{\hat{Y}}e^{\hat{\omega}_{0}}\wedge\left(1-\hat{\Lambda}\right)>0,\label{eq:Song 4.1 1}
\end{equation}
and for any subvariety $V'$ of $\hat{Y}$ with dimension $k<d$,
it holds 
\begin{equation}
\int_{V'}e^{\hat{\omega}_{0}}\wedge\left(1-\hat{\Lambda}\right)>\frac{1}{2}\int_{V'}\exp(\hat{\epsilon}t\varpi).\label{eq:Song 4.1 2}
\end{equation}
\end{lem}

\begin{proof}
For $\hat{Y}$, we have 
\begin{align}
\int_{\hat{Y}}e^{\hat{\omega}_{0}}\wedge\left(1-\hat{\Lambda}\right) & =\int_{\hat{Y}}e^{(1+Kt)\omega_{0}+\hat{\epsilon}t\varpi}\wedge(1-\Lambda)-\hat{\epsilon}^{n}t^{n}\int_{\hat{Y}}e^{\hat{\omega}_{0}}\wedge\frac{\varpi^{k_{0}}}{k_{0}!}\label{eq:est int}\\
 & \geq(1+Kt)^{d}\int_{Y}e^{\omega_{0}}\wedge(1-\Lambda)+O(\hat{\epsilon}(1+t)^{d}).\nonumber 
\end{align}
Since $\int_{Y}e^{\omega_{0}}\wedge(1-\Lambda)>0$, there is a $\hat{\epsilon}_{Y}'>0$
so that if $\hat{\epsilon}<\hat{\epsilon}_{Y}'$ then (\ref{eq:Song 4.1 1})
holds.

If $k<d$, we may assume $V'$ is a irreducible. Let $W=\Phi(V')$.
By induction hypothesis, there is a $\kah$ form $\omega_{W}$ in
a neighborhood $U_{W}$ of $W$ in $M$ so that $\omega_{W}\in[\omega_{0}]|_{W}$
and 
\[
\left(\exp\omega_{W}\wedge\left(1-\Lambda\right)\right)^{[n-1]}>0.
\]
Then by Lemma \ref{lem: submatrix and restricting to subspac}, for
any $l\leq n-1$, 
\[
(\exp\omega_{W}\wedge(1-\Lambda))^{[l]}\geq0,
\]
which implies 
\begin{equation}
(\exp\omega_{W}\wedge\left(1-\Lambda\right))^{[l]}\wedge\varpi^{k-l}\geq0,\label{eq:int est int}
\end{equation}
in $U'_{W}=\Phi^{-1}(W)$ in $M'$ for $l\leq k$. Let 
\begin{equation}
\hat{\omega}_{1}:=(1+Kt)\omega_{W}+\hat{\epsilon}t\varpi.\label{eq:OMEGahat 1}
\end{equation}
Then by (\ref{eq:int est int}), 
\begin{align}
\int_{V'}e^{\hat{\omega}_{1}}\wedge\left(1-\hat{\Lambda}\right) & =\int_{V'}e^{(1+Kt)\omega_{W}}\wedge e^{\hat{\epsilon}t\varpi}\wedge(1-\Lambda-\hat{\epsilon}^{n}t^{n}\frac{\varpi^{k_{0}}}{k_{0}!})\label{eq:tech song 4.1}\\
 & =\int_{V'}e^{(1+Kt)\omega_{W}}\wedge e^{\hat{\epsilon}t\varpi}\wedge(1-\Lambda)-\hat{\epsilon}^{n}t^{n}\int_{V'}e^{(1+Kt)\omega_{W}}\wedge e^{\hat{\epsilon}t\varpi}\wedge\frac{\varpi^{k_{0}}}{k_{0}!}\nonumber \\
 & \geq\int_{V'}\frac{(\hat{\epsilon}t)^{k}}{k!}\varpi^{k}-\sum_{a=0}^{k-k_{0}}R(a)\int_{V'}\frac{\omega_{W}^{a}}{a!}\wedge\frac{\varpi^{k-a}}{(k-k_{0}-a)!k_{0}!},\nonumber 
\end{align}
where $R(a)=(1+Kt)^{a}(\hat{\epsilon}t)^{n+k-k_{0}-a}$. There is
a uniform constant $C_{1}$, s.t. $C_{1}[\varpi]\geq[\Phi^{*}\omega_{0}]$.
Thus, 
\begin{equation}
\int_{V'}e^{\hat{\omega}_{1}}\wedge\left(1-\hat{\Lambda}\right)\geq\left(\frac{(\hat{\epsilon}t)^{k}}{k!}-C_{2}\sum_{a=0}^{k-k_{0}}R(a)\right)\int_{V'}\varpi^{k}.\label{eq:tech song 4.1-1}
\end{equation}
where $C_{2}=C_{2}(C_{1},k,k_{0})$. If $k<k_{0}$, we have 
\begin{equation}
\int_{V'}e^{\hat{\omega}_{1}}\wedge\left(1-\hat{\Lambda}\right)>\frac{(\hat{\epsilon}t)^{k}}{k!}\int_{V'}\varpi^{k}.\label{eq:st int}
\end{equation}
If $k\geq k_{0}$, the power of $\hat{\epsilon}$ in $R(a)$ is greater
than $k$ and the power of $t$ is bigger than $n$. Hence, there
is an $\hat{\epsilon}_{Y}''=\hat{\epsilon}_{Y}''(k,n,C_{2},K)$ such
that if $\hat{\epsilon}<\hat{\epsilon}_{Y}''$, 
\begin{equation}
R(a)<\frac{1}{2kC_{2}}\frac{(\hat{\epsilon}t)^{k}}{k!}.\label{eq:tech song 4.1-2}
\end{equation}
Substituting (\ref{eq:tech song 4.1-2}) into (\ref{eq:tech song 4.1-1})
yields (\ref{eq:Song 4.1 2}). Finally, choose $\hat{\epsilon}_{Y}=\min(\hat{\epsilon}_{Y}',\hat{\epsilon}_{Y}'')$
and we have finished the proof. 
\end{proof}
We proceed to check the $k_{0}$-UP condition for $\hat{\Lambda}$. 
\begin{lem}
\label{lem:K0 up for hat La}If $\mathring{\Lambda}$ is $k_{0}$-UP
with respect to $\rho$ and a uniform constant $m>0$. Then $\mathring{\hat{\Lambda}}$
is $k_{0}$-UP with respect to $\hat{\rho}$ and a constant 
\[
m'=\min\{m,k_{0}\left(\frac{1}{2}+(\hat{\epsilon}t)^{n}\right)^{k_{0}-1}\}.
\]
\end{lem}

\begin{proof}
By \ref{eq:varpi bigger than rho}, we have assumed $\varpi>\frac{\rho}{2}$.
Thus 
\begin{align}
\frac{\hat{\rho}^{k_{0}}}{k_{0}!} & =\frac{(\rho+\hat{\epsilon}^{n}t^{n}\varpi)^{k_{0}}}{k_{0}!}\label{eq:hat rho ineq}\\
 & \leq\frac{\rho^{k_{0}}}{k_{0}!}+(\hat{\epsilon}t)^{n}\frac{\varpi^{k_{0}}}{k_{0}!}\left(\sum_{a=0}^{k_{0}-1}\frac{(\hat{\epsilon}t)^{n(k_{0}-1-a)}k_{0}!}{2^{a}a!(k_{0}-a)!}\right)\nonumber \\
 & \leq\frac{\rho^{k_{0}}}{k_{0}!}+(\hat{\epsilon}t)^{n}k_{0}\left(\frac{1}{2}+(\hat{\epsilon}t)^{n}\right)^{k_{0}-1}\frac{\varpi^{k_{0}}}{k_{0}!}.\nonumber 
\end{align}
Since $\mathring{\Lambda}\geq m\frac{\rho^{k_{0}}}{k_{0}!}$, by (\ref{eq:hat rho ineq}),
\begin{align*}
\hat{\Lambda}-\hat{\Lambda}^{[n]} & =\mathring{\Lambda}+(\hat{\epsilon}t)^{n}\frac{\varpi^{k_{0}}}{k_{0}!}\geq m\frac{\rho^{k_{0}}}{k_{0}!}+(\hat{\epsilon}t)^{n}\frac{\varpi^{k_{0}}}{k_{0}!}\\
 & \geq m'\frac{\hat{\rho}^{k_{0}}}{k_{0}!}.
\end{align*}
We have proved the Lemma. 
\end{proof}

\subsection{Equations on $\hat{Y}$}

Next, we consider the following equation on $\hat{Y}:$ 
\begin{equation}
(\exp\hat{\omega}\wedge(1-\hat{\Lambda}))^{[d]}=c_{t,\hat{\epsilon}}(\exp\hat{\rho})^{[d]},\label{eq:equation on hat Y}
\end{equation}
for $t>0$, $\hat{\epsilon}\in(0,\hat{\epsilon}_{Y})$, and $\hat{\omega}\in[\hat{\omega}_{0}]$.
Here $c_{t,\hat{\epsilon}}$ is a normalization constant defined s.t.
\[
\int_{\hat{Y}}(\exp\hat{\omega}\wedge(1-\hat{\Lambda}))^{[d]}=\int_{\hat{Y}}c_{t,\hat{\epsilon}}(\exp\hat{\rho})^{[d]}.
\]
We may choose a smaller $\hat{\epsilon}_{Y}$ if necessary, such that
$\hat{\epsilon}_{Y}<\left(\frac{1}{4n!}\right)^{\frac{1}{n-k_{0}}}$.

The following lemma implies that $\hat{\omega}_{0}$ is a subsolution
of (\ref{eq:equation on hat Y}) for $t=1$. 
\begin{lem}
\label{lem:solution t equal 1}If $K>1$ is chosen as in (\ref{eq:choice of K}),
we have 
\begin{equation}
e^{(1+K)\omega_{0}}\wedge(1-\Lambda)\geq\frac{1}{2}e^{(1+K)\omega_{0}}\label{eq:solut tequal 1}
\end{equation}
on $M$. Moreover, on $\hat{Y}$, for $t=1$ and $\hat{\epsilon}<\left(\frac{1}{4n!}\right)^{\frac{1}{n-k_{0}}}$,
we have
\begin{equation}
\left(e^{\hat{\omega}_{0}}\wedge(1-\hat{\Lambda})\right)^{[d-1]}>0.\label{eq:holds on hat Y}
\end{equation}
\end{lem}

\begin{proof}
By our choice of $K$ in (\ref{eq:choice of K}), 
\begin{align}
e^{(1+K)\omega_{0}}(1-\Lambda) & \geq e^{(1+K)\omega_{0}}\wedge((1+K)-Ke^{\frac{\omega_{0}}{2n}})\label{eq:solution t equal 1}\\
 & =\left((1+K)e^{(1+K)\omega_{0}}-Ke^{(1+K+\frac{1}{2n})\omega_{0}}\right)\nonumber \\
 & =\sum_{k=0}^{n}\left((1+K)^{k+1}-K(1+K+\frac{1}{2n})^{k}\right)\frac{\omega_{0}^{k}}{k!},\nonumber 
\end{align}
on $M$. If $k=0$, then $(1+K)-K=1$. If $k\geq1$ then 
\begin{align}
(1+K)^{k+1}-K(1+K)^{k}\left(1+\frac{1}{2n(1+K)}\right)^{k}\geq(1+K)^{k}K\left(1+\frac{1}{2K}-e^{\frac{1}{2(1+K)}}+\frac{1}{2K}\right).\label{eq:solution t equal 1 2}
\end{align}
If $K>1$, we have $1+\frac{1}{2K}>e^{\frac{1}{2(1+K)}}$ . By (\ref{eq:solution t equal 1})
and (\ref{eq:solution t equal 1 2}), we have 
\begin{equation}
e^{(1+K)\omega_{0}}\wedge(1-\Lambda)\geq\sum_{k=0}^{n}\frac{(1+K)^{k}K}{2K}\frac{\omega_{0}^{k}}{k!}=\frac{1}{2}e^{(1+K)\omega_{0}}.\label{eq:holds 2-1}
\end{equation}
Thus we have proved (\ref{eq:solut tequal 1}).

At $t=1$, by (\ref{eq:holds 2-1}), we have 
\begin{align*}
(\exp\hat{\omega}_{0}\wedge(1-\hat{\Lambda}))^{[d-1]} & =\left(e^{(1+K)\omega_{0}}\wedge(1-\Lambda-\hat{\epsilon}^{n}\frac{\varpi^{k_{0}}}{k_{0}!})\wedge e^{\hat{\epsilon}\varpi}\right)^{[d-1]}\\
 & \geq\left(e^{(1+K)\omega_{0}}\wedge\left(\frac{1}{2}e^{\hat{\epsilon}\varpi}-\hat{\epsilon}^{n}\frac{\varpi^{k_{0}}}{k_{0}!}\wedge e^{\hat{\epsilon}\varpi}\right)\right)^{[d-1]}\\
 & =\left(e^{(1+K)\omega_{0}}\wedge\left(\sum_{k=0}^{d-2}\frac{1}{2}\cdot\frac{\hat{\epsilon}^{k}}{k!}\varpi^{k}-\sum_{k\geq k_{0}}^{d-2}\frac{\hat{\epsilon}^{n+k-k_{0}}}{(k-k_{0})!k_{0}!}\varpi^{k}\right)\right)^{[d-1]}\\
 & \ +\left(\frac{1}{2}\frac{\hat{\epsilon}^{d-1}}{(d-1)!}-\frac{\hat{\epsilon}^{n+d-1-k_{0}}}{k_{0}!(d-1-k_{0})!}\right)\varpi^{d-1}\\
 & \geq\frac{1}{4}\frac{\hat{\epsilon}^{d-1}}{(d-1)!}\varpi^{d-1},
\end{align*}
since $\hat{\epsilon}<\left(\frac{1}{4n!}\right)^{\frac{1}{n-k_{0}}}$. 
\end{proof}
We run a continuity argument for (\ref{eq:equation on hat Y}). Let
\begin{equation}
\mathbf{I}_{\hat{\epsilon}}:=\{t\in(0,1]:(\ref{eq:equation on hat Y})\text{ has a smooth solution for }\hat{\epsilon}\in(0,\hat{\epsilon}_{Y})\}.\label{eq:-20}
\end{equation}
By Lemma \ref{lem:solution t equal 1}, $\hat{\omega}_{0}$ is a subsolution
if $t=1$ and $\hat{\epsilon}<\hat{\epsilon}_{Y}$. We conclude that
$1\in\mathbf{I}_{\hat{\epsilon}}$. $\mathbf{I}_{\hat{\epsilon}}$
is clearly open as the cone condition is an open condition. Let 
\[
t_{\hat{\epsilon}}:=\inf\mathbf{I}_{\hat{\epsilon}}.
\]

\begin{rem}
\label{rem:We-will-assume}We may and will assume that $t_{\hat{\epsilon}}=0$
without loss of generality. If $t_{\hat{\epsilon}}=t'_{\hat{\epsilon}}>0$,
we may replace $\omega_{0}$ by $\hat{\omega}_{0}(t_{\hat{\epsilon}}',\hat{\epsilon})$.
If $t_{\hat{\epsilon}}'>0$, $\hat{\omega}(t_{\hat{\epsilon}}',\hat{\epsilon})$
is $\kah$. We can go through the same induction argument to prove
Theorem \ref{thm:cone condition in a neighborhood }. Then (\ref{eq:equation on hat Y})
admits a solution in $[\hat{\omega}_{0}(t_{\hat{\epsilon}}',\hat{\epsilon})]$.
By the openness of the cone condition, we have $t_{\hat{\epsilon}}<t_{\hat{\epsilon}}'$,
which a contradiction. Therefore, from now on, we assume that $t_{\hat{\epsilon}}=0$. 
\end{rem}

\section{Mass concentration \label{sec:Mass-concentration}}

In this Section, we prove a mass concentration result for our PDE
based on techniques from \cite{demailly2004numerical,chen2021j,Song2020NakaiMoishezonCF}.
We use notations as in previous sections.

By (\ref{eq:-20}), for $t\in\mathbf{I}_{\hat{\epsilon}}$, there
exists a $\omega_{t}\in[\hat{\omega}_{0}(t,\hat{\epsilon})]$ solving
\begin{equation}
\left(\exp\omega_{t}\wedge\left(1-\hat{\Lambda}\right)-c_{t,\hat{\epsilon}}\exp\hat{\rho}\right)^{[d]}=0,\label{eq:eq for omega_t}
\end{equation}
for some $c_{t,\hat{\epsilon}}>0$. Here $\hat{\omega}_{0}$, $\hat{\Lambda}$,
$\hat{\rho}$ are defined in (\ref{eq:perturbed classes}) and (\ref{eq:perturbed rho}).
As before, we denote $\Omega_{t}=\exp\omega_{t}.$

The main result of this section is the following theorem. 
\begin{thm}
\label{thm:Mass}Under the same assumption as in Theorem \ref{thm:cone condition in a neighborhood },
then there exists $\delta>0$, a finite Euclidean ball partition $\mathscr{P}=\{B_{j,3R}\}_{j\in\mathcal{J}}$
of $\hat{Y}$, $\varepsilon>0$, $r_{0}>0$, and a Kähler current
$\Upsilon\in(1-\delta)[\omega_{0}]$ s.t. for all $0<r<r_{0}$, $j\in\mathcal{J}$,
\[
\mathcal{P}_{\Lambda}(\Upsilon^{(r)})<1-\varepsilon,\ \text{in }B_{j,R},
\]
where $\Upsilon^{(r)}$ is given in Definition \ref{def:Local-regularization of current}.
Furthermore, $\Upsilon$ has positive Lelong numbers along $S_{\hat{Y}}.$ 
\end{thm}

\begin{rem}
\label{rem:reduction}We may simplify our argument by making following
assumptions without loss of generality:
\begin{enumerate}
\item $d>k_{0}$, by Lemma \ref{lem:base lemma};
\item $\mathbf{I}_{\hat{\epsilon}}=(0,1]$, which may be achieved by arguments
in Remark \ref{rem:We-will-assume};
\item $\hat{\Lambda}^{[d]}=0$; Since the subsolution condition is only
related to the components of degree less than $d$; If $\hat{\Lambda}^{[d]}\not=0$,
we may take another $c_{t,\hat{\epsilon}}'\geq0$ such that 
\[
\int_{\hat{Y}}\exp\omega_{t}\wedge(1-\sum_{k=1}^{d-1}\hat{\Lambda}^{[k]})=c'_{t,\hat{\epsilon}}\int_{\hat{Y}}\exp\hat{\rho};
\]
Therefore, after replacing $\hat{\Lambda}$ by $\sum_{k=1}^{d-1}\hat{\Lambda}^{[k]}$
and $c_{t,\hat{\epsilon}}$ by $c'_{t,\hat{\epsilon}}$ , the cone
condition does not change and the solvability is not affected; 
\item $\hat{\Lambda}$ is $k_{0}$-UP with respect to $\hat{\rho}$ and
a uniform constant $m'>0$. This reduction is possible due to Lemma
\ref{lem:K0 up for hat La}. 
\end{enumerate}
\end{rem}

\subsection{The lifted equation on the product manifold\label{subsec:The-lifted-equation}}

Following \cite{chen2021j}, we consider a new equation on the product
space $\mathcal{Y}=\hat{Y}\times\hat{Y}$. We fix some notations:
On $\mathcal{Y}$, let $\pi_{1}:\mathcal{Y}\to\hat{Y}$, $(x,y)\mapsto x$
and $\pi_{2}:\mathcal{Y}\to\hat{Y},\ (x,y)\mapsto y$ be canonical
projections. Let $\hat{Y}_{1}=\pi_{1}(\mathcal{Y})$ , $\hat{Y}_{2}=\pi_{2}(\mathcal{Y})$,
and let $\iota_{i}:\hat{Y}_{i}\hookrightarrow\mathcal{Y}$ be canonical
embeddings. We denote 
\begin{equation}
\Lambda_{x}:=\pi_{1}^{*}\hat{\Lambda},\ \Lambda_{y}:=\frac{1}{d}\pi_{2}^{*}\hat{\rho}.\label{eq:Define La_xy}
\end{equation}
We make the following observation: $\hat{\rho}$ can be viewed as
a solution to the equation of $\omega\in[\hat{\rho}]:$
\[
(\frac{1}{d}\hat{\rho}\wedge\exp\omega)^{[d]}=\frac{\omega^{d}}{d!}
\]
 because of the simple fact
\begin{equation}
\left(\frac{1}{d}\hat{\rho}\wedge\exp\hat{\rho}\right)^{[d]}=\frac{\hat{\rho}^{d}}{d!}.\label{eq:hat rho st eq}
\end{equation}
Let 
\begin{equation}
\boldsymbol{\Lambda}:=\Lambda_{x}+\Lambda_{y},\ \boldsymbol{\rho}:=\pi_{1}^{*}\hat{\rho}+\pi_{2}^{*}\hat{\rho},\ \boldsymbol{\varpi}=\pi_{1}^{*}\varpi+\pi_{2}^{*}\varpi.\label{eq:lifted forms}
\end{equation}
Let $\{B_{j}\}$ be a finite open cover of $\Delta=\{(x,x):x\in\hat{Y}\}\subset\hat{Y}\times\hat{Y}$
with balls $B_{j}$. Let $\theta_{j}^{2}$ be a partition of unity
subordinate to $B_{j}$. Let $g_{j,k}$ be the defining function of
$\Delta$ in $B_{j}$. Let 
\begin{equation}
\psi=\frac{1}{2}\log(\sum_{j,k}\theta_{j}^{2}|g_{j,k}|^{2}),\ \psi_{s}=\frac{1}{2}\log(\sum_{j,k}\theta_{j}^{2}|g_{j,k}|^{2}+s^{2}).\label{eq:definition of psi}
\end{equation}
Define 
\begin{equation}
\boldsymbol{\rho}_{s}=\boldsymbol{\rho}+\delta_{\rho}\sqrt{-1}\ddbar\left(\pi_{1}^{*}\phi_{S}+\pi_{2}^{*}\phi_{S}\right)+\delta_{\rho}^{2}\sqrt{-1}\ddbar\psi_{s}.\label{eq:rho_s_tau}
\end{equation}
$\delta_{\rho}$ is chosen small but fixed so that $\boldsymbol{\rho}_{s}$
is still $\kah$. We will determined $\delta_{\rho}$ later in Proposition
\ref{prop: solvability of lifted equation}. Let 
\begin{equation}
f_{t,s}=\left(\frac{\boldsymbol{\rho}_{s}^{2d}}{\boldsymbol{\rho}^{2d}}-(1+c_{t,\hat{\epsilon},\delta_{\rho}})\right)+c_{t,\hat{\epsilon}}.\label{eq:f_s,t,tau}
\end{equation}
where $c_{t,\hat{\epsilon},\delta_{\rho}}$ is chosen such that 
\[
\int_{\mathcal{Y}}\left(\boldsymbol{\rho}_{s}^{2d}-(1+c_{t,\hat{\epsilon},\delta_{\rho}})\boldsymbol{\rho}^{2d}\right)=0.
\]
Note, $c_{t,\hat{\epsilon},\delta_{\rho}}$ is uniformly bounded for
any $t\in(0,1)$ and $\hat{\epsilon}\in(0,\hat{\epsilon}_{Y})$.

We denote $\boldsymbol{\tau}=(t,s)$ with $t,s>0$. We consider\textbf{
}the following lifted equation on $\mathcal{Y}$: 
\begin{equation}
2\frac{\boldsymbol{\omega}_{\boldsymbol{\tau}}^{2d}}{(2d)!}=\sum_{k=1}^{d-1}\mathbf{\Lambda}^{[k]}\wedge\frac{\boldsymbol{\omega}_{\boldsymbol{\tau}}^{2d-k}}{(2d-k)!}+f_{\boldsymbol{\tau}}\frac{\boldsymbol{\rho}^{2d}}{(2d)!},\label{eq:product eq}
\end{equation}
for $\boldsymbol{\omega}_{\boldsymbol{\tau}}\in[\pi_{1}^{*}\omega_{t}+\pi_{2}^{*}\hat{\rho}]$.
If we denote 
\begin{equation}
\boldsymbol{\Omega}_{\boldsymbol{\tau}}:=\exp\boldsymbol{\omega}_{\boldsymbol{\tau}},\ \mathbf{P}:=\exp\boldsymbol{\rho},\label{eq:Omegatau exp}
\end{equation}
then we may re-write (\ref{eq:product eq}) as 
\begin{equation}
2\boldsymbol{\Omega}_{\boldsymbol{\tau}}^{[2d]}=(\boldsymbol{\Lambda}\wedge\boldsymbol{\Omega}_{\boldsymbol{\tau}})^{[2d]}+f_{\boldsymbol{\tau}}\mathbf{P}^{[2d]}.\label{eq:new eq big ome}
\end{equation}

\begin{lem}
\label{lem:O-up The-canonical-splitting}Notations as above. The canonical
splitting gives a labeled splitting $\mathcal{O}$ with respect to
$\boldsymbol{\rho}$. Moreover, $\boldsymbol{\Lambda}$ is $\mathcal{O}$-UP. 
\end{lem}

\begin{proof}
Following arguments in Example \ref{exa:example}, the product structure
gives the labeled splitting 
\[
\mathcal{O}_{(x,y)}=\{2,(d,d),\{\mathcal{T}_{x}\hat{Y}_{1},\mathcal{T}_{y}\hat{Y}_{2}\},(k_{0},1)\}.
\]
By Lemma \ref{lem:K0 up for hat La}, we have 
\[
\boldsymbol{\Lambda}\geq\min\{m',\frac{1}{d}\}\left(\frac{\hat{\rho}^{k_{0}}(x)}{k_{0}!}+\hat{\rho}(y)\right).
\]
Hence, $\boldsymbol{\Lambda}$ is $\mathcal{O}$-UP with uniform constant
$m''=\min\{m',\frac{1}{d}\}$. 
\end{proof}
\begin{prop}
\label{prop: solvability of lifted equation}Notations as above. For
$\delta_{\rho}$ small, there is a smooth solution to (\ref{eq:product eq})
if $\Lambda$ satisfies condition \textbf{H1.}
\end{prop}

\begin{proof}
By Lemma \ref{lem:O-up The-canonical-splitting}, $\boldsymbol{\Lambda}$
satisfies $\mathcal{O}$-UP condition with a uniform constant $m''>0$.
Let 
\[
f_{\min}=-\min\left\{ \frac{m''}{8d+2}\gamma_{\min}(\frac{2\kappa}{m''},2,(d,d),(k_{0},1)),\frac{\kappa\int_{\mathcal{Y}}(\pi_{1}^{*}\omega_{t}+\pi_{2}^{*}\hat{\rho})^{2d}}{2\int_{\mathcal{Y}}\boldsymbol{\rho}^{2d}}\right\} .
\]
Note $f_{\min}<-c<0$ where $c$ can be chosen independent of $\hat{\epsilon}$
and $t$. From Lemma \ref{lem:2.1 in DP}, there exits a uniform small
$\delta_{\rho}$ independent of $\hat{\epsilon}$ and $t$ such that
$f_{\boldsymbol{\tau}}>f_{\min}$. Therefore, $\boldsymbol{\Lambda}+f_{\boldsymbol{\tau}}\mathbf{P}^{[2d]}$
satisfies condition \textbf{H2}. It is easy to check that integrals
of both sides of (\ref{eq:product eq}) match. Thus apply Theorem
\ref{thm:Let--beanlytic}, one only needs to check the cone condition.

Let 
\begin{equation}
\boldsymbol{\omega}_{0}=\pi_{1}^{*}\omega_{t}+\pi_{2}^{*}\rho,\ \mathbf{\Omega}_{\text{0}}=\exp\boldsymbol{\omega}_{0},\ \hat{P}=\exp\hat{\rho}.\label{eq:Define ome_0}
\end{equation}
We claim $\boldsymbol{\omega}_{0}\in\mathcal{C}_{\boldsymbol{\Lambda}}^{2}$.
Since 
\begin{equation}
\mathbf{\Omega}_{0}^{[2d-1]}=\Omega_{t}(x)^{[d]}\wedge P(y)^{[d-1]}+\Omega_{t}(x)^{[d-1]}\wedge P(y)^{[d]},\label{eq:Omega_0 def}
\end{equation}
we have 
\begin{align}
(\boldsymbol{\Lambda}\wedge\boldsymbol{\Omega}_{0})^{[2d-1]} & =\left(\Lambda_{x}\wedge\Omega_{t}(x)\right)^{[d]}\wedge\hat{P}(y)^{[d-1]}+\left(\Lambda_{x}\wedge\Omega_{t}(x)\right)^{[d-1]}\wedge\hat{P}(y)^{[d]}\label{eq:conein 2d}\\
 & +\Omega_{t}(x)^{[d-1]}\wedge\left(\Lambda_{y}\wedge\hat{P}(y)\right)^{[d]}+\Omega_{t}(x)^{[d]}\wedge(\Lambda_{y}\wedge\hat{P}(y))^{[d-1]}.\nonumber 
\end{align}
By equation (\ref{eq:eq for omega_t}), 
\begin{align}
(\boldsymbol{\Lambda}\wedge\boldsymbol{\Omega}_{0})^{[2d-1]} & =\left(\Omega_{t}-c_{t,\hat{\epsilon}}\hat{P}(x)\right)^{[d]}\wedge\hat{P}(y)^{[d-1]}+\left(\Lambda_{x}\wedge\Omega_{t}\right)^{[d-1]}\wedge\hat{P}(y)^{[d]}\label{eq:tech}\\
 & +\Omega_{t}{}^{[d-1]}\wedge\hat{P}(y)^{[d]}+\Omega_{t}(x)^{[d]}\wedge(\Lambda_{y}\wedge\hat{P}(y))^{[d-1]}.\nonumber 
\end{align}
Since $\omega_{t}$ satisfies the corresponding cone condition of
equation (\ref{eq:eq for omega_t}), we have 
\begin{equation}
\left(\Lambda_{x}\wedge\Omega_{t}\right)^{[d-1]}<\Omega_{t}{}^{[d-1]}.\label{eq:tech-1}
\end{equation}
Similarly, $\rho$ satisfies equation (\ref{eq:hat rho st eq}), which
implies 
\begin{equation}
(\Lambda_{y}\wedge\hat{P}(y))^{[d-1]}<\hat{P}(y)^{[d-1]}.\label{eq:tech-2}
\end{equation}
Combining (\ref{eq:tech-1}), (\ref{eq:tech-2}), and (\ref{eq:tech}),
we obtain 
\begin{align*}
(\boldsymbol{\Lambda}\wedge\boldsymbol{\Omega}_{0})^{[2d-1]} & <\Omega_{t}(x)^{[d]}\wedge\hat{P}(y)^{[d-1]}+\Omega_{t}(x)^{[d-1]}\wedge\hat{P}(y)^{[d]}\\
 & +\Omega_{t}(x)^{[d-1]}\wedge\hat{P}(y)^{[d]}+\Omega_{t}(x)^{[d]}\wedge\hat{P}(y)^{[d-1]}\\
 & =2\Omega_{t}(x)^{[d]}\wedge\hat{P}(y)^{[d-1]}+2\Omega_{t}(x)^{[d-1]}\wedge\hat{P}(y)^{[d]}\\
 & =2\boldsymbol{\Omega}_{0}^{[2d-1]}.
\end{align*}
Therefore, $\boldsymbol{\omega}_{0}\in\mathcal{C}_{\mathbf{\Lambda}}^{2}$
satisfies the cone condition. By Theorem \ref{thm:Let--beanlytic},
there exists a smooth solution to (\ref{thm:Let--beanlytic}). 
\end{proof}
We illustrate the construction of the lifted equation with the following
example.
\begin{example}
\label{exa:non-trivial example}Suppose that $Y$ is smooth of dimension
3 and $\Lambda=\rho^{2}$. Then equations (\ref{eq:eq for omega_t})
and (\ref{eq:hat rho st eq}) imply that
\begin{equation}
\frac{\omega_{t}^{3}}{3!}=\rho^{2}\wedge\omega_{t}+c_{t}\frac{\rho^{3}}{3!},\ \frac{\rho^{3}}{3!}=\frac{\rho}{3}\wedge\frac{\rho^{2}}{2!}.\label{eq:example}
\end{equation}
Let $\boldsymbol{\omega}_{0}=\pi_{1}^{*}\omega_{t}+\pi_{2}^{*}\rho$
and $\boldsymbol{\omega}\in[\boldsymbol{\omega}_{0}]$. The lifted
equation on $\mathcal{Y}$ is 
\begin{equation}
2\frac{\boldsymbol{\omega}^{6}}{6!}=\pi_{1}^{*}\rho^{2}\wedge\frac{\boldsymbol{\omega}^{4}}{4!}+\frac{1}{3}\pi_{2}^{*}\rho\wedge\frac{\boldsymbol{\omega}^{5}}{5!}+f_{t,s}\frac{(\pi_{1}^{*}\rho+\pi_{2}^{*}\rho)^{6}}{6!}.\label{eq:exmpll}
\end{equation}
We use (\ref{eq:example}) and corresponding cone conditions to obtain
that
\begin{align*}
2\boldsymbol{\omega}_{0}^{5}/5! & =2\frac{\pi_{1}^{*}\omega_{t}^{3}}{3!}\wedge\frac{\pi_{2}^{*}\rho^{2}}{2}+2\frac{\pi_{1}^{*}\omega_{t}^{2}}{2}\wedge\frac{\pi_{2}^{*}\rho^{3}}{3!}\\
 & >\pi_{1}^{*}\left(\rho^{2}\wedge\omega_{t}\right)\wedge\frac{\pi_{2}^{*}\rho^{2}}{2}+\pi_{1}^{*}\rho^{2}\wedge\frac{\pi_{2}^{*}\rho^{3}}{3!}\\
 & +\frac{\pi_{1}^{*}\omega_{t}^{3}}{3!}\wedge\pi_{2}^{*}\rho\wedge\frac{\pi_{2}^{*}\rho}{3}+\frac{\pi_{1}^{*}\omega_{t}^{2}}{2}\wedge\frac{\pi_{2}^{*}\rho^{2}}{2}\wedge\frac{\pi_{2}^{*}\rho}{3}\\
 & =\pi_{1}^{*}\rho^{2}\wedge\frac{\boldsymbol{\omega}_{0}^{3}}{3!}+\pi_{2}^{*}\frac{\rho}{3}\wedge\frac{\boldsymbol{\omega}_{0}^{4}}{4!}.
\end{align*}
Thus $\boldsymbol{\omega}_{0}$ is a subsolution to (\ref{eq:exmpll}).
\end{example}

Suppose that $\boldsymbol{\omega}_{\boldsymbol{\tau}}$ is a solution
to the equation (\ref{eq:product eq}). At $(x,y)\in\mathcal{Y}$,
we pick a coordinate $\{x^{i}\}$, $\{y^{i}\}$ near $x$ and $y$,
respectively. Then, $\boldsymbol{\omega}$ is represented by a Hermitian
matrix 
\[
\mathbf{A}=\left(\begin{array}{cc}
H & D\\
D^{\dagger} & V
\end{array}\right).
\]
We write 
\begin{equation}
\boldsymbol{\omega}_{\boldsymbol{\tau}}=\boldsymbol{\omega}_{x}+\boldsymbol{\omega}_{y}+\boldsymbol{\omega}_{m}+\bar{\boldsymbol{\omega}}_{m},\label{eq:omega_tau}
\end{equation}
where 
\begin{align}
\boldsymbol{\omega}_{x} & =\pi_{1}^{*}\iota_{1}^{*}\boldsymbol{\omega}_{\boldsymbol{\tau}}=\frac{\sqrt{-1}}{2}H_{i\bar{j}}dx^{i}\wedge d\bar{x}^{j},\label{eq:omega _x}
\end{align}
\begin{equation}
\boldsymbol{\omega}_{y}=\pi_{2}^{*}\iota_{2}^{*}\boldsymbol{\omega}_{\boldsymbol{\tau}}=\frac{\sqrt{-1}}{2}V_{i\bar{j}}dy^{i}\wedge d\bar{y}^{j},\label{eq:omega_y}
\end{equation}
\begin{equation}
\boldsymbol{\omega}_{m}=\frac{\sqrt{-1}}{2}D_{i\bar{j}}dx^{i}\wedge d\bar{y}^{j}.\label{eq:omega_mix}
\end{equation}
Let 
\begin{equation}
\hat{c}=\hat{c}(t,\hat{\epsilon}):=\left[\frac{\hat{\rho}^{d}}{d!}\right]\cdot\hat{Y}.\label{eq:hat c}
\end{equation}
Define 
\begin{align}
\omega_{\boldsymbol{\tau}} & :=\frac{1}{\hat{c}}\int_{\{x\}\times\hat{Y}}\left(\Lambda_{y}\wedge\boldsymbol{\Omega}_{\boldsymbol{\tau}}\right)^{[d+1]}\label{eq:def om _tau}\\
 & =\frac{1}{\hat{c}}\int_{\{x\}\times\hat{Y}}\left(\left(\Lambda_{y}\wedge\exp\boldsymbol{\omega}_{y}\right)^{[d]}\wedge\boldsymbol{\omega}_{x}+\left(\Lambda_{y}\wedge\exp\boldsymbol{\omega}_{y}\right)^{[d-1]}\wedge\boldsymbol{\omega}_{m}\wedge\bar{\boldsymbol{\omega}}_{m}\right).\nonumber 
\end{align}

\begin{lem}
Notations as above. We have $\omega_{\boldsymbol{\tau}}\in[\hat{\omega}_{0}]$.
\label{lem:Notations-as-above,in teh class} 
\end{lem}

\begin{proof}
Direct computation shows 
\begin{align}
[\omega_{\boldsymbol{\tau}}] & =\frac{1}{\hat{c}}\left[\int_{\{x\}\times\hat{Y}}\left(\Lambda_{y}\wedge\frac{\boldsymbol{\omega}_{\boldsymbol{\tau}}^{d}}{d!}\right)\right]\label{eq:omega_tau in class}\\
 & =\frac{1}{\hat{c}}\left(\frac{1}{d}[\Lambda_{y}]\cdot\frac{[\boldsymbol{\omega}_{y}]^{d-1}}{(d-1)!}\cdot\hat{Y}\right)[\omega_{t}]\nonumber \\
 & =\frac{1}{\hat{c}}\left(\frac{[\hat{\rho}]^{d}}{d!}\cdot\hat{Y}\right)[\hat{\omega}_{0}]\nonumber \\
 & =[\hat{\omega}_{0}].\nonumber 
\end{align}
\end{proof}
We list following definitions.
\begin{equation}
\mathbf{F}(\boldsymbol{\omega},\mathbf{\Lambda}):=\frac{\left(\mathbf{\Lambda}\wedge\exp\boldsymbol{\omega}\right)^{[2d]}}{(\exp\boldsymbol{\omega})^{[2d]}}.\label{eq:bold F}
\end{equation}
For convenience, we abuse the notation and write 
\begin{align}
\mathbf{F}\left(\mathbf{A}\right) & =\mathbf{F}(\boldsymbol{\omega},\mathbf{\Lambda}),\label{eq:Bold F A}
\end{align}
and define 
\begin{equation}
F_{1}(H):=\frac{\left(\Lambda_{x}\wedge\exp\boldsymbol{\omega}_{x}\right)^{[d]}}{\left(\exp\boldsymbol{\omega}_{x}\right)^{[d]}},\label{eq:F1}
\end{equation}
\begin{equation}
F_{2}(V):=\frac{\left(\Lambda_{y}\wedge\exp\boldsymbol{\omega}_{y}\right)^{[d]}}{\left(\exp\boldsymbol{\omega}_{y}\right)^{[d]}}.\label{eq:F2}
\end{equation}
We also define 
\begin{equation}
\mathcal{P}_{\mathbf{\Lambda}}(\mathbf{A})=\max_{\mathbf{B}\in\overline{\Gamma_{2n\times2n}^{+}},\|\mathbf{B}\|=1}\lim_{t\to+\infty}\mathbf{F}(\mathbf{A}+t\mathbf{B}),\label{eq:P bold}
\end{equation}
\begin{equation}
\mathcal{P}_{1}(H)=\mathcal{P}_{\hat{\Lambda}}(H)=\max_{B\in\overline{\Gamma_{n\times n}^{+}},\|B\|=1}\lim_{t\to+\infty}F_{1}(H+tB),\label{eq:P1}
\end{equation}
\begin{align}
\mathcal{P}_{2}(V) & =\max_{B\in\overline{\Gamma_{n\times n}^{+}},\|B\|=1}\lim_{t\to+\infty}F_{2}(V+tB).\label{eq:P2}
\end{align}
It is easy to check that $\mathbf{F},F_{1},F_{2},\mathcal{P}_{\mathbf{\Lambda}},\mathcal{P}_{1},\mathcal{P}_{2}$
are all functions satisfying monotone and convexity conditions by
Lemma \ref{lem:P_La property}. 
\begin{lem}
\label{lem:generalized G Chen}Notations as above. We have 
\begin{equation}
\mathbf{F}(\mathbf{A})\geq F_{1}(H-DV^{-1}D^{\dagger})+F_{2}(V).\label{eq:F(A,Lambda)}
\end{equation}
Moreover, 
\begin{equation}
\mathcal{P}_{\mathbf{\Lambda}}(\mathbf{A})\geq\mathcal{P}_{1}(H-DV^{-1}D^{\dagger})+F_{2}(V).\label{eq:P(ALambda)}
\end{equation}
\end{lem}

\begin{proof}
We denote 
\begin{align*}
\chi & =\mathbf{A}^{\bar{i}j}2\sqrt{-1}\frac{\pdv}{\pdv\bar{z}^{i}}\wedge\frac{\pdv}{\pdv z^{j}}.
\end{align*}
Then by Lemma \ref{lem:rewrite}, $\mathbf{F}(\mathbf{A})=\<\mathbf{\Lambda},\exp\chi\>.$
Denote $\hat{H}=H-DV^{-1}D^{\dagger}$. Then 
\begin{equation}
\mathbf{A}^{-1}=\left(\begin{array}{cc}
\hat{H}^{-1} & -\hat{H}^{-1}DV^{-1}\\
-V^{-1}D^{\dagger}\hat{H}^{-1} & V^{-1}+V^{-1}D^{\dagger}\hat{H}^{-1}DV^{-1}
\end{array}\right).\label{eq:bold A inverse}
\end{equation}
Let 
\begin{align}
\chi_{h} & =\hat{H}^{\bar{i}j}2\sqrt{-1}\frac{\pdv}{\pdv\bar{x}^{i}}\wedge\frac{\pdv}{\pdv x^{j}},\label{eq:chi_h}\\
\chi_{v} & =\left(V^{\bar{i}j}+\left(V^{-1}D^{\dagger}\hat{H}^{-1}DV^{-1}\right)^{\bar{i}j}\right)2\sqrt{-1}\frac{\pdv}{\pdv\bar{y}^{i}}\wedge\frac{\pdv}{\pdv y^{j}},\label{eq:chi_v}\\
\chi_{m} & =\left(-\hat{H}^{-1}DV^{-1}\right)^{\bar{i}j}2\sqrt{-1}\frac{\pdv}{\pdv\bar{x}^{i}}\wedge\frac{\pdv}{\pdv y^{j}}.\label{eq:chi _m}
\end{align}
Then 
\begin{align}
\mathbf{F}(\mathbf{A}) & =\sum_{k=1}^{d-1}\frac{1}{k!}\<\mathbf{\Lambda}^{[k]},\sum_{a+2b+c=k}\frac{k!}{a!c!b!b!}\chi_{h}^{a}\wedge\chi_{v}^{c}\wedge(\chi_{m}\wedge\overline{\chi_{m}})^{b}\>.\label{eq:F(Omega)}
\end{align}
Note that $\chi_{h}^{a}\wedge\chi_{v}^{c}\wedge(\chi_{m}\wedge\overline{\chi_{m}})^{b}$
is a wedge product of some type $(a+b,a+b)$ tensor in $x$ and type
$(b+c,b+c)$ tensor in $y$. Since $\mathbf{\Lambda}^{[k]}=\Lambda_{x}^{[k]}+\Lambda_{y}^{[k]}$,
for non-vanishing terms in (\ref{eq:F(Omega)}), these indeces satisfy
$a+b=k$ or $b+c=k$; which implies that $a=k$ or $c=k$. Therefore,
\begin{align}
\mathbf{F}(\mathbf{A}) & =\sum_{k=1}^{d-1}\frac{1}{k!}\<\Lambda_{x}^{[k]},\chi_{h}^{k}\>+\<\Lambda_{y},\chi_{v}\>.\label{eq:bold FA}
\end{align}
The first term is $F_{1}(H-DV^{-1}D^{\dagger})$. Since $V^{-1}D^{\dagger}\hat{H}^{-1}DV^{-1}$
is non-negative, we have $\<\Lambda_{y},\chi_{v}\>\geq F_{2}(V)$.
We have proved (\ref{eq:F(A,Lambda)}).

Let $B$ be any $d\times d$ non-negative Hermitian matrix. Let $\mathbf{B}=\left(\begin{array}{cc}
B & 0\\
0 & 0
\end{array}\right)$. From (\ref{eq:F(A,Lambda)}), 
\begin{equation}
\mathbf{F}(\mathbf{A}+t\mathbf{B})\geq F_{1}(\hat{H}+tB)+F_{2}(V).\label{eq:limit bold FA}
\end{equation}
We obtain (\ref{eq:P(ALambda)}) by taking $t\to\infty$ and then
taking maximum for all $\|B\|=1$.
\end{proof}
\begin{lem}
\label{lem:tech lem 1} Notations as above, we have 
\[
\Lambda_{y}\wedge\frac{\boldsymbol{\omega}_{y}^{d-2}}{(d-2)!}\wedge\boldsymbol{\omega}_{m}\wedge\bar{\boldsymbol{\omega}}_{m}\geq-\Lambda_{y}\wedge\frac{\boldsymbol{\omega}_{y}^{d-1}}{(d-1)!}V^{\bar{j}l}D_{i\bar{j}}\overline{D_{k\bar{l}}}\frac{\sqrt{-1}}{2}dx^{i}\wedge d\bar{x}^{k}.
\]
\end{lem}

\begin{proof}
Notice that 
\begin{align}
\boldsymbol{\omega}_{m}\wedge\overline{\boldsymbol{\omega}}_{m} & =\left(\frac{\sqrt{-1}}{2}\right)^{2}D_{i\bar{j}}dx^{i}\wedge d\bar{y}^{j}\wedge\overline{D_{k\bar{l}}}dy^{l}\wedge d\bar{x}^{k}\label{eq:omega_mix wedge}\\
 & =-\left(\frac{\sqrt{-1}}{2}\right)^{2}D_{i\bar{j}}\overline{D_{k\bar{l}}}dx^{i}\wedge d\bar{x}^{k}\wedge dy^{l}\wedge d\bar{y}^{j}.\nonumber 
\end{align}
To prove the lemma, it is sufficient to show that for any $\zeta=(\zeta^{i})$,
\begin{equation}
\sum_{i,j,k,l}\Lambda_{y}\wedge\left(\frac{\boldsymbol{\omega}_{y}^{d-2}}{(d-2)!}\wedge\frac{\sqrt{-1}}{2}dy^{l}\wedge d\bar{y}^{j}-\frac{\boldsymbol{\omega}_{y}^{d-1}}{(d-1)!}V^{\bar{j}l}\right)D_{i\bar{j}}\overline{D_{k\bar{l}}}\zeta^{i}\bar{\zeta}^{k}\leq0.\label{eq:equiv to derivative}
\end{equation}
However, since $F_{2}^{l\bar{j}}(V)\leq0$, (\ref{eq:equiv to derivative})
is true.
\end{proof}
The following lemma illustrates that $\omega_{\boldsymbol{\tau}}$
satisfies the cone condition.
\begin{lem}
Notations as above, we have \label{lem:We lemma tech } 
\begin{equation}
\mathcal{P}_{\hat{\Lambda}}(\omega_{\boldsymbol{\tau}})<1.\label{eq:we lem tech}
\end{equation}
\end{lem}

\begin{proof}
From Lemma \ref{lem:tech lem 1}, we see that 
\begin{align}
\omega_{\boldsymbol{\tau}} & =\frac{1}{\hat{c}}\int_{\{x\}\times\hat{Y}}\Lambda_{y}\wedge\left(\frac{\boldsymbol{\omega}_{y}^{d-1}}{(d-1)!}\wedge\boldsymbol{\omega}_{x}+\frac{\boldsymbol{\omega}_{y}^{d-2}}{(d-2)!}\wedge\boldsymbol{\omega}_{m}\wedge\bar{\boldsymbol{\omega}}_{m}\right)\label{eq:compute omega_tau}\\
 & \geq\frac{1}{\hat{c}}\int_{\{x\}\times\hat{Y}}\Lambda_{y}\wedge\frac{\boldsymbol{\omega}_{y}^{d-1}}{(d-1)!}\wedge\left(H_{i\bar{k}}-V^{\bar{j}l}D_{i\bar{j}}\overline{D_{k\bar{l}}}\right)\frac{\sqrt{-1}}{2}dx^{i}\wedge d\bar{x}^{k}.\nonumber 
\end{align}
Now, we apply $\mathcal{P}_{1}$ to $H_{i\bar{k}}-V^{\bar{j}l}D_{i\bar{j}}\overline{D_{k\bar{l}}}$.
By Lemma \ref{lem:generalized G Chen}, 
\begin{align}
\mathcal{P}_{1}(H-V^{\bar{j}l}D_{i\bar{j}}\overline{D_{k\bar{l}}}) & \leq\mathcal{P}_{\mathbf{\Lambda}}\left(\begin{array}{cc}
H & D\\
D^{\dagger} & V
\end{array}\right)-F_{2}(V)\label{eq:lemma tech sub 1}\\
 & <2-F_{2}(V).\nonumber 
\end{align}
Using the convexity of $\mathcal{P}$, we have 
\begin{align}
\mathcal{P}_{\hat{\Lambda}}(\omega_{\boldsymbol{\tau}}) & <\frac{1}{\hat{c}}\int_{\{x\}\times\hat{Y}}\left(\Lambda_{y}\wedge\frac{\boldsymbol{\omega}_{y}^{d-1}}{(d-1)!}\right)\left(2-F_{2}(V)\right)\label{eq:lemma tech sub 2}\\
 & =\frac{1}{\hat{c}}\left[2\hat{c}-\int_{\{x\}\times\hat{Y}}F_{2}^{2}(V)\frac{\boldsymbol{\omega}_{y}^{d}}{d!}\right]\nonumber \\
 & \leq\frac{1}{\hat{c}}\left[2\hat{c}-\frac{1}{[\boldsymbol{\omega}_{y}^{d}/d!]}\left(\int_{\hat{Y}}F_{2}(V)\frac{\boldsymbol{\omega}_{y}^{d}}{d!}\right)^{2}\right].\nonumber 
\end{align}
Notice$\int_{\hat{Y}}F_{2}(V)\frac{\boldsymbol{\omega}_{y}^{d}}{d!}=\int_{\hat{Y}}\frac{\hat{\rho}^{d}}{d!}=\hat{c}.$Thus,
by 
\begin{equation}
\mathcal{P}_{\hat{\Lambda}}(\omega_{\boldsymbol{\tau}})<\frac{1}{\hat{c}}\left[2\hat{c}-\hat{c}\right]=1.\label{eq:we lem tech 1}
\end{equation}
We have finished the proof. 
\end{proof}

\subsection{Mass Concentration on $\Delta$}

Next, we show that any weak limit of $\boldsymbol{\omega}_{\boldsymbol{\tau}}^{d}$
(\ref{eq:product eq}) as $s$ is converging to 0 contains a positive
piece of $[\Delta]$. 
\begin{lem}
\label{lem:2.1 in DP}Let $\boldsymbol{\rho}_{s}$ be defined as in
(\ref{eq:rho_s_tau}). Then we have 
\begin{enumerate}
\item For any $\epsilon>0$, there is a $\delta_{\rho}>0$ s.t. $\boldsymbol{\rho}_{s}>(1-\epsilon)\boldsymbol{\rho}+\frac{\delta_{\rho}}{2}\boldsymbol{\varpi}.$ 
\item \label{enu:For-an-open}Let $V_{s}:=\{z:\psi(z)<\log s\}$ where $\psi$
is given in (\ref{eq:definition of psi}). Let $p$ be a point of
$\Delta$. For any open neighborhood $U$ of $p$, there is a $\delta_{1}(U)>0$
independent of $s$ s.t. 
\[
\int_{U\cap V_{s}}\boldsymbol{\rho}_{s}^{2d}\geq\delta_{1}(U)>0.
\]
\item For $p\in\Delta$ and any open neighborhood $U$ of $p$, there is
a $\delta_{2}(U)>0$ independent of $s$ such that 
\begin{equation}
\int_{U\cap V_{s}}\boldsymbol{\rho}_{s}^{d}\wedge\boldsymbol{\varpi}^{d}\geq\delta_{2}(U)>0.\label{eq:delta(u_)}
\end{equation}
\end{enumerate}
\end{lem}

The proof is the same as in Lemma 2.1 in Demailly-Păun\cite{demailly2004numerical}.

Let $\mathbf{1}_{\Delta}$ be the characteristic function of the diagonal
$\Delta\subset\mathcal{Y}$. The following proposition asserts that
when $s$ tends to $0$, a positive portion of mass of $\boldsymbol{\omega}_{\boldsymbol{\tau}}^{d}$
concentrates on $\Delta$. The proof is similar to the proof of Proposition
2.6 of Demailly-Păun. For readers' convenience, we write a detailed
proof here. 
\begin{prop}
\label{prop:concentration to T}Let $T$ be a weak limit of $\boldsymbol{\omega}_{\boldsymbol{\tau}}^{d}$
when $s\to0$ for some $t\in(0,1)$. Then $\mathbf{1}_{\Delta}T$
is a positive closed current and there is a constant $\epsilon_{T}$
s.t. $\mathbf{1}_{\Delta}T=\epsilon_{T}[\Delta]$ and $\epsilon_{T}>\epsilon_{\Delta}$,
where $\epsilon_{\Delta}(\hat{Y})>0$ is a constant independent of
$t,\hat{\epsilon}$. 
\end{prop}

\begin{proof}
We first prove 2 claims.

\textbf{Claim 1: }For any $p\in\Delta$ and $U$ a neighborhood of
$p$, there is a constant $\delta(U)>0$ independent of $s$ s.t.
$\int_{U\cap V_{s}}\boldsymbol{\omega}_{\boldsymbol{\tau}}^{d}\wedge\boldsymbol{\varpi}^{d}>\delta(U)$
for small $s$.

Notice that 
\begin{align}
2\frac{\boldsymbol{\omega}_{\boldsymbol{\tau}}^{2d}}{(2d)!} & \geq f_{\boldsymbol{\tau}}\frac{\boldsymbol{\rho}^{2d}}{(2d)!}\label{eq:mass con}\\
 & =\frac{1}{(2d)!}\left(\boldsymbol{\rho}_{s}^{2d}+(c_{t}-1)\boldsymbol{\rho}^{2d}\right),\nonumber 
\end{align}
where $c_{t}=c_{t,\hat{\epsilon}}-c_{t,\hat{\epsilon},\delta_{\rho}}$
has a uniform lower bound for all $t,\hat{\epsilon}$. Let $\lambda_{1}\leq\lambda_{2}\leq\cdots\leq\lambda_{2d}$
be the eigenvalues of $\boldsymbol{\omega}_{\boldsymbol{\tau}}$ with
respect to $\boldsymbol{\rho}_{s}$. From (\ref{eq:mass con}), we
have 
\begin{equation}
2\lambda_{1}\cdots\lambda_{2d}\frac{\boldsymbol{\rho}_{s}^{2d}}{(2d)!}-\frac{c_{t}-1}{(2d)!}\boldsymbol{\rho}^{2d}\geq\frac{\boldsymbol{\rho}_{s}^{2d}}{(2d)!}.\label{eq:mass con-1}
\end{equation}
We have 
\begin{equation}
\boldsymbol{\omega}_{\boldsymbol{\tau}}^{d}\geq\lambda_{1}\cdots\lambda_{d}\boldsymbol{\rho}_{s}^{d}.\label{eq:Omega1 con}
\end{equation}
\begin{equation}
\frac{\boldsymbol{\omega}_{\boldsymbol{\tau}}^{d}}{d!}\wedge\frac{\boldsymbol{\rho}_{s}^{d}}{d!}>\lambda_{d+1}\cdots\lambda_{2d}\frac{\boldsymbol{\rho}_{s}^{d}}{d!}.\label{eq:Omega control 2}
\end{equation}
Thus 
\begin{align*}
\int_{\hat{Y}}\lambda_{d+1}\cdots\lambda_{2d}\frac{\boldsymbol{\rho}_{s}^{d}}{d!} & \leq\int_{\hat{Y}}\frac{\boldsymbol{\omega}_{\boldsymbol{\tau}}^{d}}{d!}\wedge\frac{\boldsymbol{\rho}_{s}^{d}}{d!}\\
 & =\int_{\hat{Y}}\frac{\boldsymbol{\omega}_{0}^{d}}{d!}\wedge\frac{\boldsymbol{\rho}^{d}}{d!}\\
 & \leq C(\hat{Y}).
\end{align*}
Given $U$, let $\delta_{2}(U)$ be given in (\ref{eq:delta(u_)}).
For any $\delta>0$ s.t. $\left(\left(\frac{2}{\delta_{\rho}}\right)^{d}+1\right)\delta<(1-2^{-d})\delta_{2}$,
let $E_{\delta}$ be the set of points in $\hat{Y}$ s.t. $\lambda_{d+1}\cdots\lambda_{2d}>C(\hat{Y})/\delta$.
Then, it is clear that 
\begin{equation}
\int_{E_{\delta}}\boldsymbol{\rho}_{s}^{2d}\leq\delta.\label{eq:tech in DP mass}
\end{equation}
Therefore, we have 
\begin{align*}
\int_{U\cap V_{s}\backslash E_{\delta}}\boldsymbol{\omega}_{\boldsymbol{\tau}}^{d}\wedge\boldsymbol{\varpi}^{d} & \ge\int_{U\cap V_{s}\backslash E_{\delta}}\lambda_{1}\cdots\lambda_{d}\boldsymbol{\rho}_{s}^{d}\wedge\boldsymbol{\varpi}^{d}\\
 & =\int_{U\cap V_{s}\backslash E_{\delta}}\frac{\lambda_{1}\cdots\lambda_{2d}}{\lambda_{d}\cdots\lambda_{2d}}\boldsymbol{\rho}_{s}^{d}\wedge\boldsymbol{\varpi}^{d}\\
 & \geq\frac{\delta}{C(\hat{Y})}\left(\int_{U\cap V_{s}\backslash E_{\delta}}\lambda_{1}\cdots\lambda_{2d}\boldsymbol{\rho}_{s}^{d}\wedge\boldsymbol{\varpi}^{d}\right).
\end{align*}
From (\ref{eq:mass con-1}), we have 
\begin{equation}
\lambda_{1}\cdots\lambda_{2d}\geq\frac{1}{2}+\frac{c_{t}-1}{2}\frac{\boldsymbol{\rho}^{2d}}{\boldsymbol{\rho}_{s}^{2d}}.\label{eq:gist in D_P}
\end{equation}
We assume that $c_{t}<1$ since otherwise he right hand side of (\ref{eq:gist in D_P})$\geq\frac{1}{2}$
and the proof is easier. From Lemma (\ref{lem:2.1 in DP}), we choose
$\delta_{\rho}$ small so that $\boldsymbol{\rho}_{s}>\frac{\delta_{\rho}}{2}\boldsymbol{\varpi}$.
Then 
\begin{align*}
\frac{c_{t}-1}{2}\cdot\frac{\boldsymbol{\rho}^{2d}}{\boldsymbol{\rho}_{s}^{2d}}\cdot\boldsymbol{\rho}_{s}^{d}\wedge\boldsymbol{\varpi}^{d} & \geq\frac{c_{t}-1}{2}\cdot\frac{\boldsymbol{\rho}^{2d}}{\boldsymbol{\rho}_{s}^{2d}}\cdot\boldsymbol{\rho}_{s}^{2d}2^{d}\delta_{\rho}^{-d}\\
 & =\left(\frac{2}{\delta_{\rho}}\right)^{d}\frac{c_{t}-1}{2}\cdot\boldsymbol{\rho}^{2d}.
\end{align*}
Therefore, we see that 
\begin{align}
\int_{U\cap V_{s}\backslash E_{\delta}}\boldsymbol{\omega}_{\boldsymbol{\tau}}^{d}\wedge\boldsymbol{\varpi}^{d} & \geq\frac{\delta}{C(\hat{Y})}\int_{U\cap V_{s}\backslash E_{\delta}}\left(\frac{1}{2}\boldsymbol{\rho}_{s}^{d}\wedge\boldsymbol{\varpi}^{d}+\left(\frac{2}{\delta_{\rho}}\right)^{d}\frac{c_{t}-1}{2}\cdot\boldsymbol{\rho}^{2d}\right).\label{eq:D-P tech 1}
\end{align}
Now, for sufficiently small $s$ (independent of $t,\hat{\epsilon}$),
we have 
\begin{equation}
-\int_{U\cap V_{s}}\left(\left(\frac{2}{\delta_{\rho}}\right)^{d}\frac{c_{t}-1}{2}\cdot\boldsymbol{\rho}^{2d}\right)<2^{-d}\delta_{2}.\label{eq:DP tech 2}
\end{equation}
Also, from (\ref{eq:tech in DP mass}), 
\begin{align}
\int_{U\cap V_{s}\backslash E_{\delta}}\left(\frac{1}{2}\boldsymbol{\rho}_{s}^{d}\wedge\boldsymbol{\varpi}^{d}\right) & =\frac{1}{2}\left(\int_{U\cap V_{s}}\boldsymbol{\rho}_{s}^{d}\wedge\boldsymbol{\varpi}^{d}-\int_{E_{\delta}}\boldsymbol{\rho}_{s}^{d}\wedge\boldsymbol{\varpi}^{d}\right)\label{eq:D-P tech 3}\\
 & \geq\frac{1}{2}\left(\delta_{2}-\left(\frac{2}{\delta_{\rho}}\right)^{d}\delta\right).\nonumber 
\end{align}
Thus, by (\ref{eq:D-P tech 1}),(\ref{eq:DP tech 2}), and (\ref{eq:D-P tech 3}),
we have 
\begin{align}
\int_{U\cap V_{s}\backslash E_{\delta}}\boldsymbol{\omega}_{\boldsymbol{\tau}}^{d}\wedge\boldsymbol{\varpi}^{d} & \geq\frac{\delta}{C(\hat{Y})}\cdot\frac{1}{2}\left(\delta_{2}-\left(\frac{2}{\delta_{\rho}}\right)^{d}\delta-2^{-d}\delta_{2}\right)\label{eq:D-P tech 4}\\
 & \geq\frac{\delta^{2}}{2C(\hat{Y})}.\nonumber 
\end{align}
We have proved Claim 1.

\textbf{Claim 2:} $\boldsymbol{\omega}_{\boldsymbol{\tau}}^{d}$ has
a uniform upper bound in mass.

In fact, it is easy to check that 
\begin{align*}
\int_{\hat{Y}}\boldsymbol{\omega}_{\boldsymbol{\tau}}^{d}\wedge\boldsymbol{\varpi}^{d} & =\int_{\hat{Y}}\hat{\boldsymbol{\omega}}_{0}^{d}\wedge\boldsymbol{\varpi}^{d}\\
 & =[\pi_{1}^{*}\omega_{t}+\frac{1}{d}\pi_{2}^{*}\hat{\rho}]^{d}\cdot[\boldsymbol{\varpi}]^{d}\\
 & \leq\text{Const}.
\end{align*}

By Claims 1 and 2, if $U$ is a neighborhood of a point $p\in\Delta$,
any weak limit $T$ of $\boldsymbol{\omega}_{\boldsymbol{\tau}}^{d}$
contains positive mass in $U\cap\Delta$. By Skoda-El Mir extension
theorem (Theorem III.2.3 of \cite{demailly2012complex}) and Corollary
III.2.14 of \cite{demailly2012complex}, 
\[
\mathbf{1}_{\Delta}T=\epsilon_{T}[\Delta].
\]
\end{proof}

\subsection{The cone condition for the limit current}

We continue our discussion. Since $\omega_{\boldsymbol{\tau}}$ satisfies
the cone condition by Lemma \ref{lem:We lemma tech } and by Proposition
\ref{prop:concentration to T}, it converges weakly to a positive
current. However, $\omega_{\boldsymbol{\tau}}\in[\omega_{0}]$ instead
of $(1-\delta)[\omega_{0}]$, and the cone condition may degenerate
when passing to the limit. Thus, to get the positive current $\Upsilon$
in Theorem \ref{thm:Mass}, we need more precise estimates.

Let $\eta>0$. We denote $\Delta_{\eta}$ to be the $\eta$-neighborhood
of $\Delta$ in $\mathcal{Y}$ with respect to $\varpi$. Define following
forms 
\begin{equation}
\omega_{\boldsymbol{\tau},\eta}'=\frac{1}{\hat{c}}\int_{\{x\}\times\hat{Y}\cap\Delta_{\eta}}\left(\Lambda_{y}\wedge\boldsymbol{\Omega}_{\boldsymbol{\tau}}\right)^{[d+1]},\label{eq:omega_tau eta}
\end{equation}
\begin{equation}
\omega_{\boldsymbol{\tau},\eta}''=\frac{1}{\hat{c}}\int_{\{x\}\times\hat{Y}\cap\Delta_{\eta}}F_{2}(V)(1+K)\hat{\omega}_{0}(x)\wedge\frac{\boldsymbol{\omega}_{y}^{d}}{d!},\label{eq:omega prime prime}
\end{equation}
and 
\begin{equation}
\omega_{\boldsymbol{\tau},\eta}=\omega_{\boldsymbol{\tau}}-\omega_{\boldsymbol{\tau},\eta}'+\omega_{\boldsymbol{\tau},\eta}''.\label{eq:om tau eta}
\end{equation}

The next 2 lemmas shows that $\omega_{\boldsymbol{\tau},\eta}$ almost
satisfies the cone condition. 
\begin{lem}
\label{lem:P_Lambda leq 1+integral of F_2 }Notations as above, we
have 
\begin{equation}
\mathcal{P}_{\hat{\Lambda}}(\omega_{\boldsymbol{\tau},\eta})\leq1+\frac{1}{\hat{c}}\int_{\{x\}\times\hat{Y}\cap\Delta_{\eta}}F_{2}(V)\frac{\boldsymbol{\omega}_{y}^{d}}{d!}.\label{eq:Plambdaleq 1}
\end{equation}
\end{lem}

\begin{proof}
From the proof of Lemma \ref{lem:We lemma tech } and Jensen's inequality,
we have 
\begin{align}
\mathcal{P}_{\hat{\Lambda}}\left(\omega_{\boldsymbol{\tau},\eta}\right) & \leq\frac{1}{\hat{c}}\int_{\{x\}\times\hat{Y}\backslash\Delta_{\eta}}F_{2}(V)\left(2-F_{2}(V)\right)\frac{\boldsymbol{\omega}_{y}^{d}}{d!}\label{eq:P_L int}\\
 & +\frac{1}{\hat{c}}\int_{\{x\}\times\hat{Y}\cap\Delta_{\eta}}F_{2}(V)\mathcal{P}_{\hat{\Lambda}}((1+K)\hat{\omega}_{0}(x))\frac{\boldsymbol{\omega}_{y}^{d}}{d!}.\nonumber 
\end{align}
Now it is straightforward to check 
\begin{equation}
\mathcal{P}_{\hat{\Lambda}}((1+K)\hat{\omega}_{0})<1,\label{eq:Pl  int}
\end{equation}
as in the proof of Lemma \ref{lem:solution t equal 1}. Therefore,
by (\ref{eq:P_L int}) and (\ref{eq:Pl  int}), 
\begin{align*}
\mathcal{P}_{\hat{\Lambda}}\left(\omega_{\boldsymbol{\tau},\eta}\right) & \leq\frac{1}{\hat{c}}\int_{\{x\}\times\hat{Y}}2F_{2}(V)\frac{\boldsymbol{\omega}_{y}^{d}}{d!}-\frac{1}{\hat{c}}\int_{\{x\}\times\hat{Y}\backslash\Delta_{\eta}}F_{2}(V)^{2}\frac{\boldsymbol{\omega}_{y}^{d}}{d!}-\frac{1}{\hat{c}}\int_{\{x\}\times\hat{Y}\cap\Delta_{\eta}}F_{2}(V)\frac{\boldsymbol{\omega}_{y}^{d}}{d!}\\
 & \leq2-\frac{1}{\hat{c}}\int_{\{x\}\times\hat{Y}}F_{2}(V)^{2}\frac{\boldsymbol{\omega}_{y}^{d}}{d!}+\frac{1}{\hat{c}}\int_{\{x\}\times\hat{Y}\cap\Delta_{\eta}}F_{2}(V)^{2}\frac{\boldsymbol{\omega}_{y}^{d}}{d!}\\
 & \leq2-\frac{1}{\hat{c}^{2}}\left(\int_{\{x\}\times\hat{Y}}F_{2}(V)\frac{\boldsymbol{\omega}_{y}^{d}}{d!}\right)^{2}+\frac{1}{\hat{c}}\int_{\{x\}\times\hat{Y}\cap\Delta_{\eta}}F_{2}(V)^{2}\frac{\boldsymbol{\omega}_{y}^{d}}{d!}\\
 & =1+\frac{1}{\hat{c}}\int_{\{x\}\times\hat{Y}\cap\Delta_{\eta}}F_{2}(V)^{2}\frac{\boldsymbol{\omega}_{y}^{d}}{d!}\\
 & \leq1+\frac{2}{\hat{c}}\int_{\{x\}\times\hat{Y}\cap\Delta_{\eta}}F_{2}(V)\frac{\boldsymbol{\omega}_{y}^{d}}{d!}.
\end{align*}
The last line is because $F_{2}(V)\leq2$ from (\ref{eq:P(ALambda)}).
We have proved the lemma. 
\end{proof}
\begin{lem}
\label{lem:Song 5.7}For any $\varepsilon>0$, $t\in(0,1)$, there
exists a $\eta_{0}=\eta_{0}(\varepsilon,t,\hat{\epsilon})>0$ s.t.
for all $s\in(0,s_{0})$ and $0<\eta<\eta_{0}$, it holds 
\[
\int_{\Delta_{\eta}}F_{2}(V)\frac{\boldsymbol{\omega}_{y}^{d}}{d!}\wedge\frac{\pi_{1}^{*}\varpi^{d}}{d!}<\varepsilon.
\]
\end{lem}

\begin{proof}
We may rewrite 
\begin{align}
\int_{\Delta_{\eta}}F_{2}(V)\frac{\boldsymbol{\omega}_{y}^{d}}{d!}\wedge\frac{\pi_{1}^{*}\varpi^{d}}{d!} & =\int_{\Delta_{\eta}}\Lambda_{y}\wedge\frac{\boldsymbol{\omega}_{\boldsymbol{\tau}}^{d-1}}{(d-1)!}\wedge\frac{\pi_{1}^{*}\varpi^{d}}{d!}.\label{eq:5.7 1}
\end{align}
To prove the claim, we argue by contradiction. If the lemma is false,
then there is an $\varepsilon>0$, $t=t_{0}$, a sequence of $s_{i}\in(0,s_{0})$
and a sequence $\eta_{i}\to0$ s.t. 
\begin{equation}
\int_{\Delta_{\eta_{i}}}\Lambda_{y}\wedge\frac{\boldsymbol{\omega}_{\boldsymbol{\tau}_{i}}^{d-1}}{(d-1)!}\wedge\frac{\pi_{1}^{*}\varpi^{d}}{d!}>\varepsilon,\label{eq:song 5.7}
\end{equation}
where $\boldsymbol{\tau}_{i}=(t_{0},s_{i})$. By weak compactness,
replacing by a subsequence, we may assume that 
\[
\boldsymbol{\omega}_{\boldsymbol{\tau}_{i}}^{d-1}\rightharpoonup T',
\]
where $T'$ is a closed positive $(d-1,d-1)$-current. By Skoda-El
Mir extension theorem $\mathbf{1}_{\Delta}T'$ is a positive closed
current with support in $\Delta$. As $\Delta$ has dimension $d$,
by the first theorem of support (\cite{demailly2012complex}, III,
Corollary 2.11), $\mathbf{1}_{\Delta}T'=0$. Thus, 
\[
\lim_{i\to\infty}\int_{\Delta_{\eta_{i}}}\Lambda_{y}\wedge\frac{\boldsymbol{\omega}_{\boldsymbol{\tau}_{i}}^{d-1}}{(d-1)!}\wedge\frac{\pi_{1}^{*}\varpi^{d}}{d!}=0,
\]
which contradicts to (\ref{eq:song 5.7}). We have finished the proof. 
\end{proof}
Next, we investigate a weak limiting current of $\omega_{\boldsymbol{\tau},\eta}$
and its regularization.

To perform a local regularization as in Definition \ref{def:Local-regularization of current},
we need to choose a finite open ball covering. We pick a finite covering
$\mathscr{P}=\{B_{j,3R}\}_{j\in\mathcal{J}}$ of $\hat{Y}$ so that
each $B_{j,3R}$ is biholomorphic to a Euclidean ball $B_{3R}(0)$
in $\C^{d}$ equipped with standard Euclidean metric $g_{j}$. Furthermore,
$B_{j,2R}\simeq B_{2R}(0)\subset\C^{d}$ is also a covering of $M$.
For a small $\epsilon_{\Lambda}>0$, we can choose a sufficiently
fine cover $\mathscr{P}$ s.t. on each $B_{j,2R}$, there are constant
coefficient positive forms $\tilde{\Lambda}_{j}$ s.t. 
\begin{equation}
\tilde{\Lambda}_{j}^{[k]}\leq\hat{\Lambda}^{[k]}\leq\tilde{\Lambda}_{j}^{[k]}+\epsilon_{\Lambda}\frac{\varpi^{k}}{k!},\label{eq:tilde Lambda}
\end{equation}
for some small $\epsilon_{\Lambda}$ to be chosen later. We may choose
$R$ even smaller such that on $B_{j,2R}$ 
\begin{align}
\frac{1}{2}\varpi & <g_{j}<2\varpi,\label{eq:assumption on varpi}
\end{align}
where $g_{j}$ is the Euclidean metric on $B_{j,3R}$.
\begin{lem}
\label{lem:song 5.8}Notations as above. There exists $r_{0}>0$,
s.t. for any $\varepsilon>0$, $t\in(0,1)$, there is a $\eta_{0}(\varepsilon,t,\hat{\epsilon})>0$
s.t. for $\eta\in(0,\eta_{0})$, $s\in(0,s_{0})$ , $r\in(0,r_{0})$,
and $j\in\mathcal{J}$, such that 
\begin{equation}
\mathcal{P}_{\tilde{\Lambda}_{j}}\left(\omega_{\boldsymbol{\tau},\eta}^{(r)}\right)\leq1+\varepsilon.\label{eq:song 5.8}
\end{equation}
\end{lem}

\begin{proof}
We pick $r_{0}<R$. At a point $x\in B_{j,2R}$, we have 
\begin{align}
\mathcal{P}_{\tilde{\Lambda}_{j}}(\omega_{\boldsymbol{\tau},\eta}^{(r)})(x) & =\mathcal{P}_{\tilde{\Lambda}_{j}}\left(\int_{z\in B_{r}(0)}r^{-2d}\vartheta\left(\frac{z}{|r|}\right)\omega_{\boldsymbol{\tau},\eta}(x+z)dV_{\C^{d}}(z)\right)\label{eq:Song5.8 1}\\
 & \leq\int_{z\in B_{r}(0)}r^{-2d}\vartheta\left(\frac{z}{|r|}\right)\mathcal{P}_{\tilde{\Lambda}_{j}}\left(\omega_{\boldsymbol{\tau},\eta}(x+z)\right)dV_{\C^{d}}(z).\nonumber 
\end{align}
Since $\tilde{\Lambda}_{j}\leq\hat{\Lambda}$, we have $\mathcal{P}_{\tilde{\Lambda}_{j}}\left(\omega_{\boldsymbol{\tau},\eta}(x+z)\right)\leq\mathcal{P}_{\hat{\Lambda}}\left(\omega_{\boldsymbol{\tau},\eta}(x+z)\right)$.
Hence by (\ref{eq:Song5.8 1}), Lemma \ref{lem:P_Lambda leq 1+integral of F_2 },
and (\ref{eq:assumption on varpi}), 
\begin{align}
\mathcal{P}_{\tilde{\Lambda}_{j}}(\omega_{\boldsymbol{\tau},\eta}^{(r)})(x) & \leq\int_{z\in B_{r}(0)}r^{-2d}\vartheta\left(\frac{z}{|r|}\right)\left(1+\frac{2}{\hat{c}}\int_{\{x+z\}\times M\cap\Delta_{\eta}}F_{2}(V)\frac{\boldsymbol{\omega}_{y}^{d}}{d!}\right)dV_{\C^{d}}(z)\label{eq:song 5.8 2}\\
 & \leq1+\frac{2}{\hat{c}}2^{2d}\int_{\Delta_{\eta}}F_{2}(V)\frac{\boldsymbol{\omega}_{y}^{d}}{d!}\wedge\frac{\pi_{1}^{*}\varpi^{d}}{d!}\nonumber \\
 & <1+\varepsilon,\nonumber 
\end{align}
where last inequality is due to Lemma \ref{lem:Song 5.7} if $\eta<\eta_{0}(\varepsilon,t,\hat{\epsilon})$. 
\end{proof}
The following Lemma shows that $\omega'_{\boldsymbol{\tau},\eta}$
is almost a $\kah$ current when $s$ is small. 
\begin{lem}
\label{lem:Song 5.9}Notations as above, there exist $r_{0}>0$ and
$\delta_{\Delta}=\frac{\epsilon_{\Delta}}{100\hat{c}}$ s.t. for any
$t\in(0,1)$, there is a $s_{1}=s_{1}(r_{0},t,\eta)>0$ s.t. for $s\in(0,s_{1})$
, $r\in(0,r_{0})$, 
\begin{equation}
\left(\omega_{\boldsymbol{\tau},\eta}'+100\delta_{\Delta}\sqrt{-1}\ddbar\phi_{S}\right)^{(r)}>20\delta_{\Delta}\varpi.\label{eq:song 5.9}
\end{equation}
\end{lem}

\begin{proof}
If (\ref{eq:song 5.9}) is false, then there are sequence $s_{i}\to0$,
and a point $x\in B_{j,R}$ s.t. 
\begin{equation}
(\omega'_{\boldsymbol{\tau}_{i},\eta})^{(r)}(x)<20\delta_{\Delta}\varpi.\label{eq:contra song 5.9}
\end{equation}
Here $\boldsymbol{\tau}_{i}=(t,s_{i})$. For any sequence $s_{j}\to0$,
after passing to a subsequence, we have 
\[
\frac{\boldsymbol{\omega}_{\boldsymbol{\tau}_{i}}^{d}}{d!}\rightharpoonup T\geq\epsilon_{\Delta}[\Delta]
\]
in weak sense by Proposition \ref{prop:concentration to T}.

Fix $j$ s.t. $x\in B_{j,R}\subset M$. Let $v\in\mathcal{T}_{x}\hat{Y}$
be any vector s.t. $\|v\|_{\varpi}=1$ and we extend it in $B_{j,2R}$
so that $v$ has constant coefficients in a local coordinates and
$1/2\leq\|v\|_{\varpi}\leq2$. Let $\gamma_{v}$ be a $(d-1,d-1)$-form
in $B_{j,2R}$ such that 
\begin{equation}
\left(\iota_{\bar{v}}\iota_{v}\xi\right)dV_{\C^{d}}=\xi\wedge\gamma_{v}.\label{eq:gamma_v}
\end{equation}
for any $(1,1)$-form $\xi$. Since $v$ has constant coefficients,
$\gamma_{v}$ is a non-negative $(d-1,d-1)$-form with constant coefficients.
At $x$, 
\begin{align}
\lim_{i\to\infty}(\omega_{\boldsymbol{\tau}_{i},\eta}')^{(r)}(v,\bar{v}) & =\lim_{i\to0}\frac{1}{\hat{c}}\int_{z\in B_{r}(0)}r^{-2d}\vartheta\left(\frac{z}{|r|}\right)\omega_{\boldsymbol{\tau}_{i},\eta}'(z+x)\wedge\gamma_{v}(z+x)\label{eq:5.9}\\
 & =\lim_{i\to0}\frac{1}{\hat{c}}\int_{B_{r}(x)\times\hat{Y}\cap\Delta_{\eta}}r^{-2d}\vartheta\left(\frac{z'-x}{|r|}\right)\Lambda_{y}(y)\wedge\frac{\boldsymbol{\omega}_{\boldsymbol{\tau}_{i}}^{d}(z',y)}{d!}\wedge\gamma_{v}(z')\nonumber \\
 & =\frac{1}{\hat{c}}\int_{B_{r}(x)\times\hat{Y}\cap\Delta_{\eta}}r^{-2d}\vartheta\left(\frac{z'-x}{|r|}\right)\Lambda_{y}(y)\wedge T(z',y)\wedge\gamma_{v}(z')\nonumber \\
 & \geq\frac{\epsilon_{\Delta}}{\hat{c}}\int_{B_{r}(x)}r^{-2d}\vartheta\left(\frac{z'-x}{|r|}\right)\hat{\rho}(z')\wedge\gamma_{v}(z')\nonumber \\
 & \geq100\delta_{\Delta}\int_{B_{r}(x)}r^{-2d}\vartheta\left(\frac{z'-x}{|r|}\right)\hat{\rho}(z')\wedge\gamma_{v}(z').\nonumber 
\end{align}
By (\ref{eq:5.9}), 
\begin{align}
\lim_{i\to\infty}(\omega_{\boldsymbol{\tau}_{i},\eta}'+100\delta_{\Delta}\sqrt{-1}\ddbar\phi_{S})^{(r)}(v,\bar{v}) & \geq100\delta_{\Delta}\int_{B_{r}(x)}r^{-2d}\vartheta\left(\frac{z-x}{|r|}\right)\varpi(z)\wedge\gamma_{v}(z)\label{eq:contra 5.9}\\
 & =100\delta_{\Delta}\int_{B_{r}(x)}r^{-2d}\vartheta\left(\frac{z-x}{|r|}\right)\|v\|_{\varpi}^{2}dV_{\C^{d}}\nonumber \\
 & \geq25\delta_{\Delta},\nonumber 
\end{align}
for $r<r_{0}<R$. Since $v$ is chosen arbitrarily at $x$, (\ref{eq:contra 5.9})
contradicts to (\ref{eq:contra song 5.9}). So we have finished the
proof. 
\end{proof}
\begin{lem}
\label{lem:Song 5.10}Notations as above, for any $\varepsilon>0$,
$t\in(0,1)$, there is a $\eta_{0}(\varepsilon,t,\hat{\epsilon})>0$
s.t. for all $s\in(0,s_{0})$, and $r\in(0,r_{0})$, 
\begin{equation}
(\omega_{\boldsymbol{\tau},\eta}'')^{(r)}<\varepsilon\varpi.\label{eq:5.10 result}
\end{equation}
\end{lem}

\begin{proof}
We argue by contradiction. If the claim is false, then there is a
point $x$, a vector $v\in T_{x}^{(1,0)}\hat{Y}$ with $\|v\|_{\varpi}=1$,
an $\varepsilon>0$, $t=t_{0}$, a sequence of $s_{i}\in(0,s_{0})$,
and a sequence $\eta_{i}\to0$ s.t. 
\begin{equation}
\iota_{\bar{v}}\iota_{v}(\omega''_{\boldsymbol{\tau}_{i},\eta_{i}})^{(r)}(x)>\varepsilon,\label{eq:contr song 5.10}
\end{equation}
where $\boldsymbol{\tau}_{i}=(t_{0},s_{i})$. By weak compactness,
replacing by a subsequence, we may assume that for $k=1,\cdots,d-1$
\begin{equation}
\boldsymbol{\omega}_{\boldsymbol{\tau}_{i}}^{k}\rightharpoonup T_{k},\label{eq:T_k current}
\end{equation}
where each $T_{k}$ is a closed positive $(k,k)$-current. By Skoda-El
Mir extension theorem $\mathbf{1}_{\Delta}T_{k}$ is a positive closed
current with support in $\Delta$ which has dimension $d$. By the
first theorem of support (\cite{demailly2012complex} III, Corollary
2.11), $\mathbf{1}_{\Delta}T_{k}=0$. Therefore, for any fixed smooth
$(d,d)$-form $\gamma$ on $\hat{Y}_{1}$ it holds that 
\begin{equation}
\lim_{i\to\infty}\int_{\Delta_{\eta_{i}}}\pi_{1}^{*}(\gamma)\wedge\left(\Lambda_{y}\wedge\frac{\boldsymbol{\omega}_{\boldsymbol{\tau}_{i}}^{d-1}}{(d-1)!}\right)=0.\label{eq:contra Song 5.10}
\end{equation}

Suppose $x\in B_{j,R}$. Let $v\in\mathcal{T}_{x}\hat{Y}$ be any
vector s.t. $\|v\|_{\varpi}=1$ and we extend $v$ in $B_{j,2R}$
so that $v$ has constant coefficients and $1/2\leq\|v\|_{\varpi}\leq2$.
Let $\gamma_{v}$ be defined in (\ref{eq:gamma_v}). At $x$, 
\begin{align}
 & (\omega_{\boldsymbol{\tau}_{i},\eta_{i}}'')^{(r)}(v,\bar{v})\label{eq:coclusion 5.10}\\
 & =\frac{1+K}{\hat{c}}\int_{B_{r}(x)}r^{-2d}\vartheta\left(\frac{z-x}{|r|}\right)\left(\int_{\{z\}\times\hat{Y}\cap\Delta_{\eta_{i}}}\Lambda_{y}\wedge\frac{\boldsymbol{\omega}_{\boldsymbol{\tau}_{i}}^{d-1}(z,y)}{(d-1)!}\right)\hat{\omega}_{0}(z)\wedge\gamma_{v}(z)\nonumber \\
 & =\frac{1+K}{\hat{c}}\int_{B_{r}(x)\times\hat{Y}\cap\Delta_{\eta_{i}}}r^{-2d}\vartheta\left(\frac{z-x}{|r|}\right)\Lambda_{y}\wedge\frac{\boldsymbol{\omega}_{\boldsymbol{\tau}_{i}}^{d-1}(z,y)}{(d-1)!}\wedge\hat{\omega}_{0}(z)\wedge\gamma_{v}(z).\nonumber 
\end{align}
Therefore, by (\ref{eq:contra Song 5.10}) and (\ref{eq:coclusion 5.10}),
\begin{equation}
\lim_{i\to\infty}(\omega_{\boldsymbol{\tau}_{i},\eta_{i}}'')^{(r)}(v,\bar{v})=0,\label{eq:conclu 5.10}
\end{equation}
which contradicts to (\ref{eq:contr song 5.10}). Thus, we have finished
the proof. 
\end{proof}
We choose $K_{1}>1$ so that on $\hat{Y}$, 
\begin{equation}
K_{1}\varpi>\hat{\omega}_{0}.\label{eq:K_2}
\end{equation}
Note that $K_{1}$ may be chosen independent of $\hat{\epsilon}$
and $t$. We have the following proposition. 
\begin{prop}
\label{prop:Song 5.1}Notations as above. For any $\varepsilon>0$,
there is a small $\epsilon_{\Lambda}$ such that for any $t\in(0,1]$,
$\hat{\epsilon}\in(0,\hat{\epsilon}_{Y})$, $r\in(0,r_{0})$, 
\[
\mathcal{P}_{\hat{\Lambda}}\left(\left(\omega_{\boldsymbol{\tau}}-\frac{\delta_{\Delta}}{K_{1}}\hat{\omega}_{0}+100\delta_{\Delta}\sqrt{-1}\ddbar\phi_{S}\right)^{(r)}\right)<1+2\varepsilon,
\]
for $s\in(0,s_{2})$ where $s_{2}=s_{2}(\varepsilon,\hat{\epsilon},t,r)$. 
\end{prop}

\begin{proof}
By Lemmas \ref{lem:Song 5.9} and \ref{lem:Song 5.10}, if $s_{2},\eta<\eta_{0}$
are small, we have 
\begin{align}
\left(\omega_{\boldsymbol{\tau}}-\frac{1}{K_{1}}\delta_{\Delta}\hat{\omega}_{0}+100\delta_{\Delta}\sqrt{-1}\ddbar\phi_{S}\right)^{(r)} & =\left(\omega_{\boldsymbol{\tau},\eta}+\omega_{\boldsymbol{\tau},\eta}'-\omega_{\boldsymbol{\tau},\eta}''-\frac{\delta_{\Delta}\hat{\omega}_{0}}{K_{1}}+100\delta_{\Delta}\sqrt{-1}\ddbar\phi_{S}\right)^{(r)}\label{eq:prop song 5.1}\\
 & \geq\left(\omega_{\boldsymbol{\tau},\eta}\right)^{(r)}+20\delta_{\Delta}\varpi-\delta_{\Delta}\varpi-\delta_{\Delta}\varpi\nonumber \\
 & \geq\left(\omega_{\boldsymbol{\tau},\eta}\right)^{(r)}+\delta_{\Delta}\varpi.\nonumber 
\end{align}
For 2 positive forms $\Lambda$ and $\Lambda'$, by (\ref{lem:P_La property}),
we have 
\[
\mathcal{P}_{\Lambda+\Lambda'}(\gamma)\leq\mathcal{P}_{\Lambda}(\gamma)+\mathcal{P}_{\Lambda'}(\gamma)
\]
for any positive $(1,1)$-form $\gamma$. Let $\Lambda'=\exp\varpi$.
Fix a point $x\in B_{j}$. Since $\hat{\Lambda}<\tilde{\Lambda}_{j}+\epsilon_{\Lambda}\Lambda'$,
we have 
\begin{align}
\mathcal{P}_{\hat{\Lambda}}\left(\left(\omega_{\boldsymbol{\tau},\eta}\right)^{(r)}+\delta_{\Delta}\varpi\right) & \leq\mathcal{P}_{\tilde{\Lambda}_{j}}\left(\left(\omega_{\boldsymbol{\tau},\eta}\right)^{(r)}+\delta_{\Delta}\varpi\right)+\mathcal{P}_{\epsilon_{\Lambda}\Lambda'}\left(\left(\omega_{\boldsymbol{\tau},\eta}\right)^{(r)}+\delta_{\Delta}\varpi\right)\label{eq:tech insong5.1}\\
 & \leq\mathcal{P}_{\tilde{\Lambda}_{j}}(\omega_{\boldsymbol{\tau},\eta}^{(r)})+\mathcal{P}_{\epsilon_{\Lambda}\Lambda'}\left(\delta_{\Delta}\varpi\right)\nonumber \\
 & <1+\varepsilon+\mathcal{P}_{\epsilon_{\Lambda}\Lambda'}\left(\delta_{\Delta}\varpi\right).\nonumber 
\end{align}
We have used Lemma \ref{lem:song 5.8} in the last line. Assume $\delta_{\Delta}<1$.
We have 
\begin{equation}
\epsilon_{\Lambda}\frac{\left(\Lambda'\wedge\exp\left(\delta_{\Delta}\varpi\right)\right)^{[d]}}{\exp\left(\delta_{\Delta}\varpi\right)^{[d]}}=\epsilon_{\Lambda}\left(1+\frac{1}{\delta_{\Delta}}\right)^{d},\label{eq:song lemma tech 2}
\end{equation}
which implies $\epsilon_{\Lambda}\mathcal{P}_{\Lambda'}(\delta_{\Delta}\varpi)<\epsilon_{\Lambda}\left(1+\frac{1}{\delta_{\Delta}}\right)^{d}.$
We choose $\epsilon_{\Lambda}$ small enough so that 
\[
\epsilon_{\Lambda}\left(1+\frac{1}{\delta_{\Delta}}\right)^{d}<\varepsilon.
\]
Then by (\ref{eq:tech insong5.1}), the claim follows. 
\end{proof}
With all the preparations, we prove Theorem \ref{thm:Mass}. 
\begin{proof}[Proof of Theorem \ref{thm:Mass}]
We fix a small $\varepsilon>0$. By Proposition \ref{prop:Song 5.1},
there is $r_{0}>0$ such that for a fixed $t\in(0,t_{0})$ and $\hat{\epsilon}\in(0,\hat{\epsilon}_{Y})$
there is a sequence of $s_{i}\to0$ s.t. for all $r\in(0,r_{0})$
\begin{equation}
\mathcal{P}_{\hat{\Lambda}}((\omega_{\mathbf{\boldsymbol{\tau}}_{i}}-\frac{\delta_{\Delta}}{K_{1}}\hat{\omega}_{0}+100\delta_{\Delta}\sqrt{-1}\ddbar\phi_{S})^{(r)})\leq1+2\varepsilon,\label{eq:P_Lahat}
\end{equation}
where $\boldsymbol{\tau}_{i}=(t,s_{i})$. Let $\tilde{\omega}_{t}$
be a weak limit a subsequence of $\omega_{\boldsymbol{\tau}_{i}}-\frac{\delta_{\Delta}}{K_{1}}\hat{\omega}_{0}+100\delta_{\Delta}\sqrt{-1}\ddbar\phi_{S}$
as $i\to\infty$. Then $\tilde{\omega}_{t}\in(1-\frac{\delta_{\Delta}}{K_{1}})[\hat{\omega}_{0}]$
and $\tilde{\omega}_{t}\geq\delta_{\Delta}\varpi$ by (\ref{eq:prop song 5.1}).
In each $B_{j,R}$, we may assume that 
\begin{equation}
\omega_{\mathbf{\boldsymbol{\tau}}_{i}}-\frac{\delta_{\Delta}}{K_{1}}\hat{\omega}_{0}+100\delta_{\Delta}\sqrt{-1}\ddbar\phi_{S}=\sqrt{-1}\ddbar\phi_{j,\boldsymbol{\tau}_{i}},\label{eq:inloc}
\end{equation}
for some local PSH function $\phi_{j,\boldsymbol{\tau}_{i}}$. After
passing to a subsequence, in each $B_{j,R}$, define $\tilde{\phi}_{j,t}=\lim_{i\to\infty}\phi_{j,\boldsymbol{\tau}_{i}}$
in $L^{1}$ and 
\begin{equation}
\tilde{\omega}_{t}=\sqrt{-1}\ddbar\tilde{\phi}_{j,t}.\label{eq:tild ome t}
\end{equation}
Then for any $r\in(0,r_{0})$, $\phi_{j,\boldsymbol{\tau}_{i}}^{(r)}$
converges to $\tilde{\phi}_{j,t}^{(r)}$ uniformly in any compact
subset in $B_{j,R}$. Therefore, by (\ref{eq:P_Lahat}), 
\begin{align}
\mathcal{P}_{\hat{\Lambda}}(\sqrt{-1}\ddbar\phi_{j,t}^{(r)}) & =\lim_{i\to\infty}\mathcal{P}_{\hat{\Lambda}}(\sqrt{-1}\ddbar\phi_{j,\boldsymbol{\tau}_{i}}^{(r)})\label{eq:smaller than 1+2ep}\\
 & \leq1+2\varepsilon.\nonumber 
\end{align}
Now we take a sequence $t_{k}\to0$ such that $\tilde{\omega}_{t_{k}}$
converges to a closed positive current $\tilde{\omega}_{0}\in(1-\frac{\delta_{\Delta}}{K_{1}})[\hat{\omega}_{0}]$.
By the same argument, we have 
\begin{equation}
\mathcal{P}_{\Lambda}(\tilde{\omega}_{0}^{(r)})\leq1+2\varepsilon.\label{eq:ag est}
\end{equation}
Let 
\begin{equation}
\Upsilon=\left(1+\frac{\delta_{\Delta}}{K_{1}}\right)\tilde{\omega}_{0}\in\left(1-\left(\frac{\delta_{\Delta}}{K_{1}}\right)^{2}\right)[\omega_{0}].\label{eq:Ups def}
\end{equation}
Assume $\delta_{\Delta}/K_{1}<1$. If 
\begin{equation}
\varepsilon<\frac{\frac{\delta_{\Delta}}{K_{1}}}{3+\frac{\delta_{\Delta}}{K_{1}}}<\frac{\delta_{\Delta}}{4K_{1}},\label{eq:est ep}
\end{equation}
then by (\ref{eq:ag est}), 
\[
\mathcal{P}_{\Lambda}(\Upsilon^{(r)})\le\frac{1+2\varepsilon}{1+\delta_{\Delta}/K_{1}}<1-\varepsilon.
\]
Since $\phi_{S}$ has positive Lelong number along $S_{\hat{Y}}$,
$\Upsilon$ has positive Lelong number along $S_{\hat{Y}}$. 
\end{proof}
\begin{rem}
We remark on various constants introduced before and their dependence.
First, $\delta_{\Delta}$ depends on $\epsilon_{\Delta}$ and $\hat{c}$.
Although Lemma \ref{lem:Song 5.9} is stated for a covering $\mathscr{P}$,
such $\delta_{\Delta}$ is uniform as long as the covering is fine
enough. Second, $K_{1}$ is independent of the choice of covering.
Thus, we may choose $\varepsilon<\frac{\delta_{\Delta}}{4K_{1}}$.
Once we have fixed $\varepsilon$, we pick the covering $\mathscr{P}$
fine enough such that $\epsilon_{\Lambda}\left(1+\frac{1}{\delta_{\Delta}}\right)^{d}<\varepsilon$,
where $\epsilon_{\Lambda}$ is used to control the variation of $\hat{\Lambda}$
in each $B_{j,2R}$. It is important to note that all above constants
may be chosen independent of $t,\hat{\epsilon}$. 
\end{rem}

\section{Completing the induction\label{sec:Completing-the-induction}}

In this section, we complete the induction argument started in the
Section \ref{sec:Set-up-for-the} and finish the proof of Theorem
\ref{thm:cone condition in a neighborhood }. We adopt notations in
the last section.

From Theorem \ref{thm:Mass}, we obtain a $\kah$ current which satisfies
the cone condition after local regularization. Together with the induction
step, we obtain a global $\kah$ current smooth away from singular
points of $Y$ by a gluing process. Near the singular points, this
$\kah$ current has positive Lelong number. Our discussion follows
J. Song's\textbf{ }modification \cite{Song2020NakaiMoishezonCF} of\textbf{
}G. Chen's argument \cite{chen2021j} based on the trick of\textbf{
}Błocki-Kołodziej \cite{blocki2007regularization}. 
\begin{thm}
\label{thm:gluing theorem}Under the same assumption in Theorem \ref{thm:cone condition in a neighborhood },
let $Y$ be a $d$-dimensional analytic subvariety of $M$ and $S_{Y}$
be the singular points of $Y$. Then there is a $\varphi_{Y}\in C^{\infty}(Y\backslash S_{Y})$
such that $\omega_{Y}=\omega_{0}+\sqrt{-1}\ddbar\varphi_{Y}$ is a
Kähler metric on $Y\backslash S_{Y}$ satisfying the cone condition
(\ref{eq:recal kapp}) on $Y\backslash S_{Y}$. Moreover, $\varphi_{Y}$
has positive Lelong number along $S_{Y}$. 
\end{thm}

We may assume that $Y$ is irreducible for simplicity. Otherwise,
we just apply the same argument to each component of $Y$. Since the
main argument is same as in J. Song's\textbf{ }work \cite{Song2020NakaiMoishezonCF},
we will only state some key lemmas and sketch the proof of Theorem
\ref{thm:gluing theorem}.

We pick a new covering $\mathscr{P}'=\{B_{j,R'}\}_{j\in\mathscr{\mathcal{J}}}$
of $\hat{Y}$ such that $R'<\frac{1}{4}r_{0}<\frac{R}{4}$, and $\{B_{j,4R'}\}$
as a cover is finer than $\mathscr{P}$. Let $\{z^{i}\}$ be the local
coordinate in $B_{j,R'}$ . We require that in each $B_{j,R'}$, 
\begin{equation}
\varpi=\sqrt{-1}\ddbar\phi_{\varpi,j},\ |\phi_{\varpi,j}-\frac{1}{4}|z|^{2}|<(R')^{2},\label{eq:-8}
\end{equation}
\begin{equation}
\omega_{0}=\sqrt{-1}\ddbar\phi_{\omega_{0},j},\ |\nabla\phi_{\omega_{0},j}|<K_{3}R',\ \phi_{\omega_{0},j}(0)=0,\label{eq:-9}
\end{equation}
\begin{equation}
\rho=\sqrt{-1}\ddbar\phi_{\rho,j},\ |\nabla\phi_{\rho,j}|<K_{3}R',\ \phi_{\rho,j}(0)=0.\label{eq:-10}
\end{equation}
Let $\Upsilon$ be the current given in Theorem \ref{thm:Mass}. Denote
\begin{equation}
\Upsilon=(1-\delta)\omega_{0}+\sqrt{-1}\ddbar\phi_{\Upsilon},\label{eq:-11}
\end{equation}
where $\phi_{\Upsilon}\in$PSH$(\hat{Y},(1-\delta)\omega_{0})$. Let
\begin{equation}
\phi_{\Upsilon,j}=(1-\delta)\phi_{\omega_{0},j}+\phi_{\Upsilon}\label{eq:-12}
\end{equation}
be the local potential of $\Upsilon$ in $B_{j,4R'}$. Let 
\begin{equation}
\varphi_{j,r}=\phi_{\Upsilon,j}^{(r)}-(1-\delta)\phi_{\omega_{0},j}-\delta\left(\delta'^{2}\phi_{\varpi,j}-\delta'\phi_{S}\right).\label{eq:-13}
\end{equation}
We choose $\delta'$ small so that $\delta'^{2}\varpi\leq\omega_{0}+\delta'\sqrt{-1}\ddbar\phi_{S}$.
Then in each $B_{j,4R'}$, by (\ref{eq:-13}) 
\begin{equation}
\omega_{0}+\sqrt{-1}\ddbar\varphi_{j,r}\geq\Upsilon^{(r)}.\label{eq:-14}
\end{equation}
Therefore from Theorem \ref{thm:Mass}, there exists $\varepsilon_{0}>0$
so that in each $B_{j,4R'}$,
\begin{equation}
\mathcal{P}_{\Lambda}(\omega_{0}+\sqrt{-1}\ddbar\varphi_{j,r})<1-\varepsilon_{0},\label{eq:-15}
\end{equation}
for $r\in(0,4R')$.

Denote $S_{\tilde{\epsilon}}:=\{p\in\hat{Y}:\nu_{\Upsilon}(p)\geq\tilde{\epsilon}\}$
for some $\tilde{\epsilon}>0$ to be chosen later. Siu's decomposition
theorem \cite{Siu1974AnalyticityOS} shows that $S_{\tilde{\epsilon}}$
is an analytic subvariety of dimension$<d$ of $\hat{Y}$ which contains
$S_{\hat{Y}}$.

By the induction hypothesis, there is an open neighborhood $U$ of
$\Phi(S_{\tilde{\epsilon}})$ in $M$ and a smooth $\kah$ metric
$\omega_{U}$ such that in $U$ 
\begin{equation}
(\exp\omega_{U}\wedge(1-\Lambda))^{[n-1]}>0.\label{eq:ind}
\end{equation}
We may assume that $(\exp\omega_{U}\wedge(1-\varepsilon_{1}-\Lambda))^{[n-1]}>0$
for some $\varepsilon_{1}<\varepsilon_{0}$. Let $\hat{U}=\Phi^{-1}(U)$
and $\omega_{\hat{U}}=\Phi^{*}\omega_{U}=\omega_{0}+\sqrt{-1}\ddbar\varphi_{\hat{U}}$
for some smooth $\varphi_{\hat{U}}$. We have 
\begin{equation}
(\exp\omega_{\hat{U}}\wedge(1-\varepsilon_{1}-\Lambda))^{[n-1]}\geq0\label{eq:more ind}
\end{equation}
on $\hat{U}$ and it is strictly positive in $\hat{U}\backslash S_{\hat{Y}}$.

The following Lemma is almost identical to\textbf{ }Lemma 6.3 of \cite{Song2020NakaiMoishezonCF},
which is a modification of Proposition 4.1 of \cite{chen2021j}. We
skip its proof.
\begin{lem}
\label{lem:using rich}Notations as above. If $\tilde{\epsilon}=\tilde{\epsilon}(\delta,R',d,n,\inf_{p\in S_{\hat{Y}}}\nu_{\Upsilon}(p))$
is chosen small, then there exists a neighborhood $\hat{V}\Subset\hat{U}$
of $S_{\tilde{\epsilon}}$ and $0<r_{1}<R'/2$ depending on $\phi_{\varpi,j},\phi_{\omega_{0},j},\varphi_{\hat{U}},$$\phi_{S}|_{\hat{Y}\backslash\hat{V}},$$\mathscr{P}',K_{3},\tilde{\epsilon}$
such that for any $r<r_{1}$ the following holds. 
\begin{enumerate}
\item If $p\in\hat{Y}\backslash\hat{V}$, then 
\[
\max_{p\in B_{j,3R'}}\nu_{\phi_{\Upsilon,j}}(p)\leq2\tilde{\epsilon},\ \max_{p\in B_{j,3R'}}\varphi_{j,r}(p)>\sup_{\hat{U}}\varphi_{\hat{U}}+3\tilde{\epsilon}\log r+1.
\]
\item If $\max_{p\in B_{j,3R'}}\nu_{\phi_{\Upsilon,j}}(p,r)\geq4\tilde{\epsilon},$
then 
\[
\max_{p\in B_{j,3R'}}\varphi_{j,r}(p)\leq\inf_{\hat{U}}\varphi_{\hat{U}}+3\tilde{\epsilon}\log r-1.
\]
\item If $\max_{p\in B_{j,3R'}}\nu_{\phi_{\Upsilon,j}}(p,r)\leq4\tilde{\epsilon}$,
then 
\[
\max_{p\in B_{i,3R'}\backslash B_{i,2R'}}\varphi_{i,r}(p)<\max_{p\in B_{j,R'}}\varphi_{j,r}(p)-2\tilde{\epsilon}.
\]
\end{enumerate}
\end{lem}

For some sufficiently small $\epsilon<\tilde{\epsilon}$, $0<r<r_{1}$,
we define
\begin{align}
\tilde{\varphi}_{\epsilon}(p) & :=\widetilde{\max}_{(\epsilon,\cdots,\epsilon)}\{\varphi_{\hat{U}}+3\tilde{\epsilon}\log r,\varphi_{j,r}:j\in\mathcal{J}\}\nonumber \\
 & =\int_{\R^{|\mathcal{J}|+1}}\max_{j\in\mathcal{J}}\{\varphi_{\hat{U}}+3\tilde{\epsilon}\log r+h_{0},\ \varphi_{j,r}+h_{j}\}\prod_{0\leq j\leq|\mathcal{J}|}\theta\left(\frac{h_{j}}{\epsilon}\right)\frac{dh_{0}\cdots dh_{|\mathcal{J}|}}{\epsilon^{|\mathcal{J}|+1}}.\label{eq:-16}
\end{align}
Then by Lemma \ref{lem:using rich} and Corollary \ref{cor:re max},
$\tilde{\varphi}_{\epsilon}\in\text{PSH}(\hat{Y},\omega_{0})\cap C^{\infty}(\hat{Y})$.
Denote $\Upsilon_{1}=\omega_{0}+\sqrt{-1}\ddbar\tilde{\varphi}_{\epsilon}$.
By (\ref{eq:-15}) and (\ref{eq:more ind}), 
\begin{equation}
\mathcal{P}_{\Lambda}(\Upsilon_{1})<1-\varepsilon_{2}\label{eq:Up1}
\end{equation}
on $\hat{Y}\backslash S_{\hat{Y}}$ for some $0<\varepsilon_{2}<\varepsilon_{1}$.
Let 
\begin{equation}
\Upsilon_{2}=(1-\varepsilon_{2})\Upsilon_{1}+\varepsilon_{2}(\omega_{0}+\delta'\sqrt{-1}\ddbar\phi_{S}).\label{eq:UP2}
\end{equation}
And let 
\[
\sqrt{-1}\ddbar\varphi_{Y}=\Upsilon_{2}-\omega_{0}.
\]
If $\varepsilon_{2}$ is sufficiently small, by (\ref{eq:Up1}), 
\begin{equation}
\mathcal{P}_{\Lambda}(\Upsilon_{2})<1-\varepsilon_{2}/2\label{eq:UP3}
\end{equation}
on $\hat{Y}\backslash S_{\hat{Y}}$. 

We are finally ready to finish proofs of several theorems stated earlier.
\begin{proof}[Proof of Theorem \ref{thm:gluing theorem}]
Since $\Phi(\hat{Y}\backslash S_{\hat{Y}})=Y\backslash S_{Y}$, we
see that $\omega_{Y}=\omega_{0}+\sqrt{-1}\ddbar\varphi_{Y}$ is smooth
on $Y\backslash S_{Y}$ and satisfies the cone condition on $Y\backslash S_{Y}$
by (\ref{eq:more ind}). Moreover, since $\phi_{S}$ has positive
Lelong number along $S_{Y}$, by (\ref{eq:more ind}) $\varphi_{Y}$
also has positive Lelong number on $S_{Y}$. 
\end{proof}
Finally, we complete the induction step and prove Theorem \ref{thm:cone condition in a neighborhood }. 
\begin{proof}[Proof of Theorem \ref{thm:cone condition in a neighborhood }]
Let $Y$ be a $d$-dimensional subvariety of $M$ and $S_{Y}$ be
the singular set of $Y$. By our induction assumption, there is an
open neighborhood $U$ of $S_{Y}$ in $M$ such that there exists
$\varphi_{U}$ in $C^{\infty}(U)\cap\text{PSH}(U,\omega_{0})$ and
$\omega_{U}=\omega_{0}+\sqrt{-1}\ddbar\varphi_{U}$ satisfying 
\[
(\exp\omega_{U}\wedge(1-\Lambda))^{[n-1]}>0.
\]
Let $\omega_{Y}$ be given in Theorem \ref{thm:gluing theorem}. We
will take neighborhoods of $S_{Y}$, $S_{Y}\subset U_{0}\Subset U_{1}\Subset U_{2}\Subset U$.
Subtracting a large number from $\varphi_{U}$, we may assume that
\begin{align}
\varphi_{Y} & >\varphi_{U}+2,\ \text{in}\ Y\backslash U_{2}.\label{eq:varphi}
\end{align}
Since $\varphi_{Y}\to-\infty$ near $S_{Y}$, we may assume 
\begin{equation}
\varphi_{Y}<\varphi_{U}-2,\ \text{in }Y\cap U_{1}.\label{eq:varphi 2}
\end{equation}
Let $W$ be a neighborhood of $Y$ and $\text{pr}_{Y}$ be the projection
form $W$ to $Y$. Using the same notation as in Lemma \ref{lem: construct a local extension},
we may choose $N>>1$ and 
\begin{equation}
\tilde{\varphi}_{Y}=\text{pr}_{Y}^{*}\varphi_{Y}+Nd_{\rho}^{2}\label{eq:-17}
\end{equation}
where $d_{\rho}(p)$ is the distance function to $Y\backslash U_{0}$
with respect to metric $\rho$. Then the similar arguments as in Lemma
\ref{lem: construct a local extension} shows that $\omega_{0}+\sqrt{-1}\ddbar\tilde{\varphi}_{Y}$
satisfies the cone condition (\ref{eq:recal kapp}) in a neighborhood
$\tilde{U}$ of $Y\backslash U_{1}$ in $M$. By (\ref{eq:varphi 2}),
shrinking $\tilde{U}$ if necessary, we may assume that in $\tilde{U}\cap U_{1}$,
\begin{align}
\tilde{\varphi}_{Y} & <\varphi_{U}-1,\label{eq:-18}
\end{align}
and in $\tilde{U}\backslash U_{2}$, by (\ref{eq:varphi}), 
\begin{equation}
\tilde{\varphi}_{Y}>\varphi_{U}+1.\label{eq:-19}
\end{equation}
Let 
\begin{equation}
\tilde{\varphi}=\widetilde{\max}_{(\frac{1}{2},\frac{1}{2})}\left\{ \tilde{\varphi}_{Y},\varphi_{U}\right\} .\label{eq:tilde varphi s}
\end{equation}
Let $U_{Y}=\tilde{U}\cup U_{1}$. Then by Corollary \ref{cor:re max},
$\tilde{\varphi}$ is in $C^{\infty}(U_{Y})\cap\text{PSH}(U_{Y},\omega_{0})$
and equals to $\varphi_{U}$ in $U_{0}$. By Corollary \ref{cor:reg max preserves subsolution},
$\omega_{0}+\sqrt{-1}\ddbar\tilde{\varphi}$ satisfies the cone condition
in $U_{Y}$. This conclude the proof of Theorem \ref{thm:cone condition in a neighborhood }. 
\end{proof}
\begin{proof}[Proof of Theorem \ref{thm:numerical criterion}]
From Theorem \ref{thm:cone condition in a neighborhood }, $([\Lambda],\kappa)$-positivity
implies the existence of a smooth subsolution on $M$. Finally, we
apply Theorem \ref{thm:Analytic theorem} to show the existence of
a smooth solution to (\ref{eq:equation with f}). 
\end{proof}

\section{An application to supercritical dHYM equations\label{sec:Application-to-Hypercritical}}

In this section, we apply Theorem \ref{thm:numerical criterion} to
the deformed Hermitian Yang-Mills equation and prove Theorem \ref{thm:application to dhym}.

Let $\rho$ be a $\kah$ metric on a compact connected $\kah$ manifold
$M$ of dimension $n$. Let $[\omega_{0}]$ be a real $(1,1)$ cohomology
class. The dHYM equation searches for a close $(1,1)$-form $\omega\in[\omega_{0}]$
such that 
\begin{equation}
\text{Re}(\omega+\sqrt{-1}\rho)^{n}=\cot\theta\text{Im}(\omega+\sqrt{-1}\rho)^{n},\label{eq:dHYM eq}
\end{equation}
where $\theta$ is the global phase defined as the argument of complex
number $\int_{M}(\omega+\sqrt{-1}\rho)^{n}$, which depends only on
$[\rho]$ and $[\omega]$. Let $\lambda_{1},\cdots,\lambda_{n}$ be
the eigenvalues of $\omega$ with respect to $\rho$, then locally
dHYM equation can be written as 
\begin{equation}
\sum_{i=1}^{n}\text{arccot}\lambda_{i}=\theta.\label{eq:local dhym}
\end{equation}
We now write the dHYM equation using notations of this paper. Let
$\Lambda_{\theta}:=\sin\theta\cos\rho-\cos\theta\sin\rho$, where
\begin{equation}
\cos\rho=\sum_{k=0}^{\left\lfloor n/2\right\rfloor }(-1)^{2k}\frac{\rho^{2k}}{(2k)!},\ \sin\rho=\sum_{k=0}^{\left\lfloor n/2\right\rfloor }(-1)^{k}\frac{\rho^{2k+1}}{(2k+1)!}.\label{eq:notation for cosrho}
\end{equation}
Then (\ref{eq:dHYM eq}) can be written as
\begin{equation}
\Lambda_{\theta}\wedge\exp\omega=0.\label{eq:another way of dhym}
\end{equation}

If $\theta\in(0,\pi)$ (resp. $(0,\frac{\pi}{2})$), (\ref{eq:dHYM eq})
is called \emph{supercritical} (resp. \emph{hypercritical}). Collins-Jacob-Yau
\cite{collins20201} proved that if there exists a solution to supercritical
dHYM equation then the following numerical condition holds: 
\begin{equation}
\int_{V}\left(\text{Re}(\omega_{0}+\sqrt{-1}\rho)^{d}-\cot\theta\text{Im}(\omega_{0}+\sqrt{-1}\rho)^{d}\right)>0\label{eq:stability}
\end{equation}
for any $d$ dimensional subvariety $V$. Collins-Jacob-Yau then conjectured
that (\ref{eq:dHYM eq}) is also sufficient for the existence of a
solution to supercritical dHYM equation.

G. Chen \cite{chen2021j} has proved Collins-Jacob-Yau conjecture
under some stronger assumption, which can be stated as follows: There
exists $\epsilon>0$ such that 
\begin{equation}
\int_{V}\left(\text{Re}(\omega_{0,t}+\sqrt{-1}\rho)^{d}-\cot\theta\text{Im}(\omega_{0,t}+\sqrt{-1}\rho)^{d}\right)>(n-d)\epsilon\int_{V}\rho^{d},\label{eq:Chen condition for dhym}
\end{equation}
for any test ray $\omega_{0,t}$. Here a test ray is $\omega_{0,t}=\omega_{0}+t\rho'$
for $t\in[0,\infty)$ and some $\kah$ form $\rho'$. Following J.
Song's \cite{Song2020NakaiMoishezonCF} modification of G. Chen's
argument, Chu-Lee-Takahashi \cite{chu2021nakai} remove the constant
$\epsilon$ but their condition requires (\ref{eq:stability}) be
positive for any test ray. In particular, when $M$ is projective,
Chu-Lee-Takahashi confirmed Collins-Jacob-Yau's conjecture. See also
Ballal \cite{Ballal10.1215/00192082-10417484} for an alternative
proof. In dimension 3, Collins-Jacob-Yau's conjecture was also proved
in Datar-Pingali \cite{datar2021numerical} for projective 3-manifolds
with hypercritical phase. In general, Collins-Jacob-Yau's conjecture
has counterexamples, which was pointed out by J. Zhang \cite{zhang2023note}.
The counterexample is explicitly constructed, and involves a blow-up
of a complex torus in dimension 3 with the global phase $\theta>\frac{\pi}{2}$.

We confirm Collins-Jacob-Yau's conjecture for $\kah$ manifolds with
global phase $\theta\in(0,\frac{\pi}{n-1}]$. Notice that in dimension
3, $\theta$ falls in the hypercritical range. Note if $\theta\in(0,\pi)$,
and (\ref{eq:dHYM eq}) has a solution $\omega$, then $\omega-\rho\cot\theta$
is a $\kah$ form. Thus, $[\omega_{0}-\rho\cot\theta]$ is a $\kah$
class. As one direction has been confirmed, we rewrite Theorem \ref{thm:application to dhym}
in the following form. 
\begin{thm}
\label{thm:If-,-thendhym}If $\theta\in(0,\frac{\pi}{n-1}]$ and $[\omega_{0}-\rho\cot\theta]$
is a Kähler class, then the equation (\ref{eq:dHYM eq}) has a smooth
solution if condition (\ref{eq:stability}) holds for any subvariety
$V$. 
\end{thm}

\begin{proof}
In order to apply Theorem \ref{thm:numerical criterion}, we rewrite
the equation and work on the $\kah$ class $[\omega_{0}-\rho\cot\theta]$.
Let 
\begin{equation}
\hat{\omega}=\omega-\rho\cot\theta\in[\omega_{0}-\rho\cot\theta].\label{eq:hat om}
\end{equation}
We may assume that $\hat{\omega}$ is $\kah$. We have 
\begin{equation}
\omega+\sqrt{-1}\rho=\hat{\omega}+\frac{1}{\sin\theta}e^{\sqrt{-1}\theta}\rho.\label{eq:omegahat rep}
\end{equation}
By (\ref{eq:omegahat rep}), 
\begin{align}
\frac{1}{n!}(\omega+\rho\sqrt{-1})^{n} & =\sum_{k=0}^{n}\frac{1}{k!}\hat{\omega}^{k}\frac{e^{\sqrt{-1}(n-k)\theta}}{(n-k)!}\left(\frac{\rho}{\sin\theta}\right)^{n-k}.\label{eq:sume omega}
\end{align}
Similarly, 
\begin{equation}
\text{Re}\frac{1}{n!}(\omega+\rho\sqrt{-1})^{n}=\sum_{k=0}^{n}\frac{1}{k!}\hat{\omega}^{k}\wedge\frac{\cos\left((n-k)\theta\right)}{(n-k)!}\left(\frac{\rho}{\sin\theta}\right)^{n-k}.\label{eq:RE ome}
\end{equation}
\begin{equation}
\text{Im}\frac{1}{n!}(\omega+\rho\sqrt{-1})^{n}=\sum_{k=0}\frac{1}{k!}\hat{\omega}^{k}\wedge\frac{\sin\left((n-k)\theta\right)}{(n-k)!}\left(\frac{\rho}{\sin\theta}\right)^{n-k}.\label{eq:Im ome}
\end{equation}
By (\ref{eq:RE ome}) and (\ref{eq:Im ome}), we have 
\begin{align}
 & \sin\theta\text{Re}\frac{1}{n!}(\omega+\rho\sqrt{-1})^{n}-\cos\theta\text{Im}\frac{1}{n!}(\omega+\rho\sqrt{-1})^{n}\label{eq:rewrite dhym 1}\\
 & =\left(\exp\hat{\omega}\wedge\left(\sum_{k=0}^{n}\left(\sin\theta\cos\left(k\theta\right)-\cos\theta\sin\left(k\theta\right)\right)\frac{\left(\frac{\rho}{\sin\theta}\right)^{k}}{k!}\right)\right)^{[n]}\nonumber \\
 & =\left(\exp\hat{\omega}\wedge\left(-\sum_{k=0}^{n}\sin\left((k-1)\theta\right)\frac{\left(\frac{\rho}{\sin\theta}\right)^{k}}{k!}\right)\right)^{[n]}.\nonumber 
\end{align}
Then (\ref{eq:dHYM eq}) is the equivalent to 
\begin{equation}
\left(\exp\hat{\omega}\wedge\left(1-\sum_{k=2}^{n}\frac{\sin\left((k-1)\theta\right)}{\sin\theta}\cdot\frac{\left(\frac{\rho}{\sin\theta}\right)^{k}}{k!}\right)\right)^{[n]}=0.\label{eq:rewrite dhym}
\end{equation}
Let 
\begin{equation}
\Lambda=\sum_{k=2}^{n}\frac{\sin\left((k-1)\theta\right)}{\sin\theta}\frac{\left(\frac{\rho}{\sin\theta}\right)^{k}}{k!}.\label{eq:final dhym}
\end{equation}
Then if $\theta\in(0,\frac{\pi}{n-1}]$, $\Lambda$ satisfies \textbf{H1
}with $k_{0}=2$.

Note that (\ref{eq:stability}) is equivalent to $[\hat{\omega}]$
being $([\Lambda],1)$-positive. Therefore, we apply Theorem \ref{thm:numerical criterion}
to equation (\ref{eq:rewrite dhym}) to establish the existence of
a solution to (\ref{eq:dHYM eq}). We have finished the proof. 
\end{proof}

\appendix

\section{\label{sec:Functional-and-uniqueness}Functional and uniqueness}

In this appendix, we introduce a global functional that is closely
related to our PDE (\ref{eq:eq with Berezin integ}). We show that
any solution to (\ref{eq:equation with f}) is a unique minimizer
to this functional.

Let 
\begin{equation}
\omega_{\varphi}=\omega_{0}+\sqrt{-1}\ddbar\varphi.\label{eq:omega_phi}
\end{equation}
Let $\phi(t)$ be a smooth path of $C^{\infty}$ functions on $M$
such that $\omega_{\phi}$ is $\kah$ and $\phi(0)=0$ and $\phi(1)=\varphi$.
Define the following functional 
\begin{equation}
\mathcal{F}_{1}(\varphi):=\int_{0}^{1}\int_{M}\dot{\phi}\Lambda\wedge\exp\omega_{\phi}dt,\label{eq:add5}
\end{equation}
where $\dot{\phi}=\frac{d}{dt}\phi(t)$. We have the following: 
\begin{prop}
\label{prop:add6}Notations as above. The path integral ( \ref{eq:add5})
depends only on $\phi(0)$ and $\phi(1)$. Therefore, $\mathcal{F}_{1}$
is well defined. 
\end{prop}

\begin{proof}
Since the space $\text{PSH}(\omega_{0},M)\cap C^{\infty}(M)$ is contractible
, it is sufficient to show that the one form $\dot{\phi}\mapsto\int_{M}\dot{\phi}\Lambda\wedge\exp\omega_{\phi}$
is closed. Let $(t,s)$ be 2 parameters and $\phi(t,s)$ be a family
of smooth functions. Therefore, it remains to show that 
\begin{equation}
\frac{\pdv}{\pdv s}\int_{M}\frac{\pdv\phi}{\pdv t}\Lambda\wedge\exp\omega_{\phi}-\frac{\pdv}{\pdv t}\int_{M}\frac{\pdv\phi}{\pdv s}\Lambda\wedge\exp\omega_{\phi}=0.\label{eq:show that}
\end{equation}
We have the following computation
\begin{align}
\frac{\pdv}{\pdv s}\int_{M}\frac{\pdv\phi}{\pdv t}\Lambda\wedge\exp\omega_{\phi} & =\int_{M}\frac{\pdv^{2}\phi}{\pdv s\pdv t}\Lambda\wedge\exp\omega_{\phi}+\int_{M}\frac{\pdv\phi}{\pdv t}\Lambda\wedge\exp\omega_{\phi}\wedge\sqrt{-1}\ddbar\frac{\pdv\phi}{\pdv s}\label{eq:complete}\\
 & =\int_{M}\frac{\pdv^{2}\phi}{\pdv s\pdv t}\Lambda\wedge\exp\omega_{\phi}+\int_{M}\frac{\pdv\phi}{\pdv s}\sqrt{-1}\ddbar\frac{\pdv\phi}{\pdv t}\wedge\Lambda\wedge\exp\omega_{\phi}\nonumber \\
 & =\frac{\pdv}{\pdv t}\int_{M}\frac{\pdv\phi}{\pdv s}\Lambda\wedge\exp\omega_{\phi}.\nonumber 
\end{align}
The proof is now complete. 
\end{proof}
According to Proposition \ref{prop:add6}, in order to evaluate $\mathcal{F}_{1}$,
we may choose a simple path $\phi(t)=t\varphi$ for $t\in[0,1]$.
Therefore, 
\begin{align}
\mathcal{F}_{1}(\varphi) & =\int_{0}^{1}\int_{M}\varphi\Lambda\wedge\exp\omega_{t\varphi}dt\label{eq:F_1 functional}\\
 & =\int_{0}^{1}\int_{M}\varphi\Lambda\wedge\exp\omega_{0}\wedge\exp\left(t\sqrt{-1}\ddbar\varphi\right)dt\nonumber \\
 & =\int_{M}\varphi\Lambda\wedge\exp\omega_{0}\wedge\left(\sum_{k=0}^{n}\frac{(\sqrt{-1}\ddbar\varphi)^{k}}{(k+1)!}\right)dt\nonumber \\
 & =\int_{M}\varphi\Lambda\wedge\exp\omega_{0}\wedge\left(\frac{\exp(\sqrt{-1}\ddbar\varphi)-1}{\sqrt{-1}\ddbar\varphi}\right)dt.\nonumber 
\end{align}
Here for any $(1,1)$-form $\alpha$, we use the following notation
\begin{equation}
\frac{\exp\alpha-1}{\alpha}:=\sum_{k=0}^{n}\frac{\alpha^{k}}{(k+1)!}.\label{eq:notation}
\end{equation}
We are now ready to give the following
\begin{defn}
We define the following functional

\begin{equation}
\mathcal{F}(\varphi)=\int_{M}\varphi\left(\Lambda-\kappa\right)\wedge\exp\omega_{0}\wedge\frac{\exp(\sqrt{-1}\ddbar\varphi)-1}{\sqrt{-1}\ddbar\varphi}.\label{eq:global functional}
\end{equation}
\end{defn}

It is straightforward to see that the critical point of $\mathcal{F}$
satisfies 
\begin{equation}
0=\delta\mathcal{F}=\int_{M}\delta\varphi(\Lambda-\kappa)\wedge\exp\omega_{\varphi},\label{eq:Euler-La}
\end{equation}
which is equivalent to a solution of equation (\ref{eq:equation with f}),
i.e. 
\begin{equation}
\left(\kappa\exp\omega_{\varphi}-\Lambda\wedge\exp\omega_{\varphi}\right)^{[n]}=0.\label{eq:equation with f-1}
\end{equation}

The global functional is a generalization of many well-known functionals.
For example (\ref{eq:global functional}) coincides with the energy
functional for complex \MA equation if $\Lambda$ is a smooth volume
form, and Aubin's functional if $\Lambda$ is a $\kah$ form.

We proceed to discuss some basic properties of $\mathcal{F}$. 
\begin{thm}
$\mathcal{F}$ is convex in the set of subsolutions of (\ref{eq:eq with Berezin integ}).
Furthermore, if $\Lambda$ satisfies condition \textbf{H2}, then a
solution to (\ref{eq:equation with f}) is the unique minimizer of
$\mathcal{F}$. 
\end{thm}

\begin{proof}
Consider the second variation of \emph{$\mathcal{F}$. }Let $\omega_{1}$
and $\omega_{1}+\sqrt{-1}\ddbar\varphi$ be two subsolutions of (
\ref{eq:eq with Berezin integ}), with $\varphi\in\text{PSH}(M,\omega_{1})$.
Let 
\[
\omega_{t\varphi}=\omega_{1}+t\sqrt{-1}\ddbar\varphi.
\]
By Lemma \ref{lem:P_La property} (\ref{enu:-is-a}), we know that
the set of subsolution forms a convex set. Thus, $\omega_{t\varphi}$
is a subsolution for $t\in[0,1]$. Notice that
\begin{align}
\frac{d^{2}}{dt^{2}}\mathcal{F}(t\varphi) & =\frac{d}{dt}\int_{M}\varphi(\Lambda-\kappa)\wedge\exp\omega_{t\varphi}\label{eq:convexity}\\
 & =\int_{M}\varphi\left(\Lambda-\kappa\right)\wedge\exp\omega_{t\varphi}\wedge\sqrt{-1}\ddbar\varphi\nonumber \\
 & =\int_{M}\left(\kappa\exp\omega_{t\varphi}-\Lambda\wedge\exp\omega_{t\varphi}\right)^{[n-1]}\wedge\sqrt{-1}\pdv\varphi\wedge\bar{\pdv}\varphi.\nonumber 
\end{align}
Since $\omega_{1}$ is a subsolution, $\left((\kappa-\Lambda)\wedge\exp\omega_{t\varphi}\right)^{[n-1]}>0$
as an $(n-1,n-1)$-form. Thus, 
\begin{equation}
\int_{M}\left(\kappa\exp\omega-\Lambda\wedge\exp\omega\right)^{[n-1]}\wedge\sqrt{-1}\pdv\varphi\wedge\bar{\pdv}\varphi\geq0,\label{eq:convexity holds}
\end{equation}
and the equality holds only if $\varphi$ is constant. Thus, $\mathcal{F}$
is convex in the set of subsolutions. Furthermore, any critical point
of $\mathcal{F}$ is a local minimum.

If $\Lambda$ satisfies \textbf{H2}, by Lemma \ref{lem:Datar-Pingali},
a solution to (\ref{eq:equation with f}) is aways a subsolution.

If $\omega_{1}$ and $\omega_{2}=\omega_{1}+\sqrt{-1}\ddbar\varphi$
are 2 solutions of (\ref{eq:equation with f}), then $\omega_{s}=\omega_{1}+s\sqrt{-1}\ddbar\varphi$
is a family of subsolutions for $s\in[0,1]$. By the previous argument,
$\mathcal{F}$ is convex along $\omega_{s}$. However, since both
$\omega_{1}$ and $\omega_{2}$ are local minimum of $\mathcal{F}$,
$\mathcal{F}$ is constant along $\omega_{s}$ which implies that
$\varphi$ is constant. Therefore, $\omega_{1}=\omega_{2}$. Hence,
$\omega_{1}$ is the unique minimizer of $\mathcal{F}$. 
\end{proof}

\section{Proof of Proposition \ref{prop: Fang-Lai-Ma1}}

In this appendix, we prove Proposition \ref{prop: Fang-Lai-Ma1},
which follows Guan \cite{Guan2014Second-order}. Our criterion (2)
in Proposition \ref{prop:crite for Sub } for subsolutions follows\textbf{
}Szèkelyhidi's definition in \cite{szekelyhidi2018fully} closely.
In general, this notion of subsolutions differs from that of Guan
in \cite{Guan2014Second-order}. Two definitions of subsolutions have
different analytic and geometric flavors. However, for positive definite
matrices, Proposition \ref{prop: Fang-Lai-Ma1} implies these two
are equivalent. This observation was first made by Fang-Lai-Ma \cite{Fang-Lai-Ma}
for inverse $\sigma_{k}$-equations.

Let $F$ be as in (\ref{eq:definition of F}). To emphasize the dependence
of $p\in M$, we denote 
\begin{equation}
F_{p}(A)=\left.\frac{(\Lambda\wedge\Omega)^{[n]}}{\Omega^{[n]}}\right|_{p}.\label{eq:-21}
\end{equation}
We pick a local coordinate in a ball neighborhood of a point $p$.
We may assume $W\subset\C^{n}$ and trivialize the bundle of hermitian
tensors over $W$ as $W\times\Gamma_{n\times n}$. We define a distance
on $W\times\Gamma_{n\times n}$ by 
\begin{equation}
d((q_{1},A_{1}),(q_{2},A_{2}))^{2}=|q_{1}-q_{2}|^{2}+|A_{1}-A_{2}|^{2}.\label{eq:-22}
\end{equation}

\begin{defn}
For $q\in U\subset W$, denote 
\begin{equation}
\Gamma_{\Lambda}^{\kappa}(q):=\{A\in\Gamma_{n\times n}^{+}:F_{q}(A)<\kappa\},\label{eq:-23}
\end{equation}
and $\Gamma_{\Lambda}^{\kappa}(U)=\{\{q\}\times\Gamma_{\Lambda}^{\kappa}(q):q\in U\}$.
The level set $\pdv\Gamma_{\Lambda}^{\kappa}(q)$ consists of all
$A$ s.t. $F_{q}(A)=\kappa$. Denote 
\begin{equation}
\mathcal{C}_{\Lambda}^{\kappa}(U):=\{\{q\}\times\mathcal{C}_{\Lambda}^{\kappa}(q):q\in U\}.\label{eq:-24}
\end{equation}
\end{defn}

Since $\mathcal{P}_{\Lambda}(A)$ is continuous in $\Lambda$ and
$A$ (Lemma \ref{lem:P_La property} (\ref{enu:-is-a}) and (\ref{enu:-is-continuous})),
$\mathcal{C}_{\Lambda}^{\kappa}(U)$ is open in $U\times\Gamma_{n\times n}$.
By Lemma \ref{lem:Datar-Pingali}, $\Gamma_{\Lambda}^{\kappa}(q)\subset\mathcal{C}_{\Lambda}^{\kappa}(q)$,
and $\Gamma_{\Lambda}^{\kappa}(U)\subset\mathcal{C}_{\Lambda}^{\kappa}(U)$.
\begin{rem}
Assuming \textbf{H2'} in section \ref{sec:Continuity-method},\textbf{
}by Lemma \ref{lem:mononton with f} and Lemma \ref{lem:conv cont path},
$F_{p}(A)$ is strictly decreasing and strictly convex in $\mathcal{C}_{\Lambda}^{\kappa}(p)$.
\end{rem}

\begin{lem}
\label{lem:Guan lem}Let $U\Subset U'\Subset W$ be open sets. Let
$E_{\Lambda}\subset\overline{\Gamma_{\Lambda}^{\kappa}(U')}\cap\mathcal{C}_{\Lambda}^{\kappa}(U)$
be a compact set. There exist constants $N>0$ and $\mu>0$ depending
on $E_{\Lambda}$ s.t. for any $(q,\underline{A})\in E_{\Lambda}$,
and $A$ satisfying $|A|>N$, $F_{q}(A)=\kappa$, we have 
\[
F_{q}^{i\bar{j}}(A)\left(A_{i\bar{j}}-\underline{A}_{i\bar{j}}\right)\geq\mu.
\]
\end{lem}

\begin{proof}
Let $\mathcal{B}_{R}(0)$ be all Hermitian matrices with norm smaller
than $R$. For $R>|\underline{A}|$, we define 
\begin{equation}
\varrho_{R}(q,\underline{A})=\sup_{A\in\pdv\Gamma_{\Lambda}^{\kappa}(q)\cap\pdv\mathcal{B}_{R}(0)}\min_{t\in[0,1]}F_{q}(t\underline{A}+(1-t)A)-\kappa.\label{eq:varrho}
\end{equation}
$\varrho_{R}\leq0$ since $F_{q}$ is convex in the set $\mathcal{C}_{\Lambda}^{\kappa}(q)$.
Since $\pdv\Gamma_{\Lambda}^{\kappa}(q)\cap\pdv\mathcal{B}_{R}(0)$
is a compact set and $F_{q}$ is continuous, $\varrho_{R}$ is a continuous
function defined on $\mathcal{C}_{\Lambda}^{\kappa}(W)$.

We claim that $\varrho_{R}$ is non-increasing with respect to $R$
for $R>|\underline{A}|$. In fact, suppose $R'>R$. Let $\kappa'=\varrho_{R'}(q,\underline{A})+\kappa$.
We may choose $A'\in\pdv\Gamma_{\Lambda}^{\kappa}(q)\cap\pdv\mathcal{B}_{R'}(0)$
and $B$ in the segment $\{t\underline{A}+(1-t)A':t\in[0,1]\}$ s.t.
\[
\kappa'=\varrho_{R'}(q,\underline{A})+\kappa=F_{q}(B).
\]
Let $\mathbf{g}(x,y)=F_{q}(x\underline{A}+yA')$ for $(x,y)\in\mathbb{R}^{+}\times\mathbb{R}^{+}$.
Then $\mathbf{g}$ is also a strictly decreasing and convex function.
By convexity and monotonicity, $\mathbf{g}(x,y)\geq\kappa$ in $\{(x,y):x+y\geq1;x,y>0\}$.
Then, $\gamma=\mathbf{g}^{-1}(\kappa)$ is a continuous convex curve
in $\R^{+}\times\R^{+}$ and satisfies 
\begin{equation}
\gamma\subset\{(x,y):x+y\leq1;x,y>0\}.\label{eq:curve}
\end{equation}
Pick $x_{0},y_{0}>0$ such that $\mathbf{g}(x_{0},y_{0})=\kappa$
and $|x_{0}\underline{A}+y_{0}A'|=R$. Denote $A=x_{0}\underline{A}+y_{0}A'$.
Since $\gamma$ is continuous, $|\underline{A}|<R$, and $|A'|>R'$,
by the mean value theorem, such $(x_{0},y_{0})$ exists. By (\ref{eq:curve}),
$x_{0}+y_{0}\leq1$. By the convexity of $\mathbf{g}$, $\{x+y=1\}$
separates $\mathbf{g}^{-1}((-\infty,\kappa'))$ and $(x_{0},y_{0})$.
Thus,
\begin{align*}
\kappa' & \leq\min_{t\in[0,1]}\mathbf{g}((1-t)x_{0}+t,(1-t)y_{0})\\
 & =\min_{t\in[0,1]}F_{q}(t\underline{A}+(1-t)A')\\
 & \leq\varrho_{R}(q,\underline{A})+\kappa
\end{align*}
Hence $\varrho_{R}$ is non-increasing in $R$.

For $\underline{A}\in\overline{\Gamma_{\Lambda}^{\kappa}}(q)$, by
the strict convexity of $F_{q}$, we see that for large $R$, $\varrho_{R}(q,\underline{A})<0$.
Since $\varrho_{R}$ is continuous and $E_{\Lambda}$ is a compact,
for large $R>R(E_{\Lambda})$, $\varrho_{R}(q,\underline{A})<-\mu<0$
in $E_{\Lambda}$ for some $\mu>0$. Then, we have for any $|A|>R$
and $F_{q}(A)=\kappa$, there exists $A'=t'\underline{A}+(1-t')A,t'\in(0,1]$,
s.t. 
\begin{align*}
-\mu & >\varrho_{R}(q,\underline{A})>\varrho_{|A|}(q,\underline{A})\\
 & =F_{q}(A')-F_{q}(A)\\
 & \geq F_{q}^{i\bar{j}}(A)t'\left(\underline{A}_{i\bar{j}}-A_{i\bar{j}}\right).
\end{align*}
Thus, $F_{q}^{i\bar{j}}(A)(A_{i\bar{j}}-\underline{A}_{i\bar{j}})\geq\mu>0$. 
\end{proof}
\begin{lem}
\label{lem:claim}Let $U\Subset U'\Subset W$ be open sets. Let $E_{\Lambda}\subset\mathcal{C}_{\Lambda}^{\kappa}(U)\backslash\overline{\Gamma_{\Lambda}^{\kappa}(U')}$
be a compact subset. There exists a constant $N(E_{\Lambda})>0$ depending
on $E_{\Lambda}$ s.t. for any $\{q\}\times\underline{A}\in E_{\Lambda}$,
and $A$ satisfying $|A|>N$, $F_{q}(A)=\kappa$, we have 
\[
F_{q}^{i\bar{j}}(A)(A_{i\bar{j}}-\underline{A}_{i\bar{j}})\geq0.
\]
\end{lem}

\begin{proof}
We argue by contradiction. If the claim is false, there exist 
\begin{enumerate}
\item a sequence $(q_{i},\underline{A}_{i})\in E_{\Lambda}$ , 
\item a sequence of real numbers $t_{i}\in(0,\infty)$ and $t_{i}\to\infty$, 
\item a sequence of positive Hermitian matrices $A_{i}=\underline{A}_{i}+t_{i}B_{i}$
with $|B_{i}|=1$ and $F_{q_{i}}(A_{i})=\kappa$, 
\end{enumerate}
such that 
\begin{equation}
F_{q_{i}}^{k\bar{j}}(A_{i})(\left(A_{i}\right)_{k\bar{j}}-\left(\underline{A_{i}}\right)_{k\bar{j}})<0.\label{eq:add1}
\end{equation}
Geometrically, (\ref{eq:add1}) implies that the tangent plane of
level set $\pdv\Gamma_{\Lambda}^{\kappa}(q_{i})$ at $A_{i}$ separates
the level set $\pdv\Gamma_{\Lambda}^{\kappa}(q_{i})$ and $\underline{A}_{i}$.
Therefore, for any $t\in(0,t_{i})$, 
\begin{equation}
F_{q_{i}}(\underline{A}_{i}+tB_{i})>\kappa.\label{eq: no inter to level set-1}
\end{equation}
Since the unit sphere in the space of Hermitian matrices is compact,
we may find subsequences such that

\[
(q_{i},\underline{A}_{i})\to(q_{\infty},\underline{A}_{\infty})\in E_{\Lambda},\ B_{i}\to B_{\infty}\in\mathcal{\pdv}\mathcal{B}_{1}(0).
\]
Since $F_{q}$ is continuous in $q$ and upper semicontinuous in $A$
(by convexity), $F_{q_{\infty}}(\underline{A}_{\infty}+tB_{\infty})\geq\kappa$
for all $t\in(0,\infty)$. From the strict convexity of $F_{q}$,
$F_{q_{\infty}}(\underline{A}_{\infty}+tB_{\infty})>\kappa$ for all
$t\in(0,\infty)$. Also, as $\overline{\Gamma_{n\times n}^{+}}$ is
closed, $\underline{A}_{\infty}+tB_{\infty}\in\overline{\Gamma_{n\times n}^{+}}$
for all $t\in[0,\infty)$. We have either of the following two cases.

Case 1. $B_{\infty}\not\in\overline{\Gamma_{n\times n}^{+}}$. Then,
$B_{\infty}$ has a negative eigenvalue. For large enough $t$, $\underline{A}_{\infty}+tB_{\infty}\not\in\overline{\Gamma_{n\times n}^{+}}$.
Here we reach a contradiction.

Case 2. $B_{\infty}\in\overline{\Gamma_{n\times n}^{+}}$. Then, by
the definition of $\underline{A}_{\infty}\in\mathcal{C}_{\Lambda}^{\kappa}(q_{\infty})$,
\[
\lim_{t\to\infty}F_{q_{\infty}}(\underline{A}_{\infty}+tB_{\infty})<\kappa.
\]
It is clearly a contradiction to (\ref{eq: no inter to level set-1}).

Therefore, we have proved the lemma.
\end{proof}
\begin{lem}
\label{lem:fanglaima}Let $U\Subset U'\Subset W$ be open sets. Let
$E_{\Lambda}\subset\mathcal{C}_{\Lambda}^{\kappa}(U)$ be a compact
set. There exist constants $N>0$ and $\mu>0$ depending on $E_{\Lambda}$
s.t. for any $(q,\underline{A})\in E_{\Lambda}$, and $A$ satisfying
$|A|>N$, $F_{q}(A)=\kappa$, it holds 
\begin{equation}
F_{q}^{i\bar{j}}(A)\left(A_{i\bar{j}}-\underline{A}_{i\bar{j}}\right)\geq\mu\left(1-\sum_{i}F_{q}^{i\bar{i}}(A)\right).\label{eq:-28}
\end{equation}
\end{lem}

\begin{proof}
Since $E_{\Lambda}$ is compact, we pick $\delta<\text{dist}(E_{\Lambda},(\mathcal{C}_{\Lambda}^{\kappa}(U))^{c})/(4n)$.
Let 
\begin{equation}
E_{\Lambda}^{\delta}:=\{(q,B-\delta\text{Id}):(q,B)\in E_{\Lambda}\}\subset\mathcal{C}_{\Lambda}^{\kappa}(U).\label{eq:-25}
\end{equation}
For $(q,\underline{A})\in E_{\Lambda}$, let $\underline{A}^{\delta}=\underline{A}-\delta\text{Id}$.
We discuss the following two cases:

Case 1. $(q,\underline{A}^{\delta})\in E_{\Lambda}^{\delta}\cap\overline{\Gamma_{\Lambda}^{\kappa}(U')}$.
If $F_{q}(A)=\kappa$, the line $l=\{t\underline{A}^{\delta}+(1-t)A:t\in[0,1]\}$
lies entirely in $\overline{\Gamma_{\Lambda}^{\kappa}(q)}$. From
Lemma \ref{lem:Guan lem}, if $R>R(E_{\Lambda}^{\delta})$, $|A|>R$,
$F_{q}(A)=\kappa$, then for some $\mu_{0}>0$ depending on $E_{\Lambda}^{\delta}$
\begin{equation}
F_{q}^{i\bar{j}}(A)(A_{i\bar{j}}-\underline{A}_{i\bar{j}}^{\delta})\geq\mu_{0}>0.\label{eq:-26}
\end{equation}
(\ref{eq:-26}) then implies 
\begin{equation}
F_{q}^{i\bar{j}}(A)\left(A_{i\bar{j}}-\underline{A}_{i\bar{j}}\right)\geq\mu_{0}-\delta\sum_{i=1}^{n}F^{i\bar{i}}(A).\label{eq:-27}
\end{equation}
Then we obtain (\ref{eq:-28}) by taking $\mu=\min\{\mu_{0},\delta\}$
in (\ref{eq:-27}). 

Case 2. $(q,\underline{A}^{\delta})\in E_{\Lambda}^{\delta}\backslash\overline{\Gamma_{\Lambda}^{\kappa}(U')}$.
We denote $G_{\Lambda}^{\delta}:=\{(q,B-\delta\text{Id}):(q,B)\in\overline{E_{\Lambda}^{\delta}\backslash\overline{\Gamma_{\Lambda}^{\kappa}(U')}}\}$.
Then $G_{\Lambda}^{\delta}$ is a compact subset in $\mathcal{C}_{\Lambda}^{\kappa}(U)\backslash\overline{\Gamma_{\Lambda}^{\kappa}(U')}$
and $(q,\underline{A}^{2\delta})\in G_{\Lambda}^{\delta}$. From Lemma
\ref{lem:claim}, for large enough $R>R(G_{\Lambda}^{\delta})$, $|A|>R$,
$F_{q}(A)=\kappa$, the line $l=\{t\underline{A}^{2\delta}+(1-t)A:t\in[0,1]\}$
intersect $\Gamma_{\Lambda}^{\kappa}(q)$ at 
\begin{equation}
A'=t_{0}\underline{A}^{2\delta}+(1-t_{0})A\label{eq:-32}
\end{equation}
for some $t_{0}\in(0,1)$. Since $F_{q}(\underline{A}^{2\delta})>\kappa$,
the tangent plane of $\pdv\Gamma_{\Lambda}^{\kappa}(q)$ at $A'$
must separate $\underline{A}^{2\delta}$ and $\Gamma_{\Lambda}^{\kappa}(q)$
which indicates
\begin{equation}
F_{q}^{i\bar{j}}(A')(A'_{i\bar{j}}-\underline{A}_{i\bar{j}}^{2\delta})<0.\label{eq:add2}
\end{equation}
From Lemma \ref{lem:claim}, $|A'|\leq N(G_{\Lambda}^{\delta})$.
Therefore, all such pairs $(q,A')$ lie in a compact subset $\tilde{E}{}_{\Lambda}$
of $\pdv\Gamma_{\Lambda}^{\kappa}(U')\cap\mathcal{C}_{\Lambda}^{\kappa}(U)$,
which depends on $G_{\Lambda}^{\delta}$. We apply Lemma \ref{lem:Guan lem}
to $\tilde{E}_{\Lambda}$ to get 
\begin{equation}
F_{q}^{i\bar{j}}(A)(A_{i\bar{j}}-A'_{i\bar{j}})>\mu(\tilde{E}_{\Lambda}),\label{eq:-31}
\end{equation}
if $|A|>R(\tilde{E}_{\Lambda})$. From (\ref{eq:-31}),
\begin{align}
F_{q}^{i\bar{j}}\left(A\right)\left(A_{i\bar{j}}-\underline{A}_{i\bar{j}}^{2\delta}\right) & =t_{0}^{-1}F_{q}^{i\bar{j}}(A)(A_{i\bar{j}}-A'_{i\bar{j}})\label{eq:-29}\\
 & >\mu(\tilde{E}_{\Lambda}).\nonumber 
\end{align}
Thus, 
\begin{equation}
F_{q}^{i\bar{j}}(A)\left(A_{i\bar{j}}-\underline{A}_{i\bar{j}}\right)\geq\mu(\tilde{E}_{\Lambda})-2\delta\sum_{i=1}^{n}F_{q}^{i\bar{i}}(A).\label{eq:-30}
\end{equation}
Again, we may choose $\mu=\min\{\mu(\tilde{E}_{\Lambda}),2\delta\}$.
We have finished the proof. 
\end{proof}
Finally, we prove Proposition \ref{prop: Fang-Lai-Ma1}. 
\begin{proof}[Proof of \ref{prop: Fang-Lai-Ma1}]
We pick a finite coordinate ball covering $\mathscr{P}=\{B_{j,4R}(q_{j})\}_{j=1}^{l}$
such that $\{B_{j,R}(q_{j})\}_{j=1}^{l}$ also covers $M$. Denote
\[
\omega_{\text{sub}}(p)=\frac{\sqrt{-1}}{2}\left(\underline{A}_{j}(p)\right)_{i\bar{k}}dz^{i}\wedge d\bar{z}^{k},
\]
in $B_{j,4R}(q_{j})$. Then, $\{(p,\underline{A}_{j}(p)):p\in\overline{B_{j,R}(q_{j})}\}$
forms a compact subset in $\mathcal{C}_{\Lambda}^{\kappa}(B_{j,2R}(q_{j}))$.
Proposition \ref{prop: Fang-Lai-Ma1} then follows by applying Lemma
\ref{lem:fanglaima} to each $\{(p,\underline{A}_{j}(p)):p\in\overline{B_{j,R}(q_{j})}\}$. 
\end{proof}

\section{List of notations}
\begin{itemize}
\item $\Omega=\exp\omega$, $P=\exp\rho$. 
\item $\kappa:$ positive constant such that $\int_{M}(\kappa-\Lambda)\wedge\Omega=0$. 
\item $m$: positive constant in the definition of uniform positivity. 
\item $k_{0}$ : degree in \textbf{H1}. 
\item $\mathring{\Lambda}$: $\Lambda$ minus $(n,n)$-component. 
\item $\mathcal{C}_{\Lambda}^{\kappa}$ :set of $\kah$ forms satisfying
the cone condition (\ref{eq:cone condition}) on $M$. 
\item $\Gamma_{n\times n},\Gamma_{n\times n}^{+},\overline{\Gamma_{n\times n}^{+}}$
: the set of Hermitian matrices, positive definite Hermitian matrices,
and non-negative Hermitian matrices, respectively. 
\item $\mathcal{O}$: a labeled orthogonal splitting structure on $M$.
Definition \ref{def:A-labeled-orthogonal}. 
\item $\mathcal{T}M$: holomorphic tangent space. 
\item $\sigma_{k}(A)$: $k$-th symmetric function of eigenvalues of a Hermitian
matrix$A$. 
\item $T_{k-1}(A)$: linearized operator of $\sigma_{k}(A)$. 
\item $F(A),F_{k}(A)$: local functionals defined in (\ref{eq:definition of F})
and (\ref{eq:F_k}). 
\item $F_{\Lambda}(A:B)$: defined in (\ref{eq:F(A:B)}). 
\item $\mathcal{P}_{\Lambda}(A)$: defined in (\ref{eq:P(A)}). 
\item $\chi$: defined in Definition \ref{def:We-define-a}. 
\item \emph{$\mathcal{F}(\varphi):$ }the global functional defined in (\ref{eq:global functional}). 
\item $\widetilde{\max}_{\eta}(\cdot,\cdots,\cdot):$ regularized maximum
with parameter $\eta$. 
\item $\mathcal{M}=M_{1}\times M_{2}$: the product manifold of $M_{1}=M_{2}=M$. 
\item $\Phi:M'\to M,$ desingularization map. 
\item $Y\subset M$, $\hat{Y}\subset M'$ are subvarieties and $\hat{Y}$
is the strict transform of $Y$. 
\item $\mathcal{Y}=\hat{Y}_{1}\times\hat{Y}_{2}$ where $\hat{Y}_{1}\simeq\hat{Y}_{2}\simeq\hat{Y}$. 
\item $\Delta=\{(y,y):y\in\hat{Y}\}$: the diagonal in $\mathcal{Y}$. 
\item $\hat{\epsilon}$: perturbation parameter. 
\item $K,K_{1}:$ constants for comparing $\varpi$, $\omega_{0}$, $\hat{\omega}_{0}$
in (\ref{eq:choice of K}) and (\ref{eq:K_2}). 
\item $c_{t,\hat{\epsilon}}$, (\ref{eq:eq for omega_t}); $c_{t,\hat{\epsilon},\delta_{\rho}}$,
(\ref{eq:f_s,t,tau}); $\hat{c}$ , (\ref{eq:hat c}). 
\item $\epsilon_{\Lambda}$: the parameter that controls the covering. 
\item $S$: the exceptional locus of $\Phi$. 
\item $\varpi$: a metric on $M'$. 
\item $\hat{\rho}$, $\hat{P}$: $\hat{\rho}$ is the perturbed reference
metric on $\hat{Y}$, (\ref{eq:perturbed rho}), $\hat{P}=\exp\hat{\rho}$. 
\item $\Lambda_{x},\Lambda_{y}$: lifted $\Lambda$ and $\frac{1}{d}\hat{\rho}$
on $\hat{Y}_{1}$ and $\hat{Y}_{2}$, (\ref{eq:Define La_xy}). 
\item $\boldsymbol{\Lambda},\boldsymbol{\rho},\boldsymbol{\varpi}:$ lifted
forms on $\mathcal{Y}$, (\ref{eq:lifted forms}). 
\item $\boldsymbol{\omega}_{x},\boldsymbol{\omega}_{y},\boldsymbol{\omega}_{m}$:
(\ref{eq:omega _x}), (\ref{eq:omega_y}), (\ref{eq:omega_mix}). 
\item $\mathbf{F},F_{1},F_{2},\mathcal{P}_{\mathbf{\Lambda}},\mathcal{P}_{1},\mathcal{P}_{2}$
: (\ref{eq:Bold F A})-(\ref{eq:P2}). 
\item $\mathbf{I}_{\hat{\epsilon}}$: the solution interval for equation
(\ref{eq:equation on hat Y}). 
\item $\Upsilon^{(r)}$: local regularization of a $(1,1)$-current $\Upsilon$
with scale $r$. 
\item $\nu_{\phi}(p),\nu_{\Upsilon}(p)$: Lelong number of $\phi,\Upsilon$
at $p$. 
\item $\delta_{\Delta}$: constant that controls the mass concentration. 
\item $\epsilon_{\Delta}$: mass concentration on $\Delta$. 
\end{itemize}
\bibliographystyle{plain}
\bibliography{General_inverse_sigma}

\end{document}